\documentclass[11pt]{article}
\usepackage{amscd,amsmath,amsfonts,amssymb,amsthm,cite,mathrsfs,color} 
\usepackage{dsfont} \usepackage{leftidx}
\usepackage{tikz}
\usetikzlibrary{arrows.meta}

\usepackage{enumerate}
\usepackage{appendix}
\usepackage{graphicx}
\usepackage{hyperref}
\usepackage{diagbox}
\usepackage{comment}
\usepackage{nomencl}
\usepackage[left=2.cm, right=2.cm, top=2.5cm, bottom=2.5cm]{geometry}
\usepackage{graphicx} \usepackage{xcolor}
\usepackage{setspace}

\usepackage{tabularx} 
\usepackage{booktabs} 

\numberwithin{equation}{section}

\theoremstyle{plain}
\newtheorem{theorem}{Theorem}[section]
\newtheorem{lemma}[theorem]{Lemma}
\newtheorem{corollary}[theorem]{Corollary}
\newtheorem{proposition}[theorem]{Proposition}

\theoremstyle{definition}
\newtheorem{definition} 
[theorem]{Definition} 
 
\newtheorem{example}[theorem]{Example}

\newtheorem{remark}{Remark}[section]

\newcommand{\tr}{\mathrm{Tr}}

\title{ { \bf Outliers for deformed inhomogeneous random matrices}}

\author{
Ruohan Geng\thanks{School of Mathematical Sciences, University of Science and Technology of China, Hefei, 230026, P.R.~China. E-mail: \texttt{ruohangeng@mail.ustc.edu.cn}}
\and
Dang-Zheng Liu\thanks{School of Mathematical Sciences, University of Science and Technology of China, Hefei, 230026, P.R.~China. E-mail: \texttt{dzliu@ustc.edu.cn}}
\and
Guangyi Zou\thanks{School of Mathematical Sciences, University of Science and Technology of China, Hefei, 230026, P.R.~China. E-mail: \texttt{zouguangyi2001@gmail.com}}
}

\begin{document}
 
\maketitle

\begin{abstract}
Inhomogeneous random matrices with non-trivial variance profiles determined by symmetric stochastic matrices and with independent sub-Gaussian entries up to Hermitian symmetry,  encompass 
 a wide range of important models, including sparse Wigner matrices and random band matrices. In these models, the maximum entry variance—a natural proxy for sparsity—serves both as a key structural feature and  a primary analytical obstacle. In this paper, we consider low-rank additive perturbations of such matrices and establish a sharp BBP phase transition for extreme eigenvalues at the level of the law of large numbers. Furthermore, in the Gaussian setting, we derive the fluctuations of spectral outliers under suitable conditions on the variance profile and perturbation. These fluctuations exhibit strong non-universality, depending on the eigenvectors, sparsity levels, and the underlying geometric structure. Our proof strategies  rely on ribbon graph expansions, upper bounds for diagram functions, large-moment estimates, and the enumeration of typical diagrams.

\end{abstract}

\tableofcontents

\section{Introduction}

\subsection{Deformed Wigner matrices}
In his  1996 Selected Papers of Freeman Dyson with Commentary \cite{dyson1996selected}, Dyson remarked   on the seminal paper \cite{dyson1962brownian} that ‘‘\textit{The Brownian-motion ensemble has every element of a matrix independently  executing a Brownian motion. The physical motivation for introducing it is that it represents a system whose Hamiltonian is a sum of two parts, one known and one unknown.}"  
Deformed models, consisting of sums of GOE (or GUE) and a deterministic matrix, not only  play  a key role   in     the resolution of the   \textit{Universality Conjecture} in Random Matrix Theory (RMT)\cite{erdHos2010bulk, johansson2001universality,tao2011random},
 but also   exhibit outlier and phase transition phenomena. These features highlight the richness of eigenvalue statistics in RMT and have attracted significant interest due to their broad applications in statistics, mathematical physics, random graphs, and signal processing. For a comprehensive overview, see the excellent survey \cite{peche2014deformed}. 
 
 The outlier   in RMT    was first observed numerically for  the rank-one perturbation of the  normalized GOE
 \begin{equation} X_{\mathrm{GOE},N}=H_{\mathrm{GOE}_N}+\frac{a}{N}{ {\bf 1}_{N}}  {\bf 1}_{N}^{t}
 \end{equation}
 in the early 1960s \cite{lang1964isolated}, where  ${\bf 1}_{N}$ denotes
  a  column vector with all entries equal to 1. 
According to Weyl's eigenvalue interlacing inequalities,   the empirical  spectral measure  of   $X_{\mathrm{GOE},N}$ still  
converges weakly, almost surely, 
 to the famous semicircular law    with density
 \begin{equation}
 \rho_{sc}(x)=\frac{1}{2\pi}\sqrt{4-x^2},\quad -2\leq x\leq 2.
 \end{equation}
  However,  when $a>1$,  such a deformation may create   a single outlier eigenvalue separating from the bulk spectrum  located at the value
\begin{equation} \label{rhoa}
    \rho_{a}:=a+\frac{1}{a}.
\end{equation}

This  exact formula    was first  derived   by Jones, Kosterlitz and Thouless  \cite{jones1978eigenvalue}, and  was later    independently rediscovered   by     F{\"u}redi   and  Koml\'{o}s   \cite{furedi1981eigenvalues} in a
more general setting motivated  by  applications in  random graphs.  Furthermore,       Gaussian  fluctuations of  these outlier eigenvalues
 were subsequently      established     for certain    non-centered Wigner random matrices in   \cite{furedi1981eigenvalues}. For a detailed review  on integrable structures hidden in rank-one perturbed random matrices, we refer to \cite{forrester2023rank}.  

   The appearance   of  an outlier  is associated with a sharp transition, where the eigenvalue  size of  the perturbed  matrix exceeds a   certain threshold $a_c=1.$    This transition was first established for spiked complex  Wishart ensembles with arbitrary finite-rank perturbations   and is now well-known as   BBP phase transition after the seminal work  of  Baik,  Ben Arous and P{\'e}ch{\'e} \cite{baik2005phase}. Similar results have been obtained for deformed GOE, GUE, and spiked real Wishart ensembles in \cite{bloemendal2013limits,bloemendal2016limits,peche2006largest}. For simplicity, we take  a rank-$q$ deformation of the normalized  GUE as an example 
   \begin{equation}
X_{\mathrm{GUE},N}=H_{\mathrm{GUE}_N}+ 
       \mathrm{diag}(
       \overbrace{ a,\ldots, a}^{q}
       , 0,\ldots,0), \quad a>0. 
      \end{equation}  
The first $q$ largest eigenvalues  $\lambda_1 \geq \cdots \geq \lambda_q $ of  $X_{\mathrm{GUE},N}$   exhibit  the following phase transitions in the large $N$ limit, both at the levels of the law of large numbers  (LLN)   and of fluctuations 
 as established  in \cite{bai2008central,baik2005phase,baik2006eigenvalues,benaych2011eigenvalues,capitaine2009largest,capitaine2012central,
peche2006largest} and \cite{baik2005phase,benaych2011fluctuations,capitaine2009largest,knowles2013isotropic,knowles2014outliers,pizzo2013finite}. 
 \begin{itemize}
   
   {\bf \item[ (I)] At the level of LLN},  for all $1\leq j\leq q $,     $\lambda_{j}\to  2  \ a.s.$    when  $a\leq 1$  while $\lambda_{j} \to  \rho_{a} \  a.s.$ when    $a>1$.
{\bf \item[(II)]  At the level of  fluctuations},  
 {\bf (Subcritical regime)} when   $0<a<1$, the vector  $N^{2/3}\big(\lambda_1-2,  \ldots,  \lambda_q-2\big )$ converges weakly to the Tracy-Widom law of
 the   first $q$ largest eigenvalues  of the GUE matrix;  
 {\bf (Critical regime)} when $a=1$,    $N^{2/3}\big(\lambda_1-2,  \ldots,  \lambda_q-2\big )$ converges weakly to the deformed Tracy-Widom law  depending on $q$;  {\bf (Supercritical regime)} when   $a>1$, for some  $C_{a}>0$,    $C_{a}N^{1/2}\big(\lambda_1-\rho_a,  \ldots,  \lambda_q-\rho_a\big )$ converges weakly to the eigenvalues of the   GUE$_{q}$  ensemble.
    \end{itemize}
    
Soon after the  seminal work of Baik, Ben Arous and Péché    \cite{baik2005phase},  much progress      has been made in understanding the behavior of outliers and BBP transition for additive and multiplicative deformations of Wigner and Wishart ensembles, with finite, large or full rank perturbations;  see, for instance,  \cite{bai2008central,baik2006eigenvalues,benaych2011eigenvalues,benaych2011fluctuations,capitaine2009largest,capitaine2012central, capitaine2014exact,  capitaine2009largest,ding2022edge,knowles2013isotropic,knowles2014outliers,loubaton2011almost,
peche2006largest, pizzo2013finite}.   We refer to   \cite{peche2014deformed,knowles2014outliers} for a  comprehensive   review on deformed random   matrices. Interestingly, unlike the universally limiting fluctuations in the subcritical and critical regimes, the fluctuations of outliers are not universal and may depend on both  the structure of the centered Wigner matrices and the geometry of eigenvectors of the perturbation matrix\cite{capitaine2009largest,capitaine2012central,knowles2013isotropic,knowles2014outliers,pizzo2013finite}.   At this point, it is worth mentioning  that all the deformed ensembles  mentioned above  are mean-field models, meaning that  the matrix entries   are either nearly i.i.d. random variables or have variances of comparable magnitude.

\subsection{Inhomogeneous random matrices}
 Inhomogeneous (or structured) random matrices ({IRM} for short)  $H_N$,   usually  referred  to  as  random  matrices with  a non-trivial variance profile $\sigma_{ij}^2$,  
include many    prominent   examples,  such as   
  {  Wigner matrices,     sparse 
  Wigner and Wishart   matrices,
  random band matrices in general dimension $d$, 
     diluted Wigner matrices  via   $d$-regular graphs}, etc.  
Compared to mean-field Wigner matrices,   two  new key features    highlighted in the IRM ensembles  
  are geometric structure and  sparsity.   
As to the former,  local  spectral properties of  random band matrices   have recently  been proved to 
  do depend on the spatial dimension $d$ \cite{liu2023edge,sodin2010spectral,spencer2011random,xu2022bulk,yang2021delocalization,yang2022delocalization}.  For the latter,      the maximum of the standard deviations of all entries  for  the IRM   matrix defined by  
\begin{equation}
\sigma^{*}_N=\max_{1 \le i,j \le N} \sigma_{ij},
\end{equation} 
serves as a natural measure of sparsity in (random) matrix theory.  Matrix sparsity is not only a key feature but also presents significant challenges across various disciplines, including random matrices,  statistical inference,  numerical linear algebra,  compressed sensing  
  and graph theory.   So these   have  been an active  subject of much recent interest   \cite{adamczak2024norms,altschuler2024spectral,Au23BBP,bandeira2016sharp,brailovskaya2024extremal,dai2024deviation,latala2018dimension,shou2024localization}.  
We refer to  \cite{brailovskaya2024extremal,van2017structured} for a detailed discussion  on  inhomogeneous  random matrices.

Since establishing local spectral statistics of random band matrices remains a challenging problem (see, e.g., \cite{bourgade2018random, liu2023edge, yang2021delocalization, xu2022bulk} for recent progress and surveys), a major challenge in investigating inhomogeneous random matrices lies in understanding how the given structure of the matrix is reflected in its spectral properties. Significant progress has been made recently on non-asymptotic matrix concentration inequalities \cite{adamczak2024norms, bandeira2016sharp, bandeira2023matrix, latala2018dimension} and on spectral outliers \cite{altschuler2024spectral}.
Closely related to the present  model   is    an essentially complete characterization     that  the   spectral measure   
 \begin{equation}
     \mu_{N}(H_N):=\frac{1}{N}\sum_{k=1}^N \delta_{\lambda_{k}(H_N)}
 \end{equation} 
  converges weakly, almost
surely, and in expectation to the semicircular law, provided that $\sigma^{*}_N \to 0$; see  \cite{gotze2015limit} and \cite{altschuler2024spectral} for alternative derivations.   
A recent sharp result for the IRM ensemble with i.i.d. sub-Gaussian entries~\cite[Theorems 1.2 \& 1.3]{altschuler2024spectral} provides a precise characterization of the emergence of spectral outliers in terms of the sparsity proxy:
\begin{equation}\label{spectralnorm} {  If \ \sigma^{*}_N \sqrt{\log N}  \to 0, \  then\ \ the\  spectral\  norm\  \|H_N\|_{\rm op} \to 2 \ a.s.} \end{equation}  
Otherwise, $H_N$ may exhibit outliers almost surely. This demonstrates a “structural” universality phenomenon: the presence of outliers depends solely on the sparsity level, independent of the specific structure of variance profiles. For related results on the absence of outliers in random band matrices under certain bandwidth conditions, see \cite{benaych2014largest,EK11Quantum,khorunzhiy2008estimates,liu2023edge,sodin2010spectral}. So this leads to two fundamental questions regarding extreme eigenvalues of the deformed IRM ensemble   
 $X_N$ defined  in Definition  \ref{def:inhomo} below:

\begin{itemize}
    \item 
{\bf   Question 1.}  {\it Under the sharp condition  $\sigma^{*}_N \sqrt{\log N}  \to 0$, does   the BBP phase transition   for the deformed IRM ensemble   $X_N$    hold true at the level of the law of large numbers? }
  \item 
{\bf  Question 2.}  {\it Under the same sharp condition, 
what is the limiting distribution  for the  fluctuations  of  extreme eigenvalues in 
  the deformed IRM ensemble $X_N$?     }
\end{itemize}

\subsection{Models and main results } 

Regarding these two questions the primary models under  consideration   are additive  deformations of symmetric  and Hermitian IRM with independent symmetrically  sub-Gaussian entries.

\begin{definition}[Inhomogeneous symmetric/Hermitian random matrices] \label{def:inhomo}
An inhomogeneous  symmetric/Hermitian random matrix $H_N=(H_{ij})_{i,j=1}^N$ with sub-Gaussian entries  is the Hadamard product of a Wigner matrix $W_N$ and  a deterministic  matrix  $\Sigma_N$
\begin{equation}
H_N=\Sigma_N \circ W_N,  
\end{equation}
   and correspondingly  the deformed IRM is 
\begin{equation}\label{eq:deformed-iid}
    X_N=H_N+A_N,
\end{equation} where  $W_N$,   $\Sigma_N$ and $A_N$ satisfy  the  following assumptions.
\begin{itemize}
   
     \item[{\bf(A1)}] ({\bf Wigner matrix})   $W_N=(W_{ij})_{i,j=1}^N$ is   a  real symmetric or complex Hermitian matrix  whose entries 
    on and above the diagonal    are  independent and  symmetrically distributed  random variables.  Also,  for all $i,j\in [N]:=\{1,2,\ldots,N\}$, the following conditions are  assumed to hold:  \\
    (i)  
   (Real  Wigner) 
   $\mathbb{E}[W_{ii}^{2}]=2, \mathbb{E}[W_{ij}^{2} ]=1 \ (i\neq j)$, 
   or   (Complex    Wigner) 
   $\mathbb{E}[W_{ii}^{2}]=1, \  \mathbb{E}[|W_{ij}|^{2} ]=1,\  \mathbb{E}[W_{ij}^{2} ]=0\ (i\neq j)$,\\
   (ii) uniform    higher  moments (sub-Gaussian)
$\mathbb{E}[|W_{ij}^{2k}| ]\le \gamma^{k-1}(2k-1)!!, \forall k\geq 2$, for    some fixed constant  
$\gamma>0$.

In particular,   the  classical Gaussian orthogonal/unitary ensemble  is denoted by
$W_{\text{G}\beta \text{E}_N}$ ($\beta=1,2$).
\item[{\bf(A2)}] ({\bf Transition 
matrix})   $\Sigma_N=(\sigma_{ij})_{i,j=1}^N$ is a symmetric matrix with non-negative entries  such that 
  $P_N=(\sigma^2_{ij})_{i,j=1}^N$ is a  stochastic matrix,  i.e., $\sum_{j=1}^N \sigma^2_{ij}=1$ for all $i$.
  
\item[{\bf(A3)}] ({\bf Rank-$r$ perturbation}) $A_N=(A_{ij})_{i,j=1}^N$ is  a  deterministic    matrix of
the same symmetry  as  $W_N$ admitting the representation
\begin{equation}
    A_N=\sum_{i,j=1}^r  a_{ij} E_{m_im_j},  
\end{equation}
 for some $r \in [N]$,   where  $m_1,  m_2, \ldots, m_r$ are distinct indices in 
  $[N]$   and $E_{ij}$ denotes a matrix with 1 only at $(i,j)$ and 0 otherwise.  

  \end{itemize}

\end{definition}

 The primary objective of this paper is twofold: 
\begin{itemize}
    \item   
  To completely characterize the BBP transition at the level of   the  law of large numbers;

  \item    
To establish the fluctuations   of outliers in the supercritical regime under the mild regularity condition $\sigma^{*}_N \log N \to 0$.
\end{itemize}

A complete resolution of {\bf Question 2} in the subcritical and critical regimes may present significant difficulties, as it involves the interplay of at least three crucial factors: sparsity, perturbation scale,  and geometric structure. 
Moreover, additional parameters, like  bandwidth, spatial dimension and symmetry class,  may also  profoundly influence the spectral properties.
This complexity is particularly evident in random band matrices, where extreme eigenvalues can exhibit a rich phase diagram   when the perturbation scale, bandwidth, and spatial dimension simultaneously approach their critical thresholds \cite{liu2023edge,sodin2010spectral}.

\bigskip

\begin{theorem}[BBP transition] \label{thm:LLN1}
With $X_N$ given in 
Definition  \ref{def:inhomo}, assume that the nontrivial eigenvalues of $A_N$  satisfy 
  \begin{equation}\label{equ:eigenvaluesa_i}
  a_1\ge a_2\ge \cdots \ge a_{r_+}> 0 > a_{-r_-}\ge \cdots \ge a_{-2}\ge a_{-1},
   \end{equation} where all $a_{j}$ and $a_{-j}$ are independent of $N$.  If $(r+1)\sigma^{*}_N  \sqrt{\log N} \to 0$ as $N\to \infty$, then 
for any fixed positive  integer $j$,  
the $j$-th  largest eigenvalue  of $X_N$ 

\begin{equation}\label{equ:lambdaj}
   \lambda_j(X_N) 
   \xrightarrow{\text{a.s.}}  
     \begin{cases}
        2, & a_j\leq 1,\\
        a_j+\frac{1}{a_j}, & a_j> 1,
    \end{cases}
\end{equation}
and the  $(N+1-j)$-th   largest eigenvalue  of $X_N$ 

\begin{equation}\label{equ:lambdaN-j}
   \lambda_{N+1-j}(X_N) \xrightarrow{\text{a.s.}}    \begin{cases}
        -2, & a_{-j}\geq -1,\\
        a_{-j}+\frac{1}{a_{-j}}, & a_{-j}<-1.
    \end{cases}
\end{equation}

\end{theorem}

Having established the law of large numbers, we   next turn to the fluctuations  of  
  the  first $q$  largest eigenvalues  of the  deformed IRM ensemble   $X_N$.  
Since $H_N$ is  not a mean-field  model, the local limits of  outliers   do  depend not only  on the eigenvalues of the perturbation $A_N$, but also on the geometric structures of both the variance profile  $\Sigma_N$ and the perturbation $A_N$.  
Our second main result, which establishes the non-universal fluctuations of the outlier eigenvalues in the Gaussian case, can now be stated as follows.
\begin{theorem}[Fluctuations of outliers]\label{main result}
Assume that   $X_N$   in 
Definition  \ref{def:inhomo} is a Gaussian  matrix     with $W_N=W_{\mathrm{G}\beta\mathrm{E}_N}$, and  that  the reduced matrix    $\widetilde{A}:=(a_{ij})_{ i,j=1}^r$  from   $A_N$ admits   a spectral decomposition  $\widetilde{A}=U^{*}  {\rm diag}(a_1,\cdots,a_r) U$ where $U$ is an $r \times r$ orthogonal/unitary matrix, and      $a>1, r$ and $  q\in \{1,\ldots,r\}$ are   fixed    such that    
   \begin{equation}   \label{finiteA}   
   a= a_1= a_2= \cdots=a_{q}>a_{q+1}\ge   \cdots   \ge a_{r} \ge -a. 
  \end{equation}   
 If 
\begin{equation}\label{assumption1}g_{ij}:=\lim_{N \to \infty}\frac{1}{(\sigma_{N}^{*})^2} \Big(-\delta_{ij}a^{-2}\big(P_{N}^2\big)_{m_im_j}+\sum_{k=2}^{\infty}\big(P_{N}^k\big)_{m_im_j} a^{-2k+2}\Big),
\end{equation}
\begin{equation}\label{assumption2}   \widetilde{\sigma}_{ij}:=\lim_{N \to \infty}\frac{\sigma_{m_im_j}}{\sigma^{*}_N}, 
\end{equation}
\begin{equation}\label{def:sigma_i}
 \chi_i:=\lim_{N \to \infty}\frac{1}{(\sigma_{N}^{*})^2 a^2}\sum_{y=1}^N\sigma_{m_iy}^4,
\end{equation}
and
\begin{equation}\label{def:tau_i}  
    \tau_i:=\lim_{N \to \infty}\frac{1}{(\sigma_{N}^{*})^2 a^2}\sum_{y=1}^N\sigma_{m_iy}^4\mathbb{E}[|W_{m_iy}|^4],
\end{equation}  
exist  for all $1 \le i,j \le r$, and also  if 
  $\sigma^{*}_N  \log N \to 0$ as $N\to \infty$, then  the  first $q$  largest eigenvalues  of $X_N$  
  
\begin{equation}
  \frac{a^2}{(a^2-1)\sigma^{*}_N }\Big(\lambda_1(X_N)-\rho_a\Big), \dots ,\ \frac{a^2}{(a^2-1)\sigma^{*}_N }
  \Big(\lambda_q(X_N)-\rho_a\Big)
\end{equation}  
converge weakly  to  the ordered  eigenvalues of a $q\times q$ random matrix 
\begin{equation}\label{equ:Z}
    Z_{\beta}=Q_{\beta}(H_{\mathrm{ID}}+{H_{\mathrm{Gaussian}}}+H_{\mathrm{Diag}}) Q_{\beta}^*,
\end{equation}
where  $Q_{\beta}$ is  a $q \times r$  sub-matrix from the first $q$ rows of 
$U$,  and  $H_{\mathrm{ID}},{H_{\mathrm{Gaussian}}},H_{\mathrm{Diag}}$ are independent random matrices of size $r\times r$ such  that 

\begin{itemize}
    \item $H_{\mathrm{ID}}$ is identically  distributed as   $ \big(\widetilde{\sigma}_{ij}W_{m_im_j}\big)_{i,j=1}^r$;
     \item  $  H_{\mathrm{Gaussian}}= \big(\sqrt{g_{ij}}\big)_{i,j=1}^r\circ W_{\mathrm{G}\beta\mathrm{E}_r}$;
    \item $H_{\mathrm{Diag}}=\mathrm{diag}\big( \mathcal{N}(0,\tau_1-\chi_1), \dots, \mathcal{N}(0,\tau_r-\chi_r)\big)$ with independent   Gaussian  diagonal  entries.   
\end{itemize}
\end{theorem}

We remark that the conclusion of Theorem~\ref{main result} is expected to remain valid for general sub-Gaussian inhomogeneous random matrices $X_N$ as  in Definition~\ref{def:inhomo}. The motivation for the decomposition of the matrix $Z_\beta$ is to isolate different sources of fluctuations. The sub-Gaussian generalization will be treated in a forthcoming paper. 

Several other immediate remarks are in order. 
\begin{remark} (Sparsity assumption) 
 Theorem \ref{thm:LLN1} establishes that when the rank $r$ is fixed, the condition $(r+1)\sigma^*_N\sqrt{\log N}\rightarrow 0$ on $\Sigma_N$ is sharp, since this constraint precisely prevents the emergence of spectral outliers in the unperturbed ensemble $H_N$. However, this condition is not optimal for potentially large low-rank perturbations $r$, as the weaker requirement $\sigma^*_N\sqrt{(r+1)\log N}\rightarrow 0$ suffices to guarantee the same conclusion in the Gaussian case.
Second, regarding Theorem \ref{main result}, the condition $\sigma^{*}_N \log N \to 0$   is required possibly only for technical reasons in the proofs of Theorem \ref{coro:expansion} and Lemma \ref{lemma:error_trace} and may not be optimal.
\end{remark}

\begin{remark} (Sub-Gaussian assumption) 
For the unperturbed IRM ensemble $H_N$ with entries possessing 
$\alpha$-exponential tails, the law of large numbers for the spectral norm $\|H_N\|_{\mathrm{op}}$ \eqref{spectralnorm} has been studied in \cite{benaych2014localization} and \cite[Remark 2.2]{altschuler2024spectral}, respectively. The relationship between the sharp sparsity threshold and the tail decay of the entries is, however, more subtle. For deformed random band matrices, Au \cite{Au23BBP} established a BBP transition at the level of the law of large numbers under  the assumptions of uniformly bounded finite moments and of bandwidth $W=N^{\epsilon}$. 
\end{remark}

\begin{remark} 
Recently, Bandeira, Cipolloni, Schröder, and van Handel \cite{bandeira2024matrixconcentrationinequalitiesfree}  independently established the BBP transition  for the largest eigenvalue   at the level of LLN, almost  under the condition  $\sigma^{*}_N  (\log N)^{2}  \to 0$; see, in particular,  Theorem 2.9, Theorem 3.1  along with Example 3.2 in \cite{bandeira2024matrixconcentrationinequalitiesfree}.  
Their approach also handles phase transitions in the anisotropic setting—that is, when the stochasticity condition (A2) in Definition \ref{def:inhomo} fails—as discussed in \cite[Section 2.3]{bandeira2024matrixconcentrationinequalitiesfree}. In contrast, our approach is fundamentally different and, moreover, yields results at the level of fluctuations.
\end{remark}

\subsection{Applications}

Inhomogeneous random matrices encompass many fundamental models through an appropriate specification of the variance profile $\Sigma_N$. To illustrate, we highlight several representative examples.

    {\bf\noindent  Model 1: Random band matrices in dimension $d$.}
    
  Introduce a  $d$-dimensional lattice 
 \begin{equation} 
    \Lambda_{L,d}=\Big( \big(-\frac{1}{2}L,\frac{1}{2}L\big] \times \cdots \times\big(-\frac{1}{2}L,\frac{1}{2}L\big]\Big) \bigcap \mathbb{Z}^d, 
\end{equation}
 and    a  canonical representative
of $x\in \mathbb{Z}^d$ through
\begin{equation} [x]_L: = \big(x+L\mathbb{Z}^d\big)\bigcap  \Lambda_{L,d}, 
\end{equation} 
 and let \begin{equation}N= |\Lambda_{L,d}|=L^d.\end{equation}
  Given a mild (e.g. continuous) symmetric density function  $f(x)$ on $\mathbb{R}^d$, i.e. $f(-x)=f(x)$, 
define    $\Sigma_N^{\rm (band)}=(\sigma_{ij})_{1 \le i,j \le N}$ via
\begin{equation}
		\sigma_{ij}^2=\frac{1}{M}\sum_{n \in \mathbb{Z}^d}f\Big(\frac{i-j+nL}{b_N}\Big), \quad M=\sum_{i\in \mathbb{Z}^d}f\big(\frac{i}{b_N}\big).	\end{equation}
Here    $b_N \le L/2$ is referred to as    the \emph{bandwidth} and    may depend on $L$.  

The deformation of random band matrices is defined as 
\begin{equation}
    X_N^{\rm (band)}:=\Sigma_{N}^{\rm (band)} \circ W_{\mathrm{G}\beta\mathrm{E}_N}+A_N
\end{equation}
where 
$A_N$ is  given 
 in Definition \ref{def:inhomo} (A3). In this case 
    $(\sigma^*_N)^2\sim \|f\|_{\infty}(b_N)^{-d}$ as $N\to \infty$. 
In  the case of $b_N/L\rightarrow 0$ (similar for $b_N\sim L$), if  
\begin{equation} \label{band-}
\lim\limits_{N\rightarrow\infty}\frac{m_i-m_j}{b_N}=x_{ij}\in \mathbb{R}^d\cup\{\infty\},
\end{equation}
then we have
\begin{equation}    (\sqrt{b_N})^{d}\sigma_{m_im_j}\rightarrow f(x_{ij}),
\end{equation}
and \begin{equation}
(b_N)^d(P_N^{k})_{m_im_j}\rightarrow (\underbrace{f*\cdots *f}_{k})(x_{ij})
\end{equation}
where the right-hand functions should  be continuous for all finite $x_{ij}$ (subject to the regularity of $f$)  and vanish if $x_{ij}=\infty$ 
\begin{equation}
{g}_{ij}
    =\begin{cases}
        \|f\|_{\infty}^{-\frac{1}{2}}\sum_{k\ge 2}a^{-2k+2} (\underbrace{f*\cdots *f}_{k})(x_{ij}),~~~i\ne j,\\
        \|f\|_{\infty}^{-\frac{1}{2}}\sum_{k\ge 3}a^{-2k+2} (\underbrace{f*\cdots *f}_{k})(x_{ij}),~~~i=j.
    \end{cases}
\end{equation}
\begin{theorem}[Band matrices]\label{thm:band_matrix}
With the same notations as in Theorem \ref{main result},  assume \eqref{finiteA} \eqref{def:sigma_i},  
\eqref{def:tau_i}   and \eqref{band-}. 
   If   $\log^2 N \ll (b_N)^d\ll N$, then  the  first $q$  largest eigenvalues  of $X_N^{\rm (band)}$ 
   \begin{equation*}
   \frac{a^2}{a^2-1}\frac{(b_N)^{\frac{d}{2}}}{\sqrt{\|f\|_{\infty}}}\Big(\lambda_1-a-\frac{1}{a},\lambda_{2}-a-\frac{1}{a},\ldots, \lambda_q-a-\frac{1}{a}\Big)
\end{equation*}converge weakly to those  of $Z_{\beta}$. 
\end{theorem}
 \begin{remark}
 The BBP transition for random band matrices at the level of LLN  follows directly from Theorem~\ref{thm:LLN1} under the condition $(r+1)(\sqrt{b_N})^{-d}\sqrt{\log N} \to 0$. In the special case $d=1$, with $r$ fixed and $b_N \ge N^{\epsilon}$ for any $\epsilon>0$, this LLN result was recently established by Au~\cite{Au23BBP}. Moreover, the fluctuation correlations in Theorem~\ref{thm:band_matrix} stem from three components: the eigenvector $Q_{\beta}$, the matrix $H_{\mathrm{ID}}$, and the factor ${g}_{ij}$.
  In particular, the factor ${g}_{ij}$ reflects the geometric structure of random band matrices and is solely related to the finite-step transition probabilities of random walk on the torus. This mechanism is fundamentally distinct from the edge statistics observed in non-deformed random band matrices \cite{sodin2010spectral,liu2023edge}.

 \end{remark}
    {\bf \bf\noindent Model 2: Random matrices with $\kappa$-regular   variance profile.} 
    
    Take   $\Sigma_N= \Psi_N/\sqrt{\kappa}$ where  $\Psi_N$ is a $\kappa$-regular matrix. In this case  $\sigma_N^*={1}/{\sqrt{\kappa}}$,  and we need to  assume 
 $\log^2 N\ll \kappa\ll N$.

\vspace{.2cm}

 {\bf \noindent Model 3: Inhomogeneous information-plus-noise model}.
 
 Consider two matrices in chiral form
\begin{equation}
    \Sigma_N = \begin{pmatrix} 0_{N/2} & \Phi_{N/2} \\ \Phi_{N/2}^t & 0_{N/2} \end{pmatrix}, \qquad
    A_N   = \begin{pmatrix} 0_{N/2} & B_{N/2}   \\ B_{N/2}^* & 0_{N/2} \end{pmatrix},
\end{equation}
where $N$ is even and $\Phi_{N/2}$ is a chiral matrix with nonnegative entries such that the $\ell^2$-norm of every row and every column vector equals $1$. In this setting, our main theorems, Theorem~\ref{thm:LLN1} and Theorem~\ref{main result}, remain valid.

\subsection{Method and structure }
To prove Theorem \ref{thm:LLN1} and Theorem \ref{main result}, we develop a ribbon graph expansion for very large powers of the deformed inhomogeneous matrix  $X_N$. Ribbon graph expansions have been widely employed in the study of classical GOE/GUE ensembles; see, for instance, \cite{kontsevich1992intersection, okounkov2009gromov, lando2004graphs} and references therein. However, to the best of our knowledge, their application to deformed models remains largely unexplored, even in the Gaussian case. While our approach relies on sharp upper bound estimates derived from the same philosophy that allows GOE to dominate sub-Gaussian  IRM models, the techniques required to establish Theorem \ref{thm:LLN1} and Theorem \ref{main result} differ significantly.

For Theorem \ref{thm:LLN1}, a key step involves applying sharp concentration inequalities for the central moments of the GOE matrix to control those of the sub-Gaussian  IRM. To achieve this, we introduce a series of domination inequalities that reduce the problem in the sub-Gaussian case to the Gaussian setting. 
A central  result is Theorem \ref{thm:spectral_measure} below concerning the almost sure convergence of spectral measures. 
This provides a sharper generalization of \cite[Proposition 2]{noiry2021spectral} and \cite[Lemma 3.7]{Au23BBP}, which study Wigner matrices and one-dimensional cut-off random band matrices, respectively. We then prove Theorem \ref{thm:LLN1} by adapting and extending the arguments used   in \cite{Au23BBP}.

Recall   that  the spectral measure of $X_N$  with respect to a unit (column) vector  $\mathbf{q}$ is defined as the unique probability measure $\mu_{X_N}^{\mathbf{q}}$ satisfying  
\begin{equation}
    \int x^m \mu_{X_N}^{\mathbf{q}}(dx)=  \mathbf{q}^{*} X_{N}^m \mathbf{q},\quad \forall  m\in \mathbb{N}.
\end{equation}
  For further details on spectral measures, we refer to \cite{noiry2021spectral,Au23BBP}.

\begin{theorem}\label{thm:spectral_measure}
    With $X_N$ given in 
Definition  \ref{def:inhomo},  let 
$\mathbf{q}_i$ be a  unit eigenvector of $A_N$  associated with eigenvalue   $a_i$ and  let $\mu_{X_N}^{\mathbf{q}_i}$   be  a   spectral measure of $X_N$ with respect to $\mathbf{q}_i$. If 
  $(r+1)\sigma^*_N \sqrt{\log N}\rightarrow 0$ as $N\to \infty$,  then for each $1\le i\le r$,    $\mu_{X_N}^{\mathbf{q}_i}$ 
  converges weakly, almost surely, to 
\begin{equation}
     \mu_{a_i}(dx)= \frac{1_{\{|x|\le 2\}}}{2\pi}\frac{\sqrt{4-x^2}}{a_i^2+1-a_ix}dx+1_{\{|a_i|\ge 1\}}\Big(1-\frac{1}{a_i^2}\Big)\delta_{a_i+\frac{1}{a_i}}(dx), 
\end{equation}
where $1_{\{\cdot\}}$ denotes the indicator function and $\delta$ is the Dirac measure.
\end{theorem}

 As to  Theorem \ref{main result},  we detect the fluctuation by computing large powers of moments.    Following a similar approach, we refer to \cite{Férall2007deform} for rank-one deformed Wigner matrices, \cite{okounkov2000random} for non-deformed GUE, \cite{soshnikov1999universality,feldheim2010universality} for non-deformed Wigner matrices, and \cite{sodin2010spectral,liu2023edge} for non-deformed random band matrices. The proof of Theorem \ref{main result} reduces to establishing the following key theorem. 

\begin{theorem}\label{Thm-a>1:baby}
With  the same notations  and assumptions  as in Theorem \ref{main result},   
   for any given positive integer $s$,   let $k_i=\lfloor t_i/\sigma_N^*\rfloor $ with $t_i>0$,   $i=1,\dots,s$,  then   
    \begin{equation}\label{equ:main theorem}
       \lim_{N\to \infty} 
\frac{1}{2^s}\mathbb{E}\bigg[\prod_{j=1}^{s}\Big(\tr \Big(\frac{X_{N}}{\rho_a}\Big)^{k_j}+\tr \Big(\frac{X_N}{\rho_a}\Big)^{k_j+1}\Big)\bigg] = \mathbb{E}\bigg[\prod_{j=1}^{s}\tr  
       \exp\!\Big\{ \frac{(a^2-1)}{(a^2+1)a} t_jZ_{\beta}\Big\} 
      \bigg].
    \end{equation}
\end{theorem}

The core of our proof for Theorem \ref{Thm-a>1:baby} follows a three-stage approach: (i) first developing a diagrammatic expansion for high powers of the  IRM through Okounkov's contraction technique; (ii) then constructing dominating functions by estimating diagram-specific upper bounds and establishing term-wise convergence; (iii) finally,  the diagram upper bounds are summable and thus the limit and the summation are exchangeable.  We now elaborate on the technical details of this argument.

\begin{itemize}
    \item[\bf Step (i).] \textbf{Diagram expansions.} 
While diagram expansions have been extensively studied for GUE \cite{okounkov2000random,feldheim2010universality} and unimodular random band matrices \cite{sodin2010spectral,EK11Quantum,liu2023edge}, our approach requires substantial modifications to the reduction procedure. The key innovation involves decomposing joint moments into sums of diagram functions via Wick formulas, first establishing results for deformed Gaussian IRM models before extending to the sub-Gaussian case (see Section \ref{section:Genus expansions and map enumeration}).
A crucial distinction in deformed models arises from the ribbon graphs being associated with Riemann surfaces with boundaries. To the best of our knowledge, this geometric connection, along with its implications for matrix models, has only recently attracted significant attention in the literature \cite{buryak2017matrix,tessler2023combinatorial}.

\item[\bf Step (ii).] \textbf{Upper bounds for   diagram functions.}
  We establish uniform upper bounds for individual diagram functions in the GOE ensemble (Section~\ref{section:upper_bound_nontypical}),  with the help of 
   the procedure  of the automation developed in \cite{feldheim2010universality} to control the enumeration of trivalent diagrams. The extension to inhomogeneous random matrices  is achieved through a comparison principle (Section~\ref{Section:Inhomogeneous}), which dates back    at least  to  \cite[Equation (7.8)]{erdHos2011quantum}, \cite{EK11Quantum},    \cite[Lemma 2.5]{bandeira2016sharp} and \cite[Theorem 2.8]{latala2018dimension}.

    \item[\bf Step (iii).]  
\textbf{Diagram classification and diagram-wise limits.}
We introduce  a classification of diagrams into typical and non-typical categories according to their role in detecting fluctuation properties. Through careful analysis of the typical diagrams, we identify and rigorously derive the dominant terms that ultimately give rise to the Laplace transforms presented in Section~\ref{sec:fluctuation Gaussian}.
\end{itemize}

\textbf{Comparison with existing proof strategies.} While our approach of combining upper bound estimates with termwise limits bears some similarity to that  in \cite{okounkov2000random}, it differs fundamentally from the   moment method that directly computes error terms \cite{soshnikov1999universality,feldheim2010universality,Férall2007deform}. In \cite{okounkov2000random}, the construction of dominating functions  strongly depends   on the Harer-Zagier formula \cite{harer1986euler} for map enumeration on closed Riemann surfaces (which is specific to GUE and inapplicable to GOE, and furthermore only addresses single moments rather than the mixed moments considered here). This three--stage strategy is  of  its own interest and has  recently been employed in \cite{liu2023edge} for studying extreme eigenvalues of random band matrices.

\textbf{The structure of this paper is as follows.} 
Section \ref{section:Genus expansions and map enumeration} develops  diagram function expansions through ribbon graph reduction, while establishing necessary graph-theoretic tools and introducing the key classification of diagrams into typical and non-typical categories. Section \ref{Section:GUE/GOE case} provides rigorous upper bound estimates 
for the diagram functions.    Section \ref{sec:proof_of_main} establishes  Theorem \ref{thm:spectral_measure} and Theorem \ref{thm:LLN1} through   a system of domination inequalities for central moments. Finally, Section \ref{sec:fluctuation Gaussian} completes our analysis by determining the limiting behavior of diagram functions, thereby proving Theorem \ref{Thm-a>1:baby} and Theorem \ref{main result}.

\section{Ribbon graphs and graph expansions}\label{section:Genus expansions and map enumeration}

This section introduces ribbon graphs and ribbon graph expansions, which are central to    analyzing spectral properties of the deformed inhomogeneous model. Our investigation starts from the Gaussian IRM ensemble and is subsequently extended to the sub-Gaussian case.

\subsection{Ribbon graph expansions}\label{sec:ribbon_expansion}

\noindent\textbf{Ribbon graphs and diagrams. }
We begin by outlining the canonical procedure for constructing a punctured ribbon graph $\Upsilon$ through polygon gluing. Consider $s$ oriented polygons $D_1, \dots, D_s$ with $k_1, \dots, k_s$ edges, respectively. Let $k = \sum_{i=1}^s k_i$ be the total number of vertices, and let $\gamma=(12\cdots k_1)(k_1+1\cdots k_1+k_2)\cdots(k_1+\cdots +k_{s-1}+1\cdots k)\in \mathcal{S}_k$ be the permutation representing the cyclic order of $D_1,\dots,D_s$. The vertices are labeled as $v_1, v_2, \dots, v_k$, and the edges are defined as $\vec{e_j} = \overrightarrow{v_j v_{\gamma(j)}}$.  Given    a subset of the edges $J\subset [k]$,  with $\pi$ as   a pairing of $J$,  we glue the edges of polygons as follows:

\begin{itemize}
\item For each pair $(s,t) \in \pi$, when $\beta=2$, we glue $\vec{e_s},\vec{e_t}$ in the opposite direction. In particular, we identify $v_s=v_{\gamma(t)},v_t=v_{\gamma(s)}$.  When $\beta=1$, we glue $\vec{e_s},\vec{e_t}$   either in an opposite direction or in the same direction (hence the resulting ribbon graph may be  non-orientable in this case).
 
 \item For each $j \in J^c$, $\vec{e_j}$ is not glued with any other edge, which corresponds to a  matrix entry from the deformation $A$ and is referred to as  a boundary edge. 
 \end{itemize}
This construction yields a punctured (open) ribbon graph with perimeter $(k_1, \dots, k_s)$ on a compact surface, as illustrated in Figure \ref{fig:ribobon graph1} and Figure \ref{fig:ribobon graph2}.

\begin{figure}[ht]
\centering \includegraphics[scale=0.17]{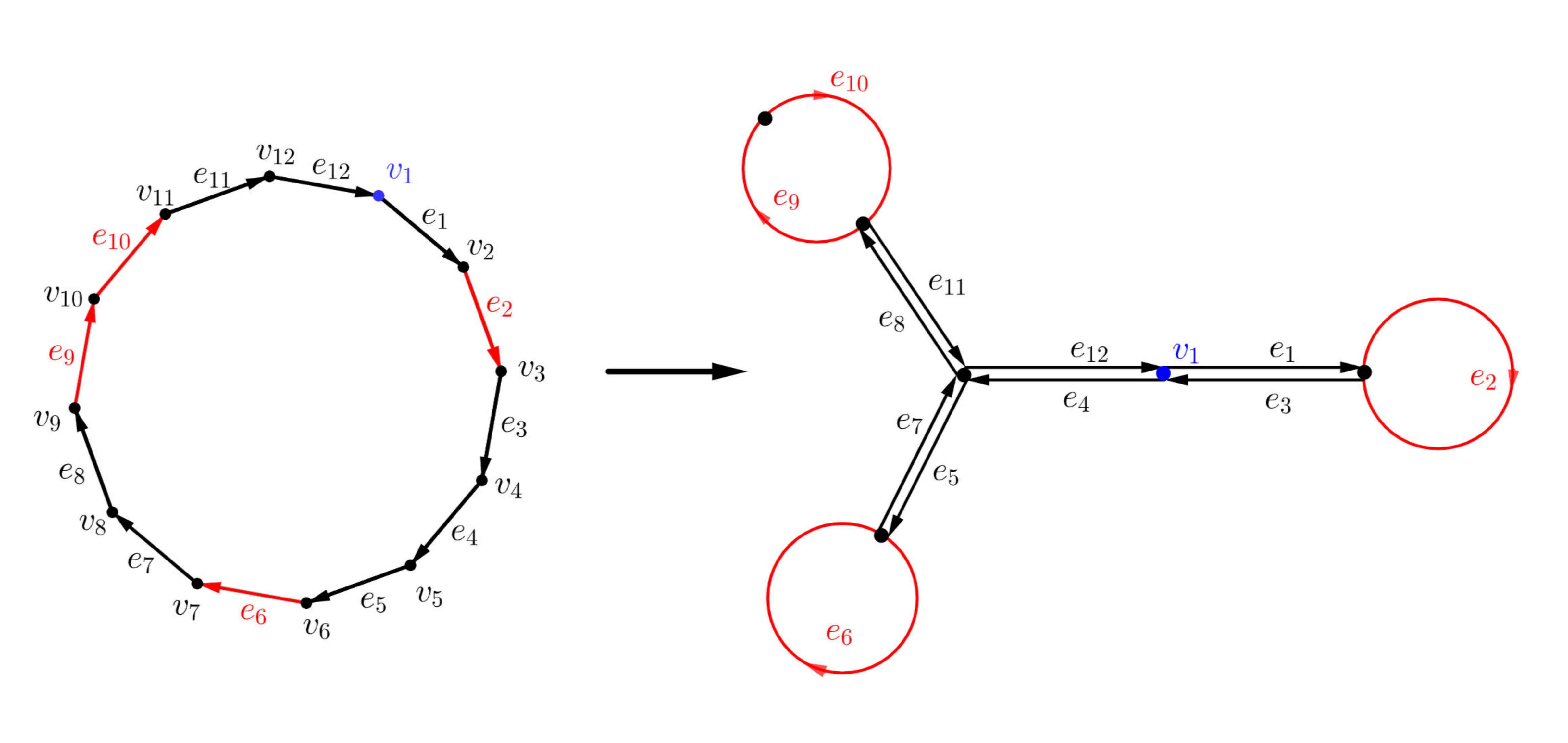}
     \caption{An example of oriented ribbon graph: $k=12$,   $J=\{1,3,4,5,7,8,11,12\}$, and $\pi=(1\ 3)(4\ 12)(5\ 7)(8 \ 11)$. The boundary edges are colored red. } 
     \label{fig:ribobon graph1}
\end{figure}

\begin{figure}[ht]
\centering \includegraphics[scale=0.15]{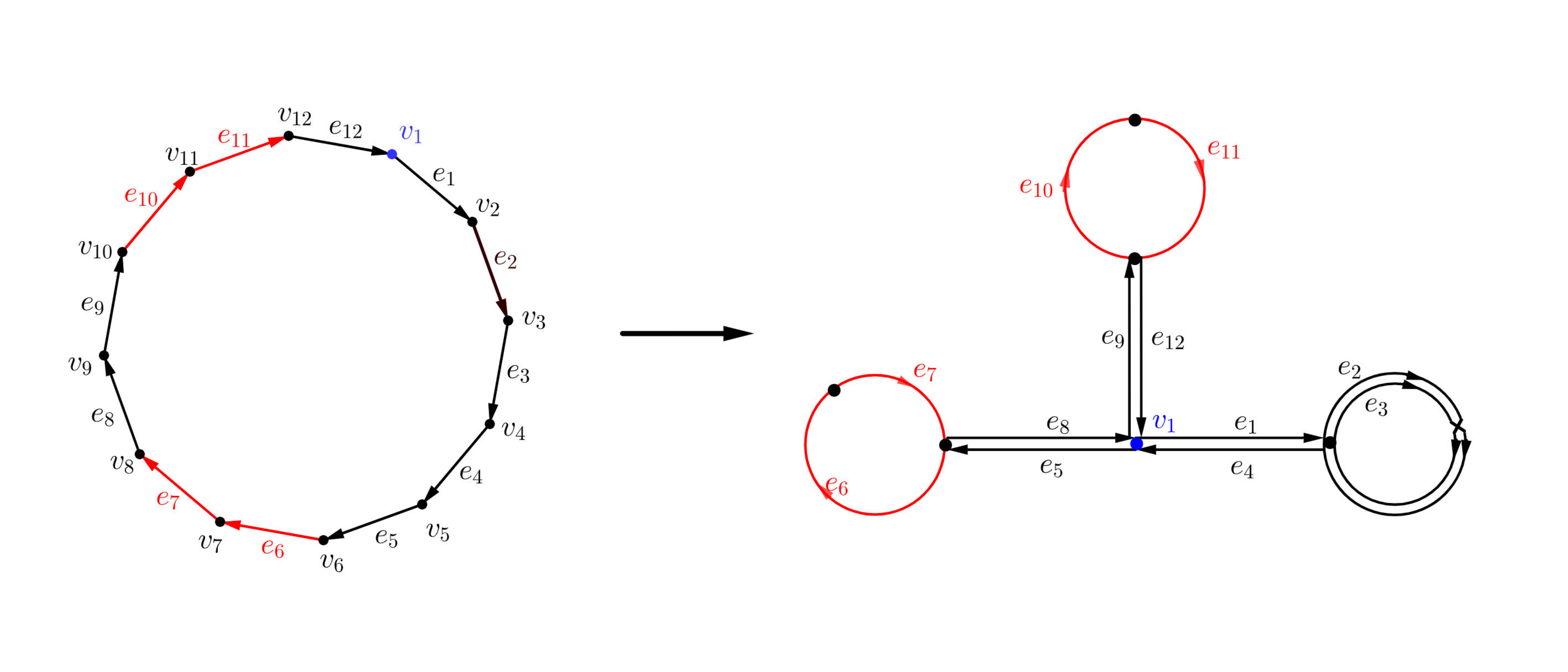}
     \caption{An example of nonoriented ribbon graph: $k=12$,   $J=\{1,2,3,4,5,8,9,12\}$, and  $\pi=(1\ 4)(2\ 3)(5\ 8)(9 \ 12)$. The boundary edges are colored red, and $\vec{e_2},\vec{e_3}$ are glued in an opposite direction. }
     \label{fig:ribobon graph2}
\end{figure}

 \bigskip

Formally, a ribbon graph possibly  with  boundary  is defined as a graph embedded in a compact surface with boundary, such that the complement of the graph is a disjoint union of open disks \cite{buryak2017matrix,lando2004graphs,tessler2023combinatorial}.

\begin{definition}
Let $\Sigma$ 
be a compact topological surface 
with or without boundary. A (punctured)
\textbf{ribbon graph} $\Upsilon$ of $s$ faces with perimeter $(m_1,\dots,m_s)$ on $\Sigma$ is a quadruple $\Upsilon=(\mathcal{V}(\Upsilon),\mathcal{E}(\Upsilon),\iota,\phi)$, where
\begin{enumerate}
    \item[(1)] $(\mathcal{V}(\Upsilon), \mathcal{E}(\Upsilon))$   is a graph;
    \item[(2)] $\iota:(\mathcal{V}(\Upsilon),\mathcal{E}(\Upsilon))\to \Sigma$ is   an embedding;
    \item[(3)] $\phi:[s] \rightarrow \mathcal{V}(\Upsilon)$ is a function that assigns a marked vertex to each face;
\end{enumerate}
  such that
 \begin{itemize}
 \item [(i)] the boundary of the surface  lies in the graph: $\partial \Sigma \subset \iota(\Upsilon)$;

 \item [(ii)] the complement 
$\Sigma\backslash \iota(\Upsilon) =
 \sqcup_{i=1}^{s}D_i$,
 where each $D_i$ is an oriented  $m_i$-gon and the vertex $\phi(i)$ (called the $i$-th marked point) is on the boundary of $D_i$;

 \item [(iii)] the orientations of $D_i$ are compatible when $\beta=2$.
 \end{itemize}
\end{definition}

\begin{definition}\label{def:diagram}
    A \textbf{diagram} (reduced ribbon graph) $\Gamma$ is a ribbon graph such that
\begin{itemize}
    \item[(i)] the degree of each unmarked vertex is at least 3, and 
\item[(ii)]the degree of each marked vertex is at least 2. 
\end{itemize}
In particular, $\Gamma$ is called  
  a {\bf trivalent} diagram if every unmarked vertex has degree exactly 3 and every marked vertex for $\Gamma$ has degree exactly 2.
\end{definition}
For convenience, we also include diagrams with connected components consisting of a single vertex in the definition.
Throughout this paper, the term ``ribbon graph" refers to the punctured ribbon graphs defined above. We will frequently work with such a graph $\Upsilon$ or a diagram $\Gamma$. To describe their geometric structure precisely, we introduce the following notations.

\begin{itemize}
\item For each ribbon graph $\Upsilon$, we denote   the set of edges (resp. interior edges, boundary edges) of $\Upsilon$ by $\mathcal{E}(\Upsilon)$ (resp. $\mathcal{E}_{\rm int}(\Upsilon)$, $\mathcal{E}_{b}(\Upsilon)$). Similarly, we    denote the set of vertices (resp. interior vertices, boundary vertices) of $\Upsilon$ by  $\mathcal{V}(\Upsilon)$ (resp. $\mathcal{V}_{\rm int}(\Upsilon)$, $\mathcal{V}_{b}(\Upsilon)$).
\item  For each diagram $\Gamma$,  we denote the set of its edges (resp. interior edges, boundary edges) by  $E(\Gamma)$ (resp. $E_{\rm int}(\Gamma)$, $E_b(\Gamma)$). Similarly,  we denote the set of its vertices (resp. interior  vertices, boundary  vertices) by $V(\Gamma)$, (resp. $V_{\rm int}(\Gamma)$, $V_b(\Gamma)$)  . Moreover, for any $j \in [s]$, let  $\partial D_j$ be the set of edges on the $j$-th face and denote the set of boundary edges on the $j$-th face by $E_{b,j}=E_b \cap \partial D_j$.
\item The set of all diagrams with $s$ faces is denoted by $\mathcal{D}_{\beta=1,s}$, while  the set of  all orientable diagrams with $s$ faces is denoted by $\mathcal{D}_{\beta=2,s}$. Moreover, let $\mathcal{D}_{\beta,s}^{b}$ (resp. $\mathcal{D}_{\beta,s}^{=3,b}$) be the set of  all (resp. trivalent) diagrams with $s$ faces such that each face contains at least one boundary edge, namely $E_{b,j}\neq \emptyset$ for each $j$.

\item For any diagram $\Gamma$, let $t(\Gamma)$ be the number of boundary components; the Euler genus is defined as
\begin{equation}\label{equ:genus}
    g(\Gamma) :=-V(\Gamma)+E(\Gamma)-F(\Gamma)+2-t(\Gamma).
\end{equation}
\end{itemize}

 \bigskip

\noindent\textbf{Okounkov's contraction.} Here, we introduce Okounkov's contraction \cite{okounkov2000random}, a key operation for studying large powers of random matrices that was originally developed to analyze the GUE model. Heuristically, it reduces a ribbon graph $\Upsilon$ to a diagram $\Gamma$ by deleting all of its local trees.

\begin{definition}\label{def:ok_contraction}

The \textbf{Okounkov  contraction}     
\begin{equation}
    \Phi: \big\{ \Upsilon:\text{ribbon graphs with perimeters $(k_1,\cdots,k_s)$} \big\} \longrightarrow
    \ \big\{\Gamma: \text{diagrams with $s$ faces}\big\}
\end{equation} 
is a surjection defined by the following procedure, as illustrated in Figures \ref{fig:Okounkov reduction 2} and \ref{figure:2}:
\begin{itemize}
\item If the marked point of $\Upsilon$ lies on a (local) tree, move it to the root of that tree.
\item Collapse all univalent vertices. 
\item Collapse all divalent vertices, except for the marked points.
\end{itemize}
 
\end{definition}

\begin{figure}
    \centering 
    \includegraphics[scale=0.50]{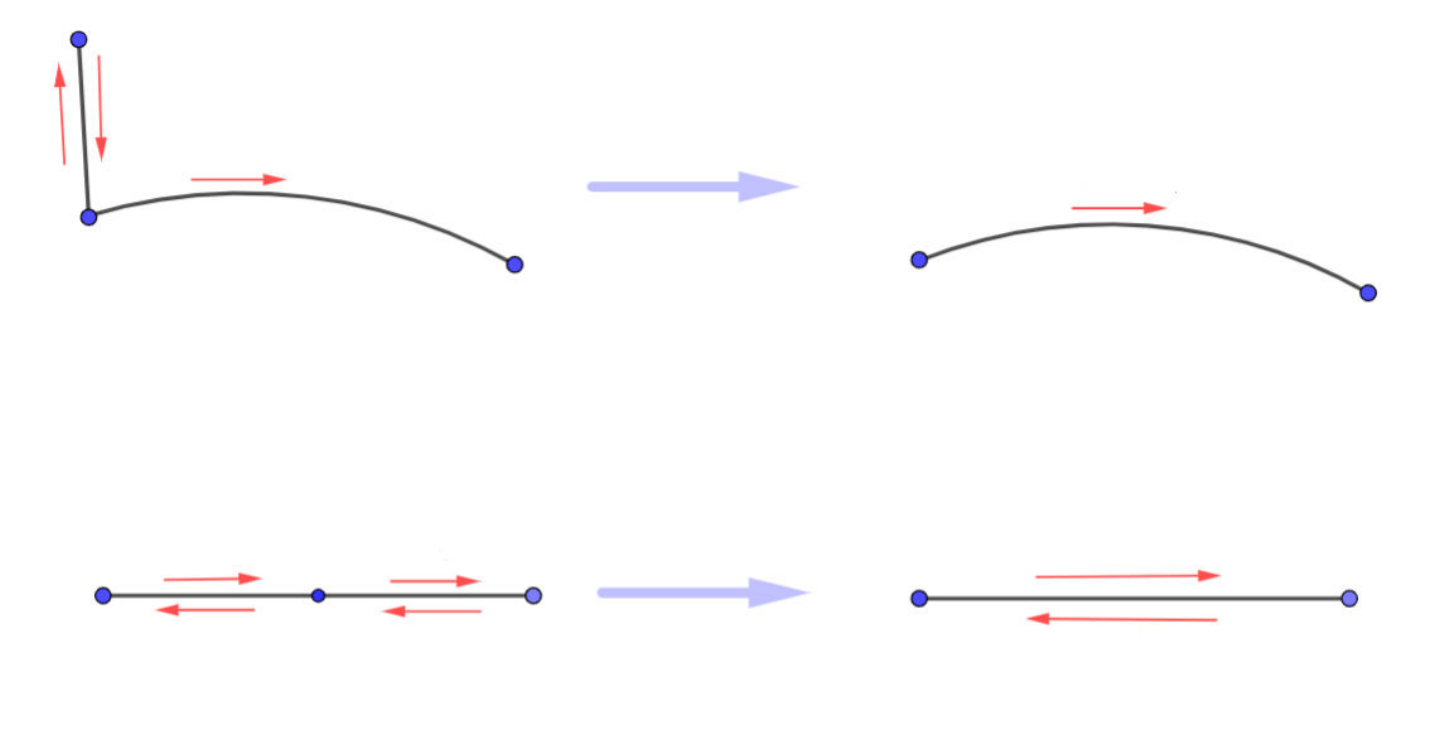}
     \caption{Okounkov's contraction} 
     \label{fig:Okounkov reduction 2}
\end{figure}

\begin{figure}[ht]
\centering
\includegraphics[scale=0.5]{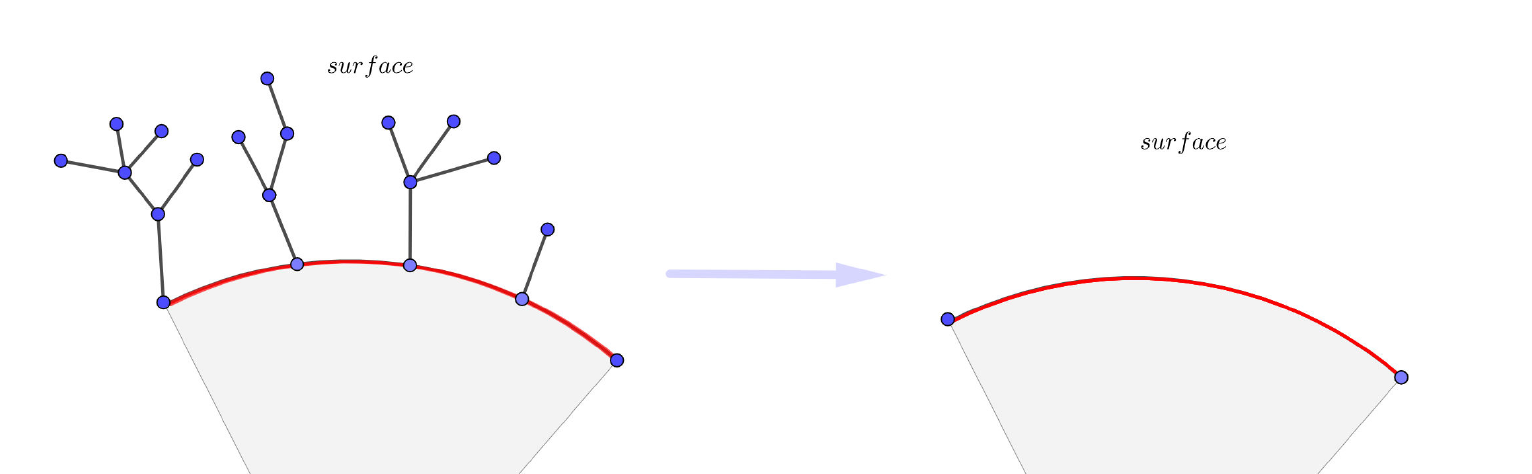}
    \label{fig:reduction_boundary}
    \caption{Reduction of a boundary edge}
    \label{figure:2}
\end{figure}
It is worth noting that the Okounkov contraction preserves the topology of a ribbon graph. In particular, it maps orientable ribbon graphs to orientable diagrams. 
\bigskip

\noindent\textbf{Diagram expansion in the Gaussian case. } We now introduce the diagram function for deformed Gaussian IRM.
\begin{definition} \label{diagramf}

Given non-negative integers $k_1,\dots ,k_s$, set $k=k_1+\cdots +k_s$. For the deformed symmetric (resp. Hermitian) Gaussian IRM $X$ defined in Definition \ref{def:inhomo} and for  any diagram $\Gamma$,  the  {\bf diagram function} of $\Gamma$ associated with $X$ is defined as
\begin{equation}\label{equ:diagram-function}
    \mathcal{W}_{\Gamma, X}(k_1,\dots,k_s):=\rho_a^{-k}\sum_{\Upsilon:\Phi(\Upsilon)=\Gamma}\sum_{\eta:\mathcal{V}(\Upsilon)\to [N]}
\prod_{(x,y) \in \mathcal{E}_{\rm int}(\Upsilon)}\sigma_{\eta(x)\eta(y)}^2 \prod_{(z,w) \in \mathcal{E}_b(\Upsilon)}A_{\eta(z)\eta(w)}.
\end{equation}
Here, the summation $\sum_{\Upsilon}$ is over all ribbon graphs $\Upsilon$ whose reduced diagram $\Phi(\Upsilon)$ is $\Gamma$, and $\sum_{\eta}$ is over all mappings from the vertex set $\mathcal{V}(\Upsilon)$ to $[N]$.
\end{definition}

The following proposition presents a ribbon graph expansion for the deformed Gaussian IRM, expressing its mixed moments in terms of the corresponding diagram functions.

\begin{proposition}\label{prop:ribbon_graph_expansion}
For $\beta=1,2$, let $X$ be the  deformed Gaussian IRM  matrix defined in   \eqref{eq:deformed-iid}, where $W_N$ is {\rm GOE} $(\beta=1)$ or {\rm GUE} $(\beta=2)$. For any  non-negative integers $k_1,\dots ,k_s$, 
we have 
    \begin{equation}
  \rho_{a}^{-(k_1+\cdots +k_s)}\mathbb{E}[\prod_{j=1}^{s}\tr X^{k_j}]=
  \sum_{\Gamma \in \mathcal{D}_{\beta,s}}\mathcal{W}_{\Gamma,X}(k_1,\dots,k_s).
\end{equation}
\end{proposition}

\begin{proof}
Let $k=k_1+\cdots +k_s$. We introduce a permutation  $\gamma=(12\cdots k_1)(k_1+1\cdots k_1+k_2)\cdots(k_1+\cdots +k_{s-1}+1\cdots k)\in \mathcal{S}_k$. For   any function of the labeling denoted by $\eta:[k]\to[N]$, a straightforward calculation shows   
\begin{align}\label{equ:Trace_X}
\mathbb{E}\Big[\prod_{j=1}^{s}\tr X^{k_j}\Big]&=\sum_{\eta:[k]\to[N]}\mathbb{E}\Big[\prod_{j=1}^{k}\big(H_{\eta(j)\eta(\gamma(j))}+A_{\eta(j)\eta(\gamma(j))}\big)\Big]\nonumber\\
&=\sum_{\eta:[k]\to[N]}\sum_{J \subset [k]} \mathbb{E}\Big[\prod_{j \in { J}}H_{\eta(j)\eta(\gamma(j))}\Big]\prod_{j \in { J^c}}A_{\eta(j)\eta(\gamma(j))}.
\end{align}
where $\mathcal{P}_2(J)$ denotes all pairings of $J$.
Noting  the symmetry of $H$,  by the Wick formula 
  we see when $\beta=1$ 
\begin{align}
\mathbb{E}\big[\prod_{j \in { J}}H_{\eta(j)\eta(\gamma(j))}\big]&=\sum_{\pi \in \mathcal{P}_2(J)} \prod_{(s,t) \in \pi}\mathbb{E}\big[H_{\eta(s)\eta(\gamma(s))}H_{\eta(t)\eta(\gamma(t))}\big]\nonumber\\
&
=\sum_{\pi \in \mathcal{P}_2(J)} \prod_{(s,t) \in \pi} \sigma_{\eta(s)\eta(\gamma(s))}^2\Big(\delta_{\eta(s)\eta(\gamma(t))} \delta_{\eta(t)\eta(\gamma(s))}+\delta_{\eta(\gamma(t))\eta(\gamma(s))}\delta_{\eta(s)\eta(t)}\Big).
\end{align}
If $|{ J}|$ is odd, then $\mathbb{E}\big[\prod_{j \in { J}}H_{\eta(j)\eta(\gamma(j))}\big]$ vanishes. Hence,
\begin{multline}
\label{equ:2.1RHS}
\mathbb{E}[\prod_{j=1}^{s}\tr X^{k_j}]=\\
\sum_{J \subset [k]} \sum_{\pi \in \mathcal{P}_2(J)} \sum_{\eta:[k]\to[N]}\prod_{j \in  J^c}A_{\eta(j)\eta(\gamma(j))}\prod_{(s,t) \in \pi} \sigma_{\eta(s)\eta(\gamma(s))}^2\Big(\delta_{\eta(s)\eta(\gamma(t))} \delta_{\eta(t)\eta(\gamma(s))}+\delta_{\eta(\gamma(t))\eta(\gamma(s))}\delta_{\eta(s)\eta(t)}\Big).
\end{multline}

By the construction of the ribbon graph, the above summation precisely enumerates the ribbon graphs $\Upsilon$ associated with subset $J$ and pairing $\pi$. Namely, 

\begin{equation}\label{equ:wick expansion beta=1ii}
  \rho_{a}^{-k}\mathbb{E}[\prod_{j=1}^{s}\tr X^{k_j}]=\sum_{\Upsilon} \rho_a^{-k}\sum_{\eta:\mathcal{V}(\Upsilon)\to[N]}\prod_{(x,y) \in \mathcal{E}_{int}(\Upsilon)}\sigma_{\eta(x)\eta(y)}^2 \prod_{(z,w) \in \mathcal{E}_{b}(\Upsilon)}A_{\eta(z)\eta(w)}.
\end{equation}
Classifying $\Upsilon$ according to the contraction map $\Phi(\Upsilon)$ and recalling the definition of $\mathcal{W}_{\Gamma, X}(k_1, \dots, k_s)$, we thus arrive at
\begin{equation}
\rho_{a}^{-k}\mathbb{E}[\prod_{j=1}^{s}\tr X^{k_j}]=\sum_{ \Gamma \in \mathcal{D}_{1,s}} \mathcal{W}_{\Gamma}(k_1,\dots,k_s).
\end{equation}

When $\beta=2$, the proof is similar and follows from the construction of orientable ribbon graphs. 
\end{proof}

Indeed,   the diagram function $\mathcal{W}_{\Gamma}(k_1,\dots,k_s)$ in Definition \ref{diagramf} admits a simplification in terms of transition probabilities for the induced Markov chain from the variance profile matrix. For convenience, we adopt the convention that $\binom{n}{k}=0$ whenever $n$ or $k$ is not an integer, and $\partial D_j =\emptyset$ means the $j$-th face $D_j$ consists of a single point.

\begin{proposition}\label{prop:compute_diagram_function}
   Let $p_{n}(x,y)$ be the $n$-th step transition probability  associated with  the Markov matrix   $P_N$
  in Definition  \ref{def:inhomo};  
  then  for any diagram $\Gamma$ we have
\begin{multline}\label{equation:W_Gamma}
\mathcal{W}_{\Gamma,X}(k_1,\dots,k_s)=\rho_a^{-(k_1+\cdots+k_s)}\prod_{j: \partial D_j=\emptyset}C_{\frac{k_j}{2}}\sum_{\eta:V(\Gamma) \to [N]}\sum_{n_{e}\ge 1, {e \in E(\Gamma)}}\prod_{j:\partial D_j\neq\emptyset}\binom{k_j}{\frac{k_j-\sum_{\partial D_j}n_e}{2}}\\
\times\prod_{(x,y) \in E_{\mathrm{int}}}p_{n_e}(\eta(x),\eta(y))
\prod_{(z,w) \in E_b}(A^{n_e})_{\eta(z)\eta(w)}.
\end{multline}
Here  $C_n:=\frac{1}{n+1}\binom{2n}{n}$ is  
the $n$-th Catalan number.

\end{proposition}
\begin{proof}  
We count the pre-images of $\Gamma$ under Okounkov's contraction function $\Phi$.  
For any edge $e \in E(\Gamma)$, suppose there are exactly $n_e-1$ divalent vertices deleted in the contraction, that is,  there are exactly $n_e$ segments before the reduction. We see the total number of divalent vertices deleted in $\partial D_j$ is $n_j:=\sum_{e \in \partial D_j}n_e$. 

We now compute the number of trees on the $j$-th face. If  $\partial D_j\neq\emptyset$, that is, the component of $D_j$ is not a single point, the total number of all possible trees on the $j$-th face is  
\begin{equation}
    \mathcal{T}_j=\sum_{ u_1+\cdots+ u_{n_j}=(k_j-n_j)/2
    }(2u_1+1)\prod_{i=1}^{n_j} C_{u_i},
\end{equation}
where all $u_i$ are nonnegative integers. 
Noting the symmetry of the above summation with respect to $u_j$, and then using the  Catalan multi-fold convolution formula \cite{catalan1887nombres} 
\begin{equation}
    \sum_{i_1+i_2+\cdots+i_n=m} C_{i_1 } \cdots C_{i_n} = \frac{n}{2m+n} \binom{2m+n}{m}
\end{equation} 
where the summation is taken over all nonnegative integers $i_l\geq 0$;   also see, for instance, \cite[Eqn (14)]{larcombe2003catalan}, we obtain
\begin{align} \label{equ:tree}
        \mathcal{T}_j=&\sum_{ u_1+\cdots+ u_{n_j}=(k_j-n_j)/2}\frac{\sum_{l=1}^{n_j}(2u_l+1)}{n_j}\prod_{i=1}^{n_j} C_{u_i}\nonumber \\
       = & \frac{k_j}{n_j} \sum_{ u_1+\cdots+ u_{n_j}=(k_j-n_j)/2} \prod_{i=1}^{n_j} C_{u_i}\nonumber 
       =  \binom{k_j}{\frac{k_j-n_j}{2}}.
\end{align}

While when ${\partial D_j}=\emptyset$, from the definition of the Catalan numbers we see (see, for instance \cite[Section 2.2]{okounkov2000random}) $\mathcal{T}_j=C_{{k_j}/{2}}$.
Fix $\eta$'s value at the endpoint and  sum over all possible vertices,  we thus obtain the  factor of $(A^{n_e})_{\eta(z)\eta(w)}$.

Similarly, as for   $e\in E_{\mathrm{int}}$, fix $\eta$'s value at the endpoint $(\eta(x),\eta(y))$ and  sum over all possible values of the divalent vertices, by the doubly-stochastic property of $P_N$, we obtain the factor of $p_{n_e}(\eta(x),\eta(y))$.

 This thus completes the proof. 
\end{proof}

\begin{example}[\textbf{One-boundary  loop diagram}]
\label{example:degenerate}
By \eqref{equation:W_Gamma} and \eqref{equ:eigenvaluesa_i}, we see the  diagram function of $\Gamma_0$ is

\begin{figure}[ht]
\centering
\begin{tikzpicture}
    \draw[thick,red] (0,0) circle (1cm);
    \fill[blue] (1,0) circle (2pt);
\end{tikzpicture}
     \label{2.001}
   \caption{One-boundary loop diagram }
\end{figure}
\begin{equation}
  \mathcal{W}_{\Gamma_0,X}(k)
=\rho_a^{-k}\sum_{n \ge 1}\binom{k}{\frac{k-n}{2}}\tr A_N^n=\rho_a^{-k}\sum_{i=-r_-}^{r_+}\sum_{n \ge 1}\binom{k}{\frac{k-n}{2}}a_i^n.
\end{equation}
A straightforward combinatorial calculation shows that, for any $0 \leq b \leq a$,
\begin{equation}
    \lim_{k \to \infty}\rho_a^{-k}\sum_{n \ge 1}\binom{k}{\frac{k-n}{2}}b^n=\begin{cases}
        1, & \ \  b=a,\\
        0, & \ \ b<a.
    \end{cases}
\end{equation}
Therefore,   as $k \to \infty$ the diagram function $\mathcal{W}_{\Gamma_0,X}(k)$ converges to the multiplicity of the top eigenvalue $a$ of $A$.

\end{example}

\subsection{Graph expansions for sub-Gaussian IRM}
We now investigate   the ribbon graph expansion in the sub-Gaussian case. Since a direct application of the Wick formula is not feasible, we adopt an alternative approach: comparing the moments to those of the deformed real Gaussian IRM. This yields an analogous expansion, but with a more complex structure due to the presence of non-trivial coupling factors.

    For the IRM matrix $X$  in Definition \ref{def:inhomo} and for any ribbon graph $\Upsilon$, we define the diagram function as follows:
    
\begin{enumerate}  
\item[(1)] With   the Wigner matrix  $W_N=(W_{xy})$   in Definition \ref{def:inhomo}, for  any  $1 \le x \neq  y \le N$ and any $n_1,n_2 \ge 1$ with   $n_1+n_2$ even, we define 
\begin{equation}\label{equ:c-xy}
 c_{xy}(n_1,n_2)=\frac{\mathbb{E}[W_{xy}^{n_1}W_{yx}^{n_2}]}{\mathbb{E}[|\mathcal{N}(0,1)|^{n_1+n_2}]},
\end{equation}

while 
for  $1 \le x \le N$  and an even integer   $n \ge 1$,
 \begin{equation}
     c_{x}(n)=\frac{\mathbb{E}[W_{xx}^{n}]}{\mathbb{E}[|\mathcal{N}(0,2)|^{n}]}.
 \end{equation}
 \item[(2)] For a map  $\eta:\mathcal{V}(\Upsilon) \to [N]$,   write $n_{\vec{xy}}(\eta)$ for the number of directed interior edges in $\Upsilon$ 
 such that the values of $\eta$ on  starting and ending  points are  exactly $x$ and  $y$ respectively. Furthermore,  put  $n_x(\eta) := n_{\vec{xx}}(\eta)$ for short.

 \item[(3)] Define a \textbf{coupling factor } \begin{equation}\label{def:c(U,eta)}
     c(\Upsilon,\eta):=\prod_{1 \le x < y \le N}c_{xy}(n_{\vec{xy}}(\eta),n_{\vec{yx}}(\eta))\prod_{1 \le x \le N}c_{x}(n_{x}(\eta)),
 \end{equation}
 which measures  the difference between  the sub-Gaussian $H$ and Gaussian IRM ensembles  with the same variance profile. 
 \item[(4)] 
 For the sub-Gaussian matrix $X$   in Definition \ref{def:inhomo} and  for any diagram  
 $\Gamma$, define    {\bf diagram function} 
\begin{equation}\label{equ:diagram_function_wigner}
    \mathcal{W}_{\Gamma, X}(k_1,\dots,k_s):=\rho_a^{-k}\sum_{\Upsilon:\Phi(\Upsilon)=\Gamma}\sum_{\eta:\mathcal{V}(\Upsilon)\to [N]}c(\Upsilon,\eta)
\prod_{(x,y) \in \mathcal{E}_{\rm int}(\Upsilon)}\sigma_{\eta(x)\eta(y)}^2 \prod_{(z,w) \in \mathcal{E}_b(\Upsilon)}A_{\eta(z)\eta(w)}.
\end{equation}
This is consistent with   Definition~\ref{diagramf} whenever  $X$ is a deformed real Gaussian IRM  discussed  in Section \ref{sec:ribbon_expansion}, since in that case $c(\Upsilon,\eta)=1$.

\end{enumerate}

\begin{proposition}\label{prop:moments-iid}
For $\beta=1,2$, let $X=X_N$ be the symmetric/Hermitian IRM  in Definition \ref{def:inhomo}.   With the diagram function $\mathcal{W}_{\Gamma,X}$  defined in \eqref{equ:diagram_function_wigner}, we have
    \begin{equation}\label{equ:wick expansion subgaussian iid}
  \rho_{a}^{-(k_1+\cdots +k_s)}\mathbb{E}\big[\prod_{j=1}^{s}\tr X^{k_j}\big]= \sum_{\Gamma \in \mathcal{D}_{\beta,s}} \mathcal{W}_{\Gamma,X}(k_1,\dots,k_s),
\end{equation}
where the sum runs  over all diagrams defined in Definition  \ref{def:diagram}.
\end{proposition} 
\begin{proof}[Proof]

Set   $k=k_1+\cdots +k_s$ and $\gamma=(12\cdots k_1)(k_1+1\cdots k_1+k_2)\cdots(k_1+\cdots +k_{s-1}\cdots k) \in \mathcal{S}_k$.
A direct expansion shows
\begin{equation}
    \begin{aligned}
        \mathbb{E}\Big[\prod_{j=1}^{s}\tr X^{k_j}\Big]&=\sum_{\eta:[k]\to[N]}\mathbb{E}\Big[\prod_{j=1}^{k}\big(H_{\eta(j),\eta(\gamma(j))}+A_{\eta(j),\eta(\gamma(j))}\big)\Big]\\
        &=\sum_{\eta:[k]\to[N]}\sum_{\text{J} \subset [k]} \mathbb{E}\Big[\prod_{j \in {\rm J}}H_{\eta(j),\eta(\gamma(j))}\Big]\prod_{j \in {\rm J^c}}A_{\eta(j),\eta(\gamma(j))}.
    \end{aligned}
\end{equation}

Let $H^{\rm Gaussian}:=\Sigma_N\circ W_{{\rm GOE}_N}$ be a real Gaussian IRM with the same variance profile as $H_N$. By the definition of $c(\Upsilon,\eta)$, we have
\begin{equation}
\begin{aligned}
&\mathbb{E}\big[\prod_{j \in {\rm J}}H_{\eta(j),\eta(\gamma(j))}\big]=c(\Upsilon,\eta)\mathbb{E}\big[\prod_{j \in {\rm J}}H^{\rm Gaussian}_{\eta(j),\eta(\gamma(j))}\big] \\
&=c(\Upsilon,\eta) \sum_{\pi \in \mathcal{P}_2(\text{J})} \prod_{(s,t) \in \pi}\mathbb{E}\big[H_{\eta(s),\eta(\gamma(s))}H_{\eta(t),\eta(\gamma(t))}\big]\\
&
=c(\Upsilon,\eta)\sum_{\pi \in \mathcal{P}_2(\text{J})} \prod_{(s,t) \in \pi} \sigma_{\eta(s),\eta(\gamma(s))}^2\big(\delta_{\eta(s),\eta(\gamma(t))} \delta_{\eta(t),\eta(\gamma(s))}+\delta_{\eta(s),\eta(t)} \delta_{\eta(\gamma(t)),\eta(\gamma(s))}\big).
\end{aligned}
\end{equation}
Note that if $|J|$ is odd, we have $\mathbb{E}\left[\prod_{j \in {\rm J}}H_{\eta(j),\eta(\gamma(j))}\right]=0$. 

Analogous to the argument in in Proposition \ref{prop:ribbon_graph_expansion}, gluing  the polygons $D_1,\dots,D_s$ to ribbon graphs and observing that the coupling factor $c(\Upsilon,\eta)$ does not affect the gluing procedure, we can use ribbon graph enumeration  to rewrite  
\begin{equation}
  \mathbb{E}[\prod_{j=1}^{s}\tr X^{k_j}]=\sum_{\Upsilon}\sum_{\eta:\mathcal{V}(\Upsilon) \to [N]}c(\Upsilon,\eta)\prod_{(x,y) \in \mathcal{E}_{\rm int}(\Upsilon)}\sigma_{\eta(x)\eta(y)}^2 \prod_{(z,w) \in \mathcal{E}_{b}(\Upsilon)}A_{\eta(z)\eta(w)}.
\end{equation}
Substituting \eqref{equ:diagram_function_wigner} into the above equation  completes  the proof of Proposition \ref{prop:moments-iid}. 
\end{proof}

For technical reasons, in what follows, we mainly consider the following mixed product. Recall $\mathcal{D}_{\beta,s}^{b}$ is the set of  all diagrams with $s$ faces such that each face contains at least one boundary edge.

\begin{proposition}\label{prop:diagram_expansion_subGaussian_boundary}Let $X=X_N$ be the symmetric/Hermitian IRM  in Definition \ref{def:inhomo}.   With the diagram function $\mathcal{W}_{\Gamma,X}$  defined in \eqref{equ:diagram_function_wigner}, we have
    \begin{equation}\label{equ:expansion_boundary}        \rho_{a}^{-(k_1+\cdots +k_s)}\mathbb{E}\big[\prod_{j=1}^{s}\big(\tr X^{k_j}-\tr H^{k_j}\big)\big]=\sum_{\Gamma \in \mathcal{D}_{\beta,s}^{b}} \mathcal{W}_{\Gamma,X}(k_1,\dots,k_s).
    \end{equation}
\end{proposition}
\begin{proof}
    Recalling the general diagram expansion established in Proposition \ref{prop:moments-iid},  we see that the total moment $\mathbb{E}[\prod_{j=1}^s \tr X^{k_j}]$ is expressed as a sum over all diagrams $\Gamma \in \mathcal{D}_{\beta,s}$.
    
    Observe that for each $j \in [s]$, the trace $\tr X^{k_j}$ expands into a sum over paths on the $j$-th face involving both $H$ (interior edges) and $A$ (boundary edges). The term $\tr H^{k_j}$ corresponds precisely to the subset of these paths containing \emph{only} $H$ entries (i.e., faces with no boundary edges, $E_{b,j} = \emptyset$).
    
    Consequently, the difference $\tr X^{k_j} - \tr H^{k_j}$ restricts the expansion to configurations where the $j$-th face contains at least one $A$ factor (i.e., $E_{b,j} \neq \emptyset$).
 Taking the product over all $j=1,\dots,s$ implies that we sum only over diagrams where \emph{every} face contains at least one boundary edge. By definition, this is exactly the set $\mathcal{D}_{\beta,s}^{b}$.
\end{proof}

\subsection{Trivalent and typical  diagrams}

 This subsection examines topological properties of diagrams and introduces a special class of diagrams, termed typical diagrams. 
 The following graph inequality connects the number of boundary edges, interior vertices and faces, and will be  of central importance in Section \ref{Section:GUE/GOE case}.

\begin{lemma}\label{Lemma-ineq}
    For any  diagram $\Gamma\in \mathcal{D}_{\beta,s}^{b}$, we have \begin{equation}\label{ieq1.1}
        |E_b(\Gamma)|+3|V_{\mathrm{int}}(\Gamma)| \le 2|E_{\mathrm{int}}(\Gamma)|+s.
    \end{equation}
    Moreover,  the equality holds if and only if $\Gamma$ is trivalent and the marked points are distinct.
\end{lemma}
\begin{proof}
Since $E_{b,j} \neq \emptyset$ for each $j\in [s]$, we know that $\Gamma$ contains no trivial connected component (the single point). Recalling that each unmarked vertex in $\Gamma$ has degree   at least 3,  each marked point has degree  at least 2 (here any loop incident to the vertex is counted
twice), we obtain 
\begin{equation}\label{iequ1.2}
    2|E_b(\Gamma)|+2|E_{\mathrm{int}}(\Gamma)|=\sum_{v \in V(\Gamma)}{\rm deg}_{\Gamma}(v) \ge 3|V(\Gamma)|-s=3|V_b(\Gamma)|+3|V_{\mathrm{int}}(\Gamma)|-s.
\end{equation}

We claim    that the boundary of $\Gamma$ satisfies
\begin{equation}\label{eq1.3}
    |E_b(\Gamma)|= |V_b(\Gamma)|.
\end{equation}
This can be shown from a topological point of view: the boundary of $\Gamma$, as the boundary of a 2-dimensional  compact manifold, is a 1-dimensional closed manifold, hence is a disjoint union of some circles. Thus we obtain \eqref{eq1.3}. 

Combining \eqref{eq1.3} with \eqref{iequ1.2} yields 
the inequality \eqref{ieq1.1}. Moreover,  the equality holds if and only if the equality in \eqref{iequ1.2} holds, and hence    if and  only if $\Gamma$ is trivalent, and marked points of $\Gamma$ are distinct.
\end{proof} 

Recalling that $t(\Gamma)$ is the number of boundary components and $g(\Gamma)$ is defined in \eqref{equ:genus}, the following lemma provides explicit formulas for the numbers of vertices and edges of any trivalent diagram in terms of 
$g(\Gamma)$ and $t(\Gamma)$.

\begin{lemma}\label{lem:trivalent_graph} 
For any trivalent diagram $\Gamma\in \mathcal{D}_{\beta,s}^{=3,b}$ with $g(\Gamma)=g$ and $t(\Gamma)=t$, we have

\begin{equation}\label{ve}
|V(\Gamma)|=2g+2t+3s-4,\quad |E(\Gamma)|=3g + 3t + 4s- 6.
\end{equation}
Setting $l=|V_{\mathrm{int}}({\Gamma})|$, we have
\begin{equation}\label{vbeb}
|V_b({\Gamma})|=|E_b({\Gamma})|=2g+2t+3s-4-l,\quad |E_{\mathrm{int}}({\Gamma})|=g+t+s+l-2.
\end{equation}
\end{lemma}
\begin{proof}Since each unmarked point in $V(\Gamma)$ is trivalent, and all marked points are distinct and divalent, we see that
\begin{equation}\label{equ:ve_1}
    2|E(\Gamma)|=\sum_{v \in V(\Gamma)}{\rm deg}_{\Gamma}(v)=3(|V(\Gamma)|-s)+2s=3|V(\Gamma)|-s.
\end{equation}
Recalling $g(\Gamma)$ defined in \eqref{equ:genus}, we have
\begin{equation}\label{equ:ve_2}
    |V(\Gamma)|-|E(\Gamma)|+s=2-g(\Gamma)-t(\Gamma).
\end{equation}
So \eqref{ve} directly follows from \eqref{equ:ve_1} and \eqref{equ:ve_2}. \end{proof}

We now introduce a key family of diagrams, referred to as \textbf{typical diagrams}, which capture the fluctuation behavior of outliers.
\begin{definition}\label{def:typical}
A diagram $\Gamma$ is called \textbf{typical} if it is a trivalent graph with no interior vertices (i.e., $V_{\text{int}}(\Gamma) = \emptyset$) and every marked point is a boundary vertex. Otherwise,    it is said to be non-typical.
\end{definition}

By Lemma \ref{Lemma-ineq}, a diagram $\Gamma$ is typical if and only if
\begin{equation}\label{equ:ineq}
    |E_b(\Gamma)|= 2|E_{\mathrm{int}}(\Gamma)|+s.
\end{equation} This equivalent characterization of typical diagrams will be used repeatedly in Sections \ref{Section:GUE/GOE case} and \ref{sec:fluctuation Gaussian}.

\begin{figure}[ht]
\centering
   \includegraphics[scale=0.18]{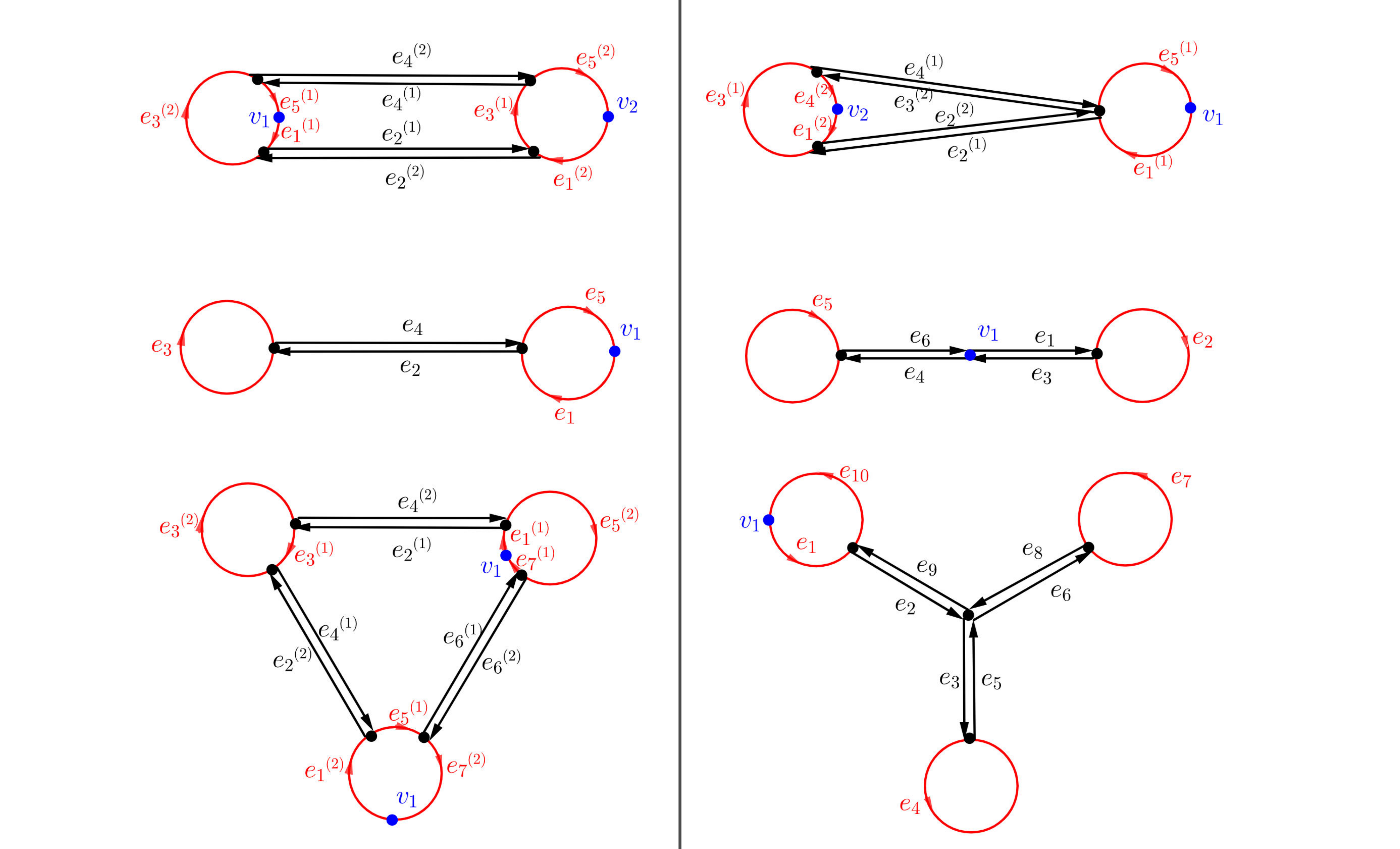}
     \caption{Typical (left) and non-typical (right) diagrams where the boundary edges are colored red.} 
     \label{fig:typical_diagram}
\end{figure}

\section{Upper bound estimates}\label{Section:GUE/GOE case}

The goal of this section is to establish term-wise   upper bounds for  diagram functions, which we call   dominating functions. We begin by introducing    these   functions and then present  the main results in this section.

\begin{definition}
   
Given parameters $s \ge 1,\xi>0$ and    $\lambda,a >1$, the dominating function of the variables $g\ge 0$ and $t\geq 1$ is defined as
   \begin{equation}
   \mathcal{D}_{g,t,s}(\xi;a,\lambda):=(2\xi s)^{2g+2t+2s-4}\Big(\frac{2\lambda  a^2}{a^2-1}\Big)^{3g+3t+4s-6} \frac{2048^{g+t+2s}}{(g+t)!}.\end{equation}
        
\end{definition}

\begin{theorem}\label{thm:dominating_function} 
With $X_N$ given in Definition  \ref{def:inhomo} and  with $\mathcal{W}_{\Gamma}$   in \eqref{equ:diagram_function_wigner},
assume that 
 \begin{equation}\label{equ:rkxi}
  \sum_{j=1}^s k_j \le \frac{\xi}{\sqrt{\gamma} (r+1)\sigma_N^*}.  
\end{equation} If $\xi \le \frac{1}{2}\gamma^{-1/2} ((r+1)\sigma_N^*)^{-1}$,  then for any  integer  $g\ge 0$ and $t\ge 1$ we have 
\begin{equation}\label{equ:dominating}
    \sum_{ \Gamma \in \mathcal{D}_{\beta,s}^{b}:  g(\Gamma)=g,t(\Gamma)=t}|\mathcal{W}_{\Gamma,X_N}(k_1,\cdots,k_s)|\le   (r+1)^{s} e^{2\xi^2}\mathcal{D}_{g,t,s}(\xi;a,1).
\end{equation}    
Furthermore, there exist positive constants $C_1(s,a)$ and $C_2(s,a)$    depending  on $s$ and $a$ such that 
\begin{equation}\label{equ:sum_D}
    \sum_{g \ge 0}\sum_{t \ge 1}\mathcal{D}_{g,t,s}(\xi;a,1) \le C_1(s,a)\, e^{C_{2}(s,a)
    \xi^2}.
\end{equation}

\end{theorem}

Theorem \ref{thm:dominating_function} provides a dominating function for the diagram functions with parameters $g$ and $t$. Since these dominating functions are summable, studying the asymptotic behavior of the full diagram sum reduces to analyzing the limit of each individual diagram function.

The  underlying idea of the proof  of Theorem  \ref{thm:dominating_function} is that a GOE matrix of a smaller size—chosen to match the sparsity level—can dominate a sub-Gaussian IRM. For this, we need to  establish   upper bounds for a rank-1 deformation of a GOE model, defined as follows:
\begin{equation}\label{equ:Y_M}
    Y_M=H_{{\rm GOE}_M}+aE_{11}=\frac{1}{\sqrt{M}}W_{{\rm GOE}_M}+aE_{11},
\end{equation}
     and to bound the diagram functions of $X_N$ by those of $Y_M$.

\begin{theorem}\label{Prop:Upper bound}
Let    $Y_M$ be given  in 
 \eqref{equ:Y_M}. Let  $k=\sum_{j=1}^s k_j$ and $\xi=\frac{k}{\sqrt{M}}$. Then the following holds.
 \begin{enumerate}
     \item[(1)]  For any $\xi>0$, we have
 \begin{equation}\label{equ:upper-bounds-GOE-1}
        \sum_{ \Gamma \in \mathcal{D}_{\beta,s}^{b}:  g(\Gamma)=g,t(\Gamma)=t}\lambda^{|E(\Gamma)|}\mathcal{W}_{\Gamma,Y_M}(k_1,\cdots,k_s)\le   \mathcal{D}_{g,t,s}(\xi;a,\lambda),
 \end{equation}
 and 
\begin{equation}\label{equ:upper-bounds-GOE-2}
    \sum_{g \ge 0}\sum_{t \ge 1}\mathcal{D}_{g,t,s}(\xi;a,\lambda) \le C_1(s,a,\lambda)\, e^{C_{2}(s,a,\lambda)
    \xi^2}
\end{equation}
for some positive constants $C_1(s,a,\lambda),C_2(s,a,\lambda)$ depending  only on $s,a$ and $\lambda$. 
\item[(2)] If $\xi=O(1)$, we have for any non-typical diagram $\Gamma\in \mathcal{D}_{\beta,s}^{b}$,  
\begin{equation}\label{equ:upperboundnontypical}
    \lim_{M\rightarrow\infty}\mathcal{W}_{\Gamma,Y_M}(k_1,\cdots,k_s)=0.
\end{equation}
 \end{enumerate}

\end{theorem}

\begin{theorem}\label{prop:gauss_irm} 

With $X_N$ given in Definition  \ref{def:inhomo} and $\mathcal{W}_{\Gamma, X_N}$ in \eqref{equ:diagram_function_wigner}, assume that the integer $M>0$   satisfies  $M\gamma\big((r+1)\sigma_{N}^*\big)^2\le 1$ and $k=k_1+\cdots+k_s\le  M/2$. Then for any diagram $\Gamma \in \mathcal{D}_{\beta,s}$, 
\begin{equation}\label{upper_bounds_IRM_b}
    \left|\mathcal{W}_{\Gamma, X_N}(k_1,\dots,k_s)\right|\le (r+1)^{s}e^{\frac{2k^2}{M}}\Big(\frac{N}{M}\Big)^{|\mathcal{C}(\Gamma)|}\mathcal{W}_{\Gamma, Y_M}(k_1,\dots,k_s),
\end{equation}
where   $|\mathcal{C}(\Gamma)|$ denotes  the number of connected components of $\Gamma$ without boundary edge.
\end{theorem}

Theorem \ref{thm:dominating_function}  follows directly from Theorem \ref{Prop:Upper bound} and Theorem \ref{prop:gauss_irm}. 

\begin{proof}[Proof of Theorem \ref{thm:dominating_function}] 
Taking $M=\lfloor \gamma^{-1}((r+1)\sigma_N^*)^{-2} \rfloor$, we see from \eqref{equ:rkxi} $\sum_{j=1}^s k_j \le \xi \sqrt{M}$. Since $\xi \le \frac{1}{2}\gamma^{-1/2} ((r+1)\sigma_N^*)^{-1} \le \sqrt{M}/2 $, we have $\sum_{j=1}^{s}k_j \le   M/2$. Therefore, the assumptions of Theorem \ref{Prop:Upper bound} and Theorem \ref{prop:gauss_irm} are satisfied. For any diagram $\Gamma \in \mathcal{D}_{\beta,s}^{b}$, $|\mathcal{C}(\Gamma)|=0$ and  Theorem \ref{prop:gauss_irm} yields the bound 
    \begin{equation}
            \left|\mathcal{W}_{\Gamma, X_N}(k_1,\dots,k_s)\right|\le (r+1)^{s} e^{\frac{2k^2}{M}}\mathcal{W}_{\Gamma, Y_M}(k_1,\dots,k_s).
    \end{equation}
    Theorem \ref{Prop:Upper bound} immediately implies  inequality \eqref{equ:dominating}  of Theorem \ref{thm:dominating_function}. The   second  inequality  of Theorem \ref{thm:dominating_function}  coincides with  that   of Theorem \ref{Prop:Upper bound}.
\end{proof}

The main task in this section is to prove  Theorems \ref{Prop:Upper bound} and \ref{prop:gauss_irm}. These proofs will be given   in Sections \ref{section:upper_bound_nontypical} and \ref{Section:Inhomogeneous}, respectively. Finally in Section \ref{sec:moments} we deduce the upper bound estimates for the moments of $X_N$ from Theorem \ref{thm:dominating_function}.

\subsection{Dominating functions for 
GOE}\label{section:upper_bound_nontypical}

In order to  prove Theorem \ref{Prop:Upper bound} we first obtain a diagram-wise upper bound. 

\begin{proposition}\label{prop:W+upper_bound}
For any diagram $\Gamma \in \mathcal{D}_{\beta,s}^{b}$  and for any positive  integers $k_1,\dots,k_s$, 
 we have
\begin{equation} \label{notation}
    \mathcal{W}_{\Gamma,Y_M}(k_1,\cdots,k_s)\le \big(\frac{a^2}{a^2-1}\big)^{|E(\Gamma)|}\Big(\frac{1}{M}\Big)^{|E_{\mathrm{int}}|-|V_{\mathrm{int}}|}  \prod_{j=1}^s\frac{k_j^{|E_{b,j}|-1}}{(|E_{b,j}|-1)!}.  
\end{equation}

\end{proposition}
\begin{proof}
Recall the expression of  $\mathcal{W}_{\Gamma,Y_M}(k_1,\dots,k_s)$ in \eqref{equation:W_Gamma}. Since $E_{b,j} \neq \emptyset$ for each $j\in [s]$, we know that $\Gamma$ contains no trivial connected component (i.e., the single point). Moreover, the transition probabilities of the associated random walk are all equal to ${1}/{M}$. Consequently,  we get
\begin{equation}\label{equ:3.17}
\begin{aligned}
\mathcal{W}_{\Gamma,Y_M}&(k_1,\dots,k_s)
=  \rho_a^{-\sum_{j=1}^{s}k_j}\sum_{\eta:V(\Gamma) \to [M],\eta|_{V_b}=1}\sum_{n_e\ge 1}\prod_{j=1}^{s}\binom{k_j}{\frac{1}{2}\big(k_j-\sum_{e\in \partial D_j}n_e\big)}\prod_{(x,y) \in E_{\mathrm{int}}}\frac{1}{M }
\prod_{(z,w) \in E_b}a^{n_e}\\
= & \left(\sum_{\eta:V(\Gamma) \to [M],\eta|_{V_b}=1}\prod_{(x,y) \in E_{\mathrm{int}}}\frac{1}{M }\right)\left(\sum_{n_e\ge 1}a^{\sum_{e \in E_{b}}n_e}\prod_{j=1}^{s}\binom{k_j}{ \frac{1}{2}\big(k_j-\sum_{e\in \partial D_j}n_e\big)}\right)\rho_a^{-\sum_{j=1}^{s}k_j}
\\
=& \big(\frac{1}{M}\big)^{|E_{\mathrm{int}}|-|V_{\mathrm{int}}|} \mathcal{I}_{\Gamma}, 
\end{aligned}
\end{equation} where in the last equality  $\eta$  has been evaluated at every interior vertex from $[M]$ and  
\begin{equation}
    \quad \mathcal{I}_{\Gamma}:=\sum_{n_e\ge 1,e\in E(\Gamma)}a^{\sum_{e \in E_{b}}n_e}  \prod_{j=1}^{s}\left(\binom{k_j}{ \frac{1}{2}\big(k_j-\sum_{e\in \partial D_j}n_e\big)}\rho_a^{-k_j}\right).
\end{equation}

Now we bound the above  summation. Using $\sum_{e\in E_b}+2\sum_{e\in E_{\mathrm{int}}}n_e=\sum_{j}\sum_{e\in \partial D_j}n_e$, we have
\begin{equation}
    \label{equ:3.14}
\begin{aligned}
       \mathcal{I}_{\Gamma} &=\sum_{n_e\ge 1,e\in E(\Gamma)}a^{-2\sum_{e \in E_{\mathrm{int}}}n_e}  \prod_{j=1}^{s}\left(\binom{k_j}{ \frac{k_j-\sum_{e\in \partial D_j}n_e}{2}}a^{\sum_{e\in \partial D_j}n_e}\rho_a^{-k_j}\right).
\end{aligned} \end{equation}
Let us fix $n_e$ for all $e\in E_{\mathrm{int}}$ and the summation $n_{j}=\sum_{e \in \partial D_j} n_e$ for all $j$, and also let  $n_j^{\rm b}=\sum_{e \in E_{b,j}}n_e$. 
Then the number of possible choices for $n_e$ with $e\in E_{b,j}$ is at most

\begin{equation}\label{equ:num_solution}
\binom{n^{b}_{j}}{|E_{b,j}|-1}\le \frac{{n_j}^{|E_{b,j}|-1}}{(|E_{b,j}|-1)!}.
\end{equation}

Hence, partitioning all edges into  interior and boundary ones, we see from \eqref{equ:3.14} that 
\begin{equation}
\begin{aligned}
   \mathcal{I}_{\Gamma} &\le \Big(\sum_{n_e\ge 1,e\in E_{\mathrm{int}}(\Gamma)}a^{-2\sum_{e \in E_{\mathrm{int}}}n_e} \Big) \cdot\bigg( \sum_{n_j\ge 1,j=1,\ldots s}\prod_{j=1}^{s}\bigg(\binom{k_j}{ \frac{1}{2}(k_j-n_j)}a^{n_j}\rho_a^{-k_j}\bigg)\frac{{n_j}^{|E_{b,j}|-1}}{(|E_{b,j}|-1)!}\bigg)\\
\end{aligned}
\end{equation}
For the first sum,  we have
\begin{equation}\label{equ:a_int}
    \sum_{n_e\ge 1,e\in E_{\mathrm{int}}(\Gamma)}a^{-2\sum_{e \in E_{\mathrm{int}}}n_e} \le \Big(\frac{a^2}{a^2-1}\Big)^{|E_{\mathrm{int}}(\Gamma)|}.
\end{equation}
For the second sum, we use the binomial distribution  
\begin{equation}\label{equ:binom}
\mathbb{P}\left(\text{Binom}\big(k,\frac{a^2}{a^2+1}\big)=\frac{k+n}{2}\right)=\rho_{a}^{-k}a^{n}\binom{k}{\frac{k-n}{2}}
\end{equation}
 to obtain
\begin{align}\label{eq:upper,inner part}
&\rho_a^{-k_j}\sum_{n_j\ge 1}a^{n_j}\frac{n_j^{|E_{b,j}|-1}}{(|E_{b,j}|-1)!}\binom{k_j}{\frac{k_j-n_j}{2}}\le  \frac{k_j^{|E_{b,j}|-1}}{(|E_{b,j}|-1)!}\sum_{n_j\ge 0}\frac{a^{n_j}}{\rho_a^{k_j}}\binom{k_j}{\frac{k_j-n_j}{2}}\le\frac{k_j^{|E_{b,j}|-1}}{(|E_{b,j}|-1)!}.
\end{align}
Combining the above two estimates, we obtain   
\begin{equation}
    \mathcal{I}_{\Gamma}\le \Big(\frac{a^2}{a^2-1}\Big)^{|E_{\mathrm{int}}(\Gamma)|}\prod_{j=1}^s\frac{k_j^{|E_{b,j}|-1}}{(|E_{b,j}|-1)!}.
 \end{equation}
Since  $|E_{\mathrm{int}}|\le |E(\Gamma)|$, together with   \eqref{equ:3.17} 
we  arrive at
\begin{equation}
\begin{aligned}\label{Upper1}
      \mathcal{W}_{\Gamma,Y_M}(k_1,\cdots,k_s)\le & \big(\frac{a^2}{a^2-1}\big)^{|E(\Gamma)|}\Big(\frac{1}{M}\Big)^{|E_{\mathrm{int}}|-|V_{\mathrm{int}}|}  \prod_{j=1}^s\frac{k_j^{|E_{b,j}|-1}}{(|E_{b,j}|-1)!}.
\end{aligned}
\end{equation}
This thus completes the proof.
\end{proof}

Next, we regularize all diagrams to trivalent diagrams and derive an upper bound in terms of the latter. Recall that $\mathcal{D}_{\beta,s}^{=3,b}$ is the set of  all trivalent diagrams in $\mathcal{D}_{\beta,s}^{b}$.

\begin{proposition}\label{prop:reduce to 3} Set 
\begin{equation}\label{equ:tildeW}
    \widetilde{\mathcal{W}}_{\Gamma,Y_M}(k_1,\cdots,k_s):=\big(\frac{a^2}{a^2-1}\big)^{|E(\Gamma)|}\Big(\frac{1}{M}\Big)^{|E_{\mathrm{int}}|-|V_{\mathrm{int}}|}  \prod_{j=1}^s\frac{k_j^{|E_{b,j}|-1}}{(|E_{b,j}|-1)!}.
\end{equation}
Then for any $\lambda\ge 1$, we have  
    \begin{equation}
                \sum_{\Gamma \in  \mathcal{D}_{\beta,s}^{b}}\lambda^{|E(\Gamma)|}\widetilde{\mathcal{W}}_{\Gamma,Y_M}(k_1,\cdots,k_s)\le 
        \sum_{\Gamma \in \mathcal{D}_{\beta,s}^{=3,b}}(2\lambda)^{|E(\Gamma)|}\widetilde{\mathcal{W}}_{\Gamma,Y_M}(k_1,\cdots,k_s).
    \end{equation}
\end{proposition}
\begin{proof}[Proof]
For any diagram $\Gamma$, it is  known that the neighborhood of each vertex is composed of several faces formed by the corners incident to that vertex. For vertices with degree greater than  $3$ and any marked point with degree $>2$, by using the local topological structure at this point, we add a weight one interior edge to split the vertices into those with  smaller degree. Namely, we construct a  function $\phi$ that maps each diagram $\Gamma \in  \mathcal{D}_{\beta,s}^{b}$ to a trivalent diagram $\overline{\Gamma} \in  \mathcal{D}_{\beta,s}^{=3,b}$ as shown in Figure \ref{fig:reduce}.

\begin{figure}[ht]
\centering
   \includegraphics[scale=0.5]{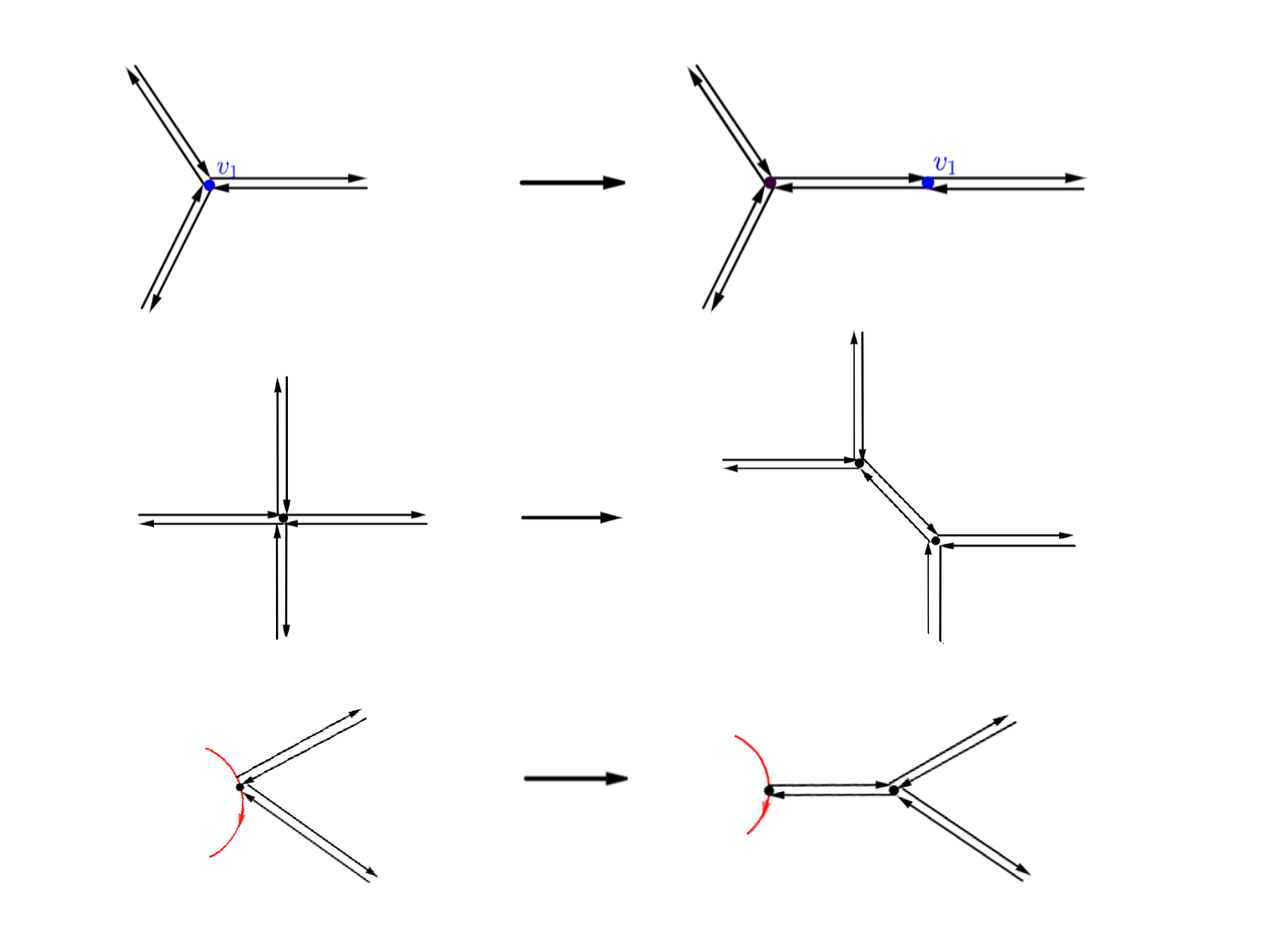}
\caption{Regularization to trivalent diagrams}
     \label{fig:reduce}
\end{figure}

Noting that both the quantities $|E_{b,j}|$ and  $|E_{\mathrm{int}}|-|V_{\mathrm{int}}|$ do not change and $|E(\Gamma)|$ is non-decreasing after regularization, we thus get  
\begin{equation}
    \widetilde{\mathcal{W}}_{\Gamma,Y_M}(k_1,\cdots,k_s)\le \widetilde{\mathcal{W}}_{\phi(\Gamma),Y_M}(k_1,\cdots,k_s).
\end{equation}
Moreover, since  all preimages in $\phi^{-1}(\Gamma)$ can be obtained by contracting some edges of $\Gamma$,  we have 
\begin{equation}
    |\phi^{-1}(\Gamma)|\le 2^{|E(\Gamma)|}.
\end{equation}
  So we have
\begin{equation}
\begin{aligned}
        \sum_{\Gamma\in \mathcal{D}_{\beta,s}^{b}}\lambda^{|E(\Gamma)|}\widetilde{\mathcal{W}}_{\Gamma,Y_M}(k_1,\cdots,k_s)&\le \sum_{\Gamma\in \mathcal{D}_{\beta,s}^{=3,b}}|\phi^{-1}(\Gamma)|\lambda^{|E(\Gamma)|}\widetilde{\mathcal{W}}_{\Gamma,Y_M}(k_1,\cdots,k_s)\\
        &\le \sum_{\Gamma\in \mathcal{D}_{\beta,s}^{=3,b}}(2\lambda)^{|E(\Gamma)|}\widetilde{\mathcal{W}}_{\Gamma,Y_M}(k_1,\cdots,k_s).
\end{aligned}
    \end{equation}
This completes the proof.
\end{proof}

Next,  we bound the number of trivalent diagrams $\Gamma\in \mathcal{D}_{\beta,s}^{=3,b}$ with $g(\Gamma)=g$ and $t(\Gamma)=t$.

\begin{proposition}\label{automaton}
For $\beta \in \{1,2\}$, for any $s \ge 1$ and $g,t \ge 0$, the number of trivalent diagrams $\Gamma\in \mathcal{D}_{\beta,s}^{=3,b}$ with $g(\Gamma)=g$ and $t(\Gamma)=t$ is at most
\begin{equation}\frac{128^{g+t+2s}(g+t+2s)^{g+t+3s-3}}{(s-1)!}.
\end{equation}
\end{proposition}
The proof follows the argument presented in \cite[Proposition II.3.3]{feldheim2010universality}, which constructs an automaton to generate all trivalent diagrams with $s$ faces on closed topological surfaces \cite[Remark II.6]{feldheim2010universality}. The automaton starts from the initial vertex of the diagram and proceeds along its edges, constructing the diagram step by step using four types of transition. For $\beta=2$, at each step, it either adds a loop (a \textbf{creation}) or removes a loop (an \textbf{annihilation}), continuing until the entire diagram is generated.   For $\beta=1$, another type of annihilation and a 
``creation and annihilation'' may take place at each step. This procedure can be viewed as the evolution of a line together with several loops; see \cite[Figure 2-5]{feldheim2010universality} for a detailed description.
By construction, the same automaton can generate all trivalent diagrams $\Gamma \in \mathcal{D}_{\beta,s}^{=3,b}$ with $t$ boundary components by additionally requiring that exactly $t$ circles remain unannihilated at the end. An example is shown in Figure~\ref{fig:automaton}.

\begin{figure}[ht]
\centering
   \includegraphics[scale=0.6]{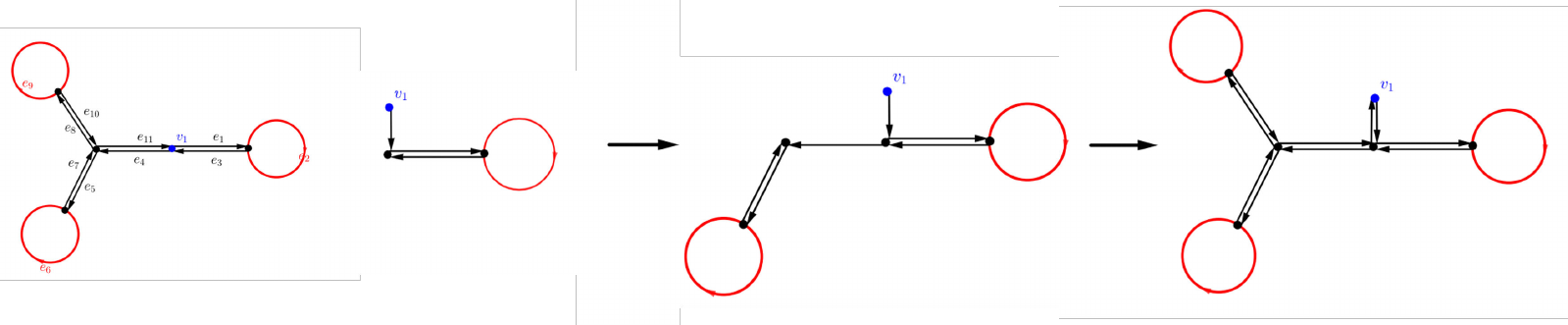}
\caption{An example of Feldheim-Sodin automaton} 
     \label{fig:automaton}
\end{figure}

\begin{proof}[Proof of Proposition \ref{automaton}]
We begin by attaching an extra interior tail edge to each marked point of the diagram. Now we use the automaton method to bound the number of  diagrams in $\mathcal{D}_{\beta,s}^{=3,b}$ with $g(\Gamma)=g$ and $t(\Gamma)=t$. Since the first (creation) step  creates 3 edges and 3
vertices, the last (annihilation) step creates 2 edges and one vertex, and every other
step creates 3 edges and 2 vertices. We see the total number of steps in the automaton is $T=g+t+2s-2$.  

First, we determine the return times $T_1, T_2, \dots, T_s$ for the automaton process starting from each of the $s$ marked points. Specifically, the automaton begins at the marked point of the first face $D_1$ and returns to it after $T_1$ steps; it then starts from the marked point of $D_2$ and returns after $T_2$ steps, and so on for all faces.
Since the total number of steps in the automaton is $T=g+t+2s-2$, the return times satisfy the constraint $T_1+\cdots+T_s= g+t+2s-2$. Hence, the number of choices for $T_1,T_2,\dots,T_s$ is at most the number of positive integer solutions to this equation, which is 
\begin{equation}\label{eq:T_choices}
       \frac{(g+t+2s-2)^{s-1}}{(s-1)!}\le \frac{(g+t+2s)^{s-1}}{(s-1)!}.
\end{equation}
 
After fixing these return times, we bound the number of possible sequences of transitions (of the four types) in the automaton. Since each step involves one of the four transition types, the total number of such sequences is at most  \begin{equation}\label{eq:transition_order}
       4^{g+t+2s-2} \le 4^{g+t+2s}.
\end{equation}

Third, we consider how the total length of the graph evolves during the automaton.  Define  $\ell_i$ as the number of unpaired edges in the diagram after each step of the automaton (which is the total length of the segment and the cycles in \cite{feldheim2010universality}), and set $m_i = 2 - (\ell_i-\ell_{i-1})$. It is easy to check that $m_i \ge 0$ and  $\sum_{i=1}^{T} m_i = 2g+2t+4s-4-|E_b(\Gamma)|$. Hence  the number of ways to choose the sequence of numbers $m_i$ is at most
\begin{equation}\label{eq:m_choices}
     \binom{2g+2t+4s-4-|E_b(\Gamma)|+g+t+2s-2}{g+t+2s-3} \le 8^{g+t+2s}.
\end{equation}

Finally, for a fixed order of transitions and a fixed sequence of   $m_i$, the number of corresponding diagrams is at most
$|E(\Gamma)|^{g+t+2s-2}=(3g+3t+4s-6)^{g+t+2s-2}$
using relation \eqref{ve}. Consequently, we obtain the bound
\begin{equation}\label{eq:diagram_choices}
     (3g+3t+4s-6)^{g+t+2s-2} \le 4^{g+t+2s}(g+t+2s)^{g+t+2s-2}.
\end{equation}

The proof is completed  by multiplying the bounds from equations \eqref{eq:T_choices}, \eqref{eq:transition_order}, \eqref{eq:m_choices}, and \eqref{eq:diagram_choices}. 
\end{proof}

Now we are ready to prove Theorem \ref{Prop:Upper bound}.
\begin{proof}[Proof of Theorem   \ref{Prop:Upper bound}]
We first prove \eqref{equ:upperboundnontypical}. By Proposition \ref{prop:W+upper_bound}, we have
\begin{equation}
    |\mathcal{W}_{\Gamma,Y_M}(k_1,\cdots,k_s)|\le C_{\Gamma,a} {M}^{|V_{\mathrm{int}}|-|E_{\mathrm{int}}|}  \prod_{j=1}^s\frac{k_j^{|E_{b,j}|-1}}{(|E_{b,j}|-1)!}\le C_{\Gamma,a}M^{|V_{\mathrm{int}}|-|E_{\mathrm{int}}|+\frac{1}{2}(\sum_{j=1}^s|E_{b,j}|-1)}.
\end{equation}
By Lemma \ref{Lemma-ineq}, we have
\begin{equation}
    |V_{\mathrm{int}}|-|E_{\mathrm{int}}|+\frac{1}{2}{\big(\sum_{j=1}^s|E_{b,j}|-1\big)}=\frac{1}{2}(|E_b(\Gamma)|-s)+|V_{\mathrm{int}}|-|E_{\mathrm{int}}|\le -\frac{1}{2}|V_{\mathrm{int}}|\le 0,
\end{equation}
where the equality holds if and only if $\Gamma\in \mathcal{D}_{\beta,s}^{= 3,b}$. This completes the proof of part (2). 

We now prove \eqref{equ:upper-bounds-GOE-1} of part (1). For simplicity, set
\begin{equation}
    W_{g,t}:=\sum_{\Gamma \in \mathcal{D}_{\beta,s}^{b}:  g(\Gamma)=g,t(\Gamma)=t}\lambda^{|E(\Gamma)|}\mathcal{W}_{\Gamma,Y_M}(k_1,\cdots,k_s).
\end{equation}
By Lemma \ref{Lemma-ineq}, Proposition \ref{prop:W+upper_bound}, and \ref{prop:reduce to 3}, we know that
\begin{align}\label{Upper4}
    W_{g,t}\le & \sum_{\Gamma \in \mathcal{D}_{\beta,s}^{=3,b}:  g(\Gamma)=g,t(\Gamma)=t}(2\lambda)^{|E(\Gamma)|}\big|\mathcal{W}_{\Gamma,Y_M}(k_1,\cdots,k_s)\big|\nonumber \\
     \le   &   \sum_{l \ge 1} \sum_{\substack{\Gamma \in \mathcal{D}_{\beta,s}^{=3,b}:  g(\Gamma)=g,t(\Gamma)=t\\ |V_{\mathrm{int}}(\Gamma)|=l}}\Big(\frac{2\lambda a^2}{a^2-1}\Big)^{|E(\Gamma)|}\Big(\prod_{j=1}^{s}\frac{(k_j/\sqrt{M})^{|E_{b,j}|-1}}{(|E_{b,j}|-1)!}\Big){M}^{-\frac{l}{2}}.
\end{align}
An application of the multinomial theorem yields
\begin{equation}
    \frac{1}{\prod_{j=1}^{s}(|E_{b,j}(\Gamma)|-1)!}=\frac{(|E_b(\Gamma)|-s)!}{\prod_{j=1}^{s}(|E_{b,j}(\Gamma)|-1)!}\times \frac{1}{(|E_b(\Gamma)|-s)!} \le \frac{s^{|E_b(\Gamma)|-s}}{(|E_b(\Gamma)|-s)!}.
\end{equation}
Therefore,
\begin{equation}
\begin{aligned}
    W_{g,t}&\le \sum_{l \ge 1} \sum_{\substack{\Gamma \in \mathcal{D}_{\beta,s}^{=3,b}:  g(\Gamma)=g,t(\Gamma)=t\\ |V_{\mathrm{int}}(\Gamma)|=l}}\Big(\frac{2\lambda a^2}{a^2-1}\Big)^{|E(\Gamma)|}\frac{s^{|E_b(\Gamma)|-s}}{(|E_b(\Gamma)|-s)!}\Big(\frac{k}{\sqrt{M}}\Big)^{|E_b(\Gamma)|-s}{M}^{-\frac{l}{2}}.
\end{aligned}
\end{equation}
Denote the number of diagrams $\Gamma \in \mathcal{D}_{\beta,s}^{=3,b}$ with  $g(\Gamma)=g,t(\Gamma)=t$ by $D_{g,t}$.
By Lemma \ref{lem:trivalent_graph}, we know that $|E_b(\Gamma)|=2g+2t+3s-4-l$ and $|E(\Gamma)|=3g+3t+4s-6$. Hence 
\begin{equation}
\begin{aligned}
    W_{g,t}&\le D_{g,t}\sum_{l} \Big(\frac{2\lambda  a^2}{a^2-1}\Big)^{|E(\Gamma)|}\frac{(k s)^{2g+2t+2s-4-l}}{(2g+2t+2s-4-l)!}M^{-(g+t+s-2)}\\
    &=D_{g,t}\big(\frac{ks}{\sqrt{M}}\big)^{2g+2t+2s-4}\Big(\frac{2\lambda  a^2}{a^2-1}\Big)^{3g+3t+4s-6}\sum_{l} \frac{1}{(2g+2t+2s-4-l)!(ks)^{l}}.
\end{aligned}
\end{equation}
Here the sum $\sum_{l}$ has a natural restriction $l\le 2g+2t+2s-4$, because of $|E_b(\Gamma)| \ge s$. Now we use $ks\ge l$ and hence $(ks)^l\ge l!$,  to get 
\begin{equation}
    \sum_{l} \frac{1}{(2g+2t+2s-4-l)!(ks)^{l}}\le \frac{1}{(2g+2t+2s-4)!}\sum_{l} \binom{2g+2t+2s-4}{l}.
\end{equation}
So 
\begin{equation}
\begin{aligned}
     W_{g,t}&\le \big(\frac{2ks}{\sqrt{M}}\big)^{2g+2t+2s-4}\Big(\frac{2\lambda  a^2}{a^2-1}\Big)^{3g+3t+4s-6}\frac{D_{g,t}}{(2g+2t+2s-4)!}.\\     
\end{aligned}
\end{equation}

By Proposition \ref{automaton}, we have
\begin{equation}
    \begin{aligned}
\frac{D_{g,t}}{(2g+2t+2s-4)!}\le \frac{128^{g+t+2s}(g+t+2s)^{g+t+3s-3}}{(2g+2t+2s-4)!(s-1)!}.
    \end{aligned}
\end{equation}
A simple algebraic computation shows\begin{equation}
    \frac{(g+t+2s)^{g+t+3s-3}}{(2g+2t+2s-4)!(s-1)!}\le \frac{16^{g+t+2s}}{(g+t)!},
\end{equation}
from which  
\begin{equation}
    W_{g,t}\le \frac{2048^{g+t+2s}}{(g+t)!}\big(\frac{2ks}{\sqrt{M}}\big)^{2g+2t+2s-4}\Big(\frac{2\lambda  a^2}{a^2-1}\Big)^{3g+3t+4s-6}.
\end{equation}
Combining the condition $k \le \xi \sqrt{M}$, we arrive at the first inequality
\begin{equation}
    W_{g,t}\le \mathcal{D}_{g,t,s}(\xi;a,\lambda).
\end{equation}

Summing over $g\ge 0, t\ge 1$, we obtain part (1) of \eqref{equ:upper-bounds-GOE-2}. 

This finally  completes the proof of the theorem.
\end{proof}

\subsection{GOE dominates IRM}\label{Section:Inhomogeneous}

In this section, we prove Theorem \ref{prop:gauss_irm}, which establishes that an appropriately chosen GOE model $Y_M$ dominates the contributions from the deformed sub-Gaussian IRM $X_N$. We recall the ribbon graph expansion in the sub-Gaussian case (Proposition \ref{prop:moments-iid} and \ref{prop:diagram_expansion_subGaussian_boundary}) and the diagram function \eqref{equ:diagram_function_wigner}. The primary goal is to bound the ribbon graph contribution $|\mathcal{L}_{\Upsilon,X}|$ by its GOE counterpart, $\mathcal{L}_{\Upsilon,Y_M}$. Here 
\begin{equation}\label{equ:def_L}
 \mathcal{L}_{\Upsilon,X} := \rho_a^{-\sum_{j=1}^s k_j} \sum_{\eta:\mathcal{V}(\Upsilon)\to [N]} c(\Upsilon,\eta)
\prod_{(x,y)\in\mathcal{E}_{\mathrm{int}}(\Upsilon)} \sigma_{\eta(x)\eta(y)}^2
 \prod_{(z,w)\in\mathcal{E}_b(\Upsilon)} A_{\eta(z)\eta(w)}.
\end{equation}
Thus,  the total diagram function can be rewritten as    \begin{equation}
    \mathcal{W}_{\Gamma, X}(k_1,\dots,k_s) = \sum_{\Upsilon:\Phi(\Upsilon)=\Gamma} \mathcal{L}_{\Upsilon,X}.
\end{equation}  For simplicity, and without loss of generality, we assume $\{m_1,\ldots,m_r\}=[r]$  unless stated otherwise.

A significant technical challenge arises because the  {coupling factor} $c(\Upsilon,\eta)$ depends intricately on the labeling $\eta$, unlike the pure Gaussian case. This dependence prevents straightforward summation over boundary labels:
\begin{equation}
\sum_{x_1,\dots,x_{n-1}} c(\Upsilon,\eta) A_{ax_1}A_{x_1x_2}\dots A_{x_{n-1}b} \ne (A^{n})_{ab}, 
\end{equation} 
where $c(\Upsilon,\eta)=1$  in the Gaussian case but $c(\Upsilon,\eta)$ changes in non-Gaussian case.
To overcome this obstacle, we employ a multi-step strategy (see also Figure \ref{fig:proof_streamline}):
\begin{enumerate}
\item \textbf{Quotient Graph ($G_0$):} 
Merge interior vertices in $\Upsilon$ with the same label $\eta(v)$ to construct the \textit{quotient graph} $G_0$. The induced labeling $\bar{\eta}$ is injective on $\mathcal{V}_{\mathrm{int}}(G_0)$, simplifying the edge coupling structure. Let $\mathcal{K}_{G_0,X}$ be the weight that will appear below.
\item \textbf{Handle Constraints ($\mathfrak{C}$):} $\mathcal{K}_{G_0,X}$ sums over labels $\bar{\eta}$ satisfying constraints $\mathfrak{C}$, which include injectivity on $\mathcal{V}_{\mathrm{int}}(G_0)$. We use the {inclusion-exclusion principle} (Lemma \ref{lem:exclusion_inclusion}) to remove the injectivity constraint on the free vertices $\mathcal{V}_{\mathrm{free}}(G_0)$ to ensure that we can sum over trees on the boundary edges and then we can sum over boundary vertices.
 \item \textbf{Bounding the Sum:} After relaxing constraints, we perform the summation over Catalan trees and boundary vertex labels to derive the desired upper bounds by comparing with the GOE model.
\end{enumerate}

\begin{table}[htbp]
\centering
\caption{Summary of key notations used in Section \ref{Section:Inhomogeneous}}
\label{tab:notation_sec32}
\small
\begin{tabularx}{\textwidth}{@{} llX @{}}
\toprule
\textbf{Symbol} & \textbf{Name} & \textbf{Description} \\
\midrule
\multicolumn{3}{l}{\textit{Graphs and Maps}} \\
$\Upsilon, \Gamma$ & Ribbon graph / Diagram & The combinatorial structures representing terms in the expansion. \\
$G_0, G_0/\pi$ & Quotient graph & Graphs obtained by merging vertices based on labels or partitions. \\
$\Phi$ & Contraction map & The map $\Phi: \Upsilon \to \Gamma$ reducing a ribbon graph to its shape. \\
\midrule
\multicolumn{3}{l}{\textit{Vertex Sets}} \\
$\mathcal{V}_{\mathrm{int}}, \mathcal{V}_b$ & Interior / Boundary vertices &  Vertices incident to boundary edges $\mathcal{V}_b$,  strictly internal vertices $\mathcal{V}_{\mathrm{int}}$.\\
$\mathcal{V}_{\mathrm{tied}}, \mathcal{V}_{\mathrm{free}}$ & Tied / Free vertices & Interior vertices with labels $\eta(v) \in [r]$ (tied) or $\eta(v) \in [N]\setminus[r]$ (free). \\
$\tilde{\mathcal{V}}_b(G_0)$ & Reduced boundary set & Boundary vertices remaining after removing boundary trees. \\
\midrule
\multicolumn{3}{l}{\textit{Edge Sets}} \\
$\mathcal{E}_{\mathrm{int}}, \mathcal{E}_b$ & Interior / Boundary edges & Edges corresponding to matrix entries of $H_N$ and $A_N$ respectively. \\
$\mathcal{E}_{\mathrm{tied}}, \mathcal{E}_{\mathrm{free}}$ & Tied / Free edges & Definition \ref{def:tied_edge}. \\
\midrule
\multicolumn{3}{l}{\textit{Others}} \\
$c(\Upsilon, \eta)$ & Coupling factor & Defined in \eqref{eq:coupling_factor_def} \\
$\mathfrak{C}$& & Constraints on $\bar\eta$ \eqref{equ:eta_constraint}.\\
$\mathcal{G}(\Upsilon)$ & Set of quotient graphs &The collection of all possible quotient graphs $G_0$ derived from a ribbon graph $\Upsilon$.\\
$\mathcal{K}_{G_0, X}$ &  & The sum of weights over labelings $\eta$ consistent with the graph $G_0$. \eqref{G0expansion} \\
$\mathcal{F}_{G_0/\pi, X}$ &  & $\mathcal{K}_{G_0, X}$ after applying the inclusion-exclusion principle. \eqref{def:F}\\
$Y_M$ & Reference GOE model & $M \times M$ deformed GOE matrix used as a comparison benchmark. \\
\bottomrule
\end{tabularx}
\end{table}

\usetikzlibrary{shapes.geometric, arrows, positioning}

\tikzset{
    process/.style={
        rectangle, 
        minimum width=10cm, 
        minimum height=1cm, 
        text width=0.8\textwidth, 
        text centered, 
        draw=black, 
        fill=white,
        inner sep=10pt
    },
    arrow/.style={thick, ->, >=stealth}
}

\begin{figure}[htbp]
\centering 

\begin{tikzpicture}[node distance=4.5cm, auto]

\node (goal) [process] {
    \textbf{Goal: GOE Dominates IRM} \\ 
    Challenge: The coupling factor $c(\Upsilon, \eta)$ depends on the labeling of boundary vertices $V_b$. 
    $$\displaystyle \sum_{x_1,\dots,x_{n-1}} c(\Upsilon,\eta) A_{ax_1}A_{x_1x_2}\dots A_{x_{n-1}b} \neq (A^{n})_{ab}$$
};

\node (step1) [process, below of=goal] {
    \textbf{Step 1: Quotient Graph $G_0$} \\ 
    Fix the partition induced by $\eta$ to determine the ``coupling type" of edges. (Proposition \ref{prop:relation:K}) 
    $$\displaystyle \mathcal{L}_{\Upsilon,X} = \sum_{G_0 \in \mathcal{G}(\Upsilon)} \mathcal{K}_{G_0,X}$$
};

\node (step2) [process, below of=step1] {
    \textbf{Step 2: Relaxing the Injectivity of $\bar{\eta}$} \\ 
    Relax injectivity constraints on vertices $\mathcal{V}_{\mathrm{free}}$ without changing the ``coupling type" of edges, enabling summation over trees. (Lemma \ref{lem:exclusion_inclusion} and Corollary \ref{coro:upper_exclusion}) 
    $$\displaystyle \mathcal{K}_{G_0,X} = \rho_a^{-k} \sum_{\pi \in \mathcal{P}(\mathcal{V}_{\mathrm{free}})} C_{\pi} \cdot \mathcal{F}_{G_0/\pi,X}$$
};

\node (step3) [process, below of=step2] {
    \textbf{Step 3: Summing over Trees in $\mathcal{F}_{G_0/\pi,X}$} \\ 
    Sum over all Catalan trees in $G_0/\pi$ and the boundary vertices. 
    $$\displaystyle \left|\sum_{p_1,\dots,p_{k}} \prod_{i=0}^{k-1} A_{p_i p_{i+1}} (\Lambda_{i+1})_{p_{i}p_{i}} \right| \le a^k$$
    Derive the upper bound for $\mathcal{F}_{G_0/\pi,X}$. (Lemma \ref{lem:F_upper_bound})
};

\node (conclusion) [process, below of=step3] {
    \textbf{Result: Bound $X_N$ weights by $Y_M$ weights} \\
    Summing over all $G_0$ bounds $|\mathcal{K}_{G_0,X}|$ by $\mathcal{K}_{G_0,Y}$, thus bounding $\mathcal{W}_{\Gamma,X}$ by $\mathcal{W}_{\Gamma,Y}$. (Corollary \ref{coro:KX<KY} and Theorem \ref{prop:gauss_irm})
};

\draw [arrow] (goal) -- (step1);
\draw [arrow] (step1) -- (step2);
\draw [arrow] (step2) -- (step3);
\draw [arrow] (step3) -- (conclusion);

\end{tikzpicture}
\caption{Structure of proofs in Section \ref{Section:Inhomogeneous}}
\label{fig:proof_streamline}
\end{figure}

\subsubsection{Cluster   expansion of weighted diagram functions }
\paragraph{Quotient graph $G_0$ and quotient labeling $\bar\eta$.}

The labeling $\eta$ induces a partition on $\mathcal{V}_{\mathrm{int}}(\Upsilon)$ by identifying vertices $u, v$ whenever $\eta(u) = \eta(v)$. Merging vertices within each block (while keeping boundary vertices distinct) yields the \textbf{quotient graph} $G_0$.

The induced labeling $\bar{\eta}$ naturally partitions the interior vertices $\mathcal{V}_{\mathrm{int}}(G_0)$ into two sets:
\begin{itemize}
    \item \textbf{Tied vertices} $\mathcal{V}_{\mathrm{tied}}(G_0)$: vertices with labels $\bar{\eta}(v) \in [r]$.
    \item \textbf{Free vertices} $\mathcal{V}_{\mathrm{free}}(G_0)$: vertices with labels $\bar{\eta}(v) \in [N]\setminus [r]$.
\end{itemize}
Thus, the quotient graph is denoted by the tuple:
\begin{equation}\label{def:G0}
    G_0 = \big(\mathcal{V}_{\mathrm{int}}(G_0),\ \mathcal{V}_b(G_0),\ \mathcal{E}_{\mathrm{int}}(G_0),\ \mathcal{E}_b(G_0)\big).
\end{equation}

For any fixed ribbon graph $\Upsilon$, denote the set of  all possible quotient graphs obtained by some function $\eta:\mathcal{V}(\Upsilon) \to [N]$ by  $\mathcal{G}(\Upsilon)$. 
By construction, the induced labeling $\bar{\eta}$ must satisfy the set of constraints $\mathfrak{C}$:
\begin{equation}\label{equ:eta_constraint}
    \mathfrak{C}:\begin{cases}
        \bar{\eta}(v) \in [r], & \forall v \in \mathcal{V}_{\mathrm{tied}}(G_0), \\
        \bar{\eta}(v) \in [N]\setminus [r], & \forall v \in \mathcal{V}_{\mathrm{free}}(G_0), \\
        \bar{\eta}(u) \ne \bar{\eta}(v), & \forall u \ne v \in \mathcal{V}_{\mathrm{int}}(G_0) \quad (\text{Injectivity}).
    \end{cases}
\end{equation}
We now rewrite $\mathcal{L}_{\Upsilon,X}$ in terms of the sum of the weights of $G_0$ by using the following identity 
\begin{equation}\label{equ:partition111}\sum_{\eta:\mathcal{V}(\Upsilon) \to [N]}\big(\cdot\big) =\sum_{G_0 \in \mathcal{G}(\Upsilon)} \sum_{\bar{\eta}:\mathcal{V}(G_0) \to [N] \ s.t.    \mathfrak{C}}\big(\cdot\big) \ . \end{equation}

For this purpose, for any quotient graph $G_0 \in \mathcal{G}(\Upsilon)$, we define the weight $\mathcal{K}_{G_0,X}$ by summing over all valid labelings $\bar\eta:\mathcal{V}(G_0) \to [N]$ satisfying $\mathfrak{C}$.
\begin{proposition}\label{prop:relation:K}
Let 
\begin{equation}\label{G0expansion}
    \mathcal{K}_{G_0,X} := \rho_a^{-k}\sum_{\mathcal{V}_{\mathrm{tied}}:\mathcal{V}_{\mathrm{tied}}\subset \mathcal{V}_{\mathrm{int}}} \sum_{\bar{\eta}:\mathcal{V}(G_0) \to [N]} 1_{\{\bar\eta:\mathfrak{C}\}} \cdot c(G_0,\bar\eta) \prod_{e \in \mathcal{E}_{\mathrm{int}}(G_0)} \sigma_{\bar\eta(u_e)\bar\eta(v_e)}^{2} \prod_{e \in \mathcal{E}_b(G_0)} A_{\bar\eta(u_e)\bar\eta(v_e)},
\end{equation}
where the coupling factor $c(G_0,\bar\eta)$ is the coupling factor in \eqref{def:c(U,eta)} 
\begin{equation}\label{eq:coupling_factor_def}
    c(G_0,\bar\eta) := \frac{\mathbb{E}\left[ \prod_{e\in \mathcal{E}_{\mathrm{int}}} W_{\bar\eta(u)\bar\eta(v)}^{\,k_{uv}} W_{\bar\eta(v)\bar\eta(u)}^{\,k_{vu}} \right]}
    {\mathbb{E}\left[ \prod_{e\in \mathcal{E}_{\mathrm{int}}} \left( W^{\mathrm{GOE}}_{\bar\eta(u)\bar\eta(v)} \right)^{k_{uv}+k_{vu}} \right]}
\end{equation}
with  the directed traversal counting  $k(e,\eta)=(k_{uv},k_{vu})$   inherited from $\Upsilon$. 
The constraints $\mathfrak{C}$ can be conveniently expressed in terms of  indicator functions:
\begin{equation}\label{equ:constraint_product}
    1_{\{\bar\eta:\mathfrak{C} \}} =
    \prod_{u\neq v\in\mathcal{V}_{\mathrm{int}}(G_0)} \big(1-1_{\{\bar\eta(u)=\bar\eta(v)\}}\big)
    \prod_{u\in\mathcal{V}_{\mathrm{free}}(G_0)} 1_{\{\bar\eta(u)\notin[r]\}}
    \prod_{u\in\mathcal{V}_{\mathrm{tied}}(G_0)} 1_{\{\bar\eta(u)\in[r]\}}.
\end{equation}
Then the total contribution defined in \eqref{equ:def_L}    decomposes  over the set of quotient graphs $\mathcal{G}(\Upsilon)$:
    \begin{equation}
        \mathcal{L}_{\Upsilon,X} = \sum_{G_0 \in \mathcal{G}(\Upsilon)} \mathcal{K}_{G_0,X}.
    \end{equation}
\end{proposition}

\begin{proof}
    The set of all labelings $\eta: \mathcal{V}(\Upsilon) \to [N]$ is partitioned as \eqref{equ:partition111} by the quotient graph $G_0$ they induce.
    Rewriting the sum in \eqref{equ:def_L} by first summing over possible structures $G_0 \in \mathcal{G}(\Upsilon)$ and then summing over labelings consistent with $G_0$ immediately yields the definition of $\mathcal{K}_{G_0,X}$.
\end{proof}

\paragraph{Decoupling the coupling factor.} The injectivity of $\bar{\eta}$ on $\mathcal{V}_{\mathrm{free}}(G_0)$ (cf. \eqref{equ:eta_constraint}) strictly limits label collisions: distinct edges can share the same label pair only if they involve boundary or tied vertices.

\begin{lemma}\label{lem:lemma3.9}
   For any quotient graph $G_0 \in \mathcal{G}(\Upsilon)$ and any labeling $\bar{\eta}:\mathcal{V}(G_0) \to [N]$ satisfying $\mathfrak{C}$, if $\bar{\eta}$ maps the endpoints of two  distinct interior edges $e_1, e_2$ to the same pair $\{\bar{\eta}(u), \bar{\eta}(v)\}$, then one of the following two holds:
    \begin{enumerate}
        \item[(1)] All endpoints belong to $\mathcal{V}_b(G_0) \cup \mathcal{V}_{\mathrm{tied}}(G_0)$.
        \item[(2)] $e_1$ and $e_2$ share exactly one common vertex in $\mathcal{V}_{\mathrm{free}}(G_0)$, while the other endpoints belong to $\mathcal{V}_b(G_0) \cup \mathcal{V}_{\mathrm{tied}}(G_0)$.
    \end{enumerate}
\end{lemma}

\begin{proof}
    The labels of vertices in $\mathcal{V}_{\mathrm{free}}(G_0)$ are distinct from each other and from those in $\mathcal{V}_{\mathrm{tied}} \cup \mathcal{V}_b$.
    If the shared label pair involves a free vertex label, the injectivity implies the corresponding vertices in $e_1$ and $e_2$ must be identical.
    Since $e_1 \ne e_2$, they cannot share \textit{two} free vertices (otherwise $e_1=e_2$). Thus, they can share at most one free vertex.
    If they share one free vertex, the other label must belong to $[r]$, implying the other endpoints are in $\mathcal{V}_b \cup \mathcal{V}_{\mathrm{tied}}$ (Case 2).
    If they share no free vertices, all labels must be in $[r]$, implying all endpoints are in $\mathcal{V}_b \cup \mathcal{V}_{\mathrm{tied}}$ (Case 1).
\end{proof}

\begin{definition}[Free and tied Edges]\label{def:tied_edge}
    An interior edge $e = \{u,v\}$ is said to be  a \textbf{free edge} if
    \begin{enumerate}
        \item[(1)] Both endpoints are free: $u, v \in \mathcal{V}_{\mathrm{free}}(G_0)$; or
        \item[(2)] $u \in \mathcal{V}_{\mathrm{free}}(G_0)$, $v \in \mathcal{V}_{\mathrm{fixed}}:= \mathcal{V}_b(G_0) \cup \mathcal{V}_{\mathrm{tied}}(G_0)$, and $e$ is the unique edge connecting $u$ to $\mathcal{V}_{\mathrm{fixed}}$.
    \end{enumerate}
    All other interior edges are called \textbf{tied edges}, denoted by  $\mathcal{E}_{\mathrm{tied}}$, implying   $\mathcal{E}_{\mathrm{int}} = \mathcal{E}_{\mathrm{free}} \sqcup \mathcal{E}_{\mathrm{tied}}$.
    We split the coupling factor $c(G_0,\bar{\eta})$ into a tied component and a free component.
\end{definition}

\begin{lemma}\label{lem:decouple-coupling}
    For any quotient graph $G_0 \in \mathcal{G}(\Upsilon)$,     the coupling factor factorizes under the constraints $\mathfrak{C}$
    \begin{equation}\label{equ:decouple}
        c(G_0,\bar{\eta}) = c\big(\mathcal{E}_{\mathrm{tied}}(G_0),\bar{\eta}\big) \cdot
        \prod_{e\in\mathcal{E}_{\mathrm{free}}(G_0)} c(e,\bar{\eta}).
    \end{equation}
\end{lemma}

\begin{proof}
    Lemma \ref{lem:lemma3.9} and the definition of free edges ensure that for any $e \in \mathcal{E}_{\mathrm{free}}(G_0)$, the label pair $\{\bar{\eta}(u_e), \bar{\eta}(v_e)\}$ is distinct from that of any other edge in $G_0$.
    Consequently, the random variables $\{W_{\bar{\eta}(u)\bar{\eta}(v)}\}_{e \in \mathcal{E}_{\mathrm{free}}}$ are mutually independent and independent of those in $\mathcal{E}_{\mathrm{tied}}$.
    The expectation of the product thus splits into the product of expectations, yielding \eqref{equ:decouple}.
\end{proof}

\paragraph{Relaxing Injectivity on $\mathcal{V}_{\mathrm{free}}$.}
The injectivity constraint on $\mathcal{V}_{\mathrm{free}}(G_0)$ makes summation over trees difficult. To address this, we employ the following inclusion–exclusion principle to relax the injectivity of $\bar\eta$ via a sum over partitions (see Appendix \ref{sec:proof_inclusion} for the proof).

\begin{lemma}[Inclusion-Exclusion for Injectivity]\label{lem:exclusion_inclusion}
    For any finite set $V$ and any function $\eta: V \to [N]$, the indicator function of injectivity $\prod_{u\neq v}(1-1_{\{\eta(u)=\eta(v)\}})$ can be expanded as
    \begin{equation}
        \sum_{\pi \in \mathcal{P}(V)} C_{\pi} \prod_{B\in\pi}1_{\{\eta \text{ is constant on } B\}}, \quad C_{\pi}:=(-1)^{\sum_{B\in\pi}(|B|-1)}\prod_{B\in\pi}(|B|-1)!.
    \end{equation}
\end{lemma}

For any partition $\pi \in \mathcal{P}(\mathcal{V}_{\mathrm{free}})$, let $G_0/\pi$ be the graph where vertices in each block $B \in \pi$ are merged. Setting
\begin{equation}
    1_{\{\bar\eta, \pi\}}=\mathbf{1}(\forall x\in \pi_i,y\in \pi_j, \eta(x)= \eta(y)\text{ if and only if } i=j),
\end{equation}
we define the class weight $\mathcal{F}_{G_0/\pi,X}$ as the sum where labels are forced to be constant within blocks of $\pi$,
\begin{equation}\label{def:F}
\begin{aligned}
\mathcal{F}_{G_0/\pi,X,\mathcal{V}_{\mathrm{tied}}}
:= \sum_{\bar{\eta}:\mathcal{V}(G_0) \to [N]}
& \Big( 1_{\{\bar\eta, \pi\}} \prod_{u\in\mathcal{V}_{\mathrm{free}}} 1_{\{\bar\eta(u)\notin[r]\}} \Big)
\Big( \prod_{u \ne v \in \mathcal{V}_{\mathrm{tied}}} (1-1_{\{\bar\eta(u)=\bar\eta(v)\}}) \prod_{u\in\mathcal{V}_{\mathrm{tied}}} 1_{\{\bar\eta(u)\in[r]\}} \Big) \\
\times & \Big( c(\mathcal{E}_{\mathrm{tied}},\bar\eta) \prod_{e\in\mathcal{E}_{\mathrm{free}}} c(e,\bar\eta) \Big)
\Big( \prod_{e\in\mathcal{E}_{\mathrm{int}}} \sigma_{\bar\eta(u_e)\bar\eta(v_e)}^{2} \prod_{e\in\mathcal{E}_b} A_{\bar\eta(u_e)\bar\eta(v_e)} \Big).
\end{aligned}
\end{equation}

Recalling that $\mathcal{K}_{G_0,X}$ is the weight of the quotient graph $G_0$ defined in \eqref{G0expansion}, we apply Lemma \ref{lem:exclusion_inclusion} to $\mathcal{K}_{G_0,X}$ and obtain the following identity:

\begin{corollary}\label{coro:upper_exclusion}
    The weight $\mathcal{K}_{G_0,X}$ admits an expansion 
    \begin{equation*}
        \mathcal{K}_{G_0,X} = \rho_a^{-k} \sum_{\mathcal{V}_{\mathrm{tied}}:\mathcal{V}_{\mathrm{tied}}\subset \mathcal{V}_{\mathrm{int}}} \sum_{\pi \in \mathcal{P}(\mathcal{V}_{\mathrm{free}})} C_{\pi} \cdot \mathcal{F}_{G_0/\pi,X,\mathcal{V}_{\mathrm{tied}}}.
    \end{equation*}
\end{corollary}

\begin{proof}
    The constraints $\mathfrak{C}$ factorize into conditions on $\mathcal{V}_{\mathrm{tied}}$ (which remain fixed) and injectivity on $\mathcal{V}_{\mathrm{free}}$. Substituting the expansion from Lemma \ref{lem:exclusion_inclusion} for the $\mathcal{V}_{\mathrm{free}}$ terms directly yields the result.
\end{proof}

\subsubsection{Bound  class weight \texorpdfstring{$|\mathcal{F}_{G_0/\pi,X}|$}{FG0}}
\paragraph{Summing over boundary trees in $G_0/\pi$.}\label{par:sum_boundary_trees}

To bound $|\mathcal{F}_{G_0/\pi,X}|$, we simplify the graph structure by removing tree-like components attached to the boundary.

\begin{definition}[Boundary Tree]\label{def:boundary_tree}
Let $G$ be a quotient graph. 
 A \emph{boundary tree} is a maximal tree $T \subset G$ such that:
\begin{enumerate}
    \item[(1)] The vertex set $V(T)$ consists entirely of free vertices, except for exactly one root vertex $v \in \mathcal{V}_b(G)$;
    \item[(2)] $T$ is connected to the rest of the graph $G \setminus E(T)$  only through the root $v$.
\end{enumerate}
The set of all such trees is denoted  by $\mathcal{T}_b(G)$. 
\end{definition}

Now we  remove these boundary trees and simplify chains of boundary edges.

\begin{definition}[Reduced Boundary Vertex Set]\label{def:reduced_boundary_set}
    The \emph{reduced boundary vertex set},  denoted by $\widetilde{\mathcal{V}}_b(G_0)$, is obtained by iteratively removing all \textbf{boundary trees} (Definition \ref{def:boundary_tree}) and then removing  
    any resulting non-marked divalent boundary vertices.  
    
\end{definition}

We need a quantitative relation between 
$|\widetilde{\mathcal{V}}_b(G)|$, $|\mathcal{E}_{\mathrm{int}}(G)|-|\mathcal{V}_{\mathrm{int}}(G)|$, 
and $|\mathcal{V}_{\mathrm{tied}}(G)|$. The proof is deferred to Appendix \ref{appendix:lemma_graph}.  

\begin{lemma}\label{lem:V_b_upper_bound}
Let $G$ be a weighted multi-graph with tied vertex set $\mathcal{V}_{\mathrm{tied}}(G)$ 
and at most $s$ marked points. Then
\begin{equation}
    |\widetilde{\mathcal{V}}_b(G)| 
    \;\le\; s \,+\, 2\bigl(|\mathcal{E}_{\mathrm{int}}(G)|-|\mathcal{V}_{\mathrm{int}}(G)|\bigr) 
    \,+\, |\mathcal{V}_{\mathrm{tied}}(G)|.
\end{equation}
\end{lemma}

Now we are ready to get the upper bound on $\mathcal{F}_{G_0/\pi,X}$.
\begin{lemma}[Bound on $\mathcal{F}_{G_0/\pi,X}$]\label{lem:F_upper_bound}
    Let $G_0 \in \mathcal{G}(\Upsilon)$ be a quotient graph derived from $\Upsilon$, and let $\pi$ be any partition of $\mathcal{V}_{\mathrm{free}}(G_0)$. Denote by $|\mathcal{C}(G_0)|$   the number of connected components of $G_0$ without boundary edges. If $M\gamma ((r+1)\sigma^*_N)^2\le 1$ Then 
    \begin{equation}
        \left| \mathcal{F}_{G_0/\pi,X,\mathcal{V}_{\mathrm{tied}}} \right|
            \;\le\;
            \big(\frac{N}{M}\big)^{|\mathcal{C}(G_0)|}a^{|\mathcal{E}_{b}(G_0)|}
            M^{-|\mathcal{E}_{\mathrm{int}}(G_0)| + |\mathcal{V}_{\mathrm{free}}(G_0/\pi )|}
            (r+1)^{s}.
    \end{equation}
\end{lemma}
\begin{proof}

\textbf{Sum over boundary trees $\mathcal{T}_b(G_0/\pi)$.} The total weights of such tree only depends on the root of the tree and the shape of the tree, we denote by $(\Lambda_i)_{xx}$ the weight of the $i$-th vertex with label $x\in [r]$. Moreover, the weights do not exceed $1$ by the doubly stochastic requirement.

    Now after removing all the boundary trees mentioned before, we sum over all non-marked divalent boundary vertices. For consecutive divalent boundary vertices with label $p_1,\ldots p_{k}$, we have
    \begin{equation}
    |\sum_{p_1,\dots,p_{k}} \prod_{i=0}^{k-1} A_{p_i p_{i+1}} (\Lambda_{i+1})_{p_{i}p_{i}} |= |\bigl(\prod_{i=1}^{k} A\Lambda_i\bigr)_{p_0 p_k}|\le \prod_{i=1}^k \|A\|_{\mathrm{op}} \|\Lambda_i\|_{\mathrm{op}} \le a^k.
    \end{equation}
    Summing over all non-marked divalent boundary vertices, the total contribution from the boundary edge terms $\prod A_{\eta(u_e)\eta(v_e)}$ is bounded by $a^{|\mathcal{E}_{b}(G_0)|}.$ The set of remaining boundary vertices is $\widetilde{\mathcal{V}}_b(G_0/\pi)$.

\paragraph{Internal Contribution.} Now we sum over labels of interior vertices. We remove all boundary edges, denoting the new graph by $G_1$.
We analyze the contribution from interior edges and vertices. 

First, recall the bound on the coupling factor:
\begin{equation}
    |c(e,\bar{\eta})| \le \gamma^{\frac{k_{u_ev_e}+k_{v_eu_e}}{2}-1}.
\end{equation}
This implies that any ``extra" traversal (creating a loop or multiple edge) contributes a factor bounded by $\gamma (\sigma_{ij})^2\le \gamma(\sigma^*_N)^2$.

To formalize this, consider a maximal spanning forest $\mathcal{F}$ of $G_1$ dividing and rooting over all marked vertices, boundary vertices $\widetilde{\mathcal{V}}_n(G_0/\pi)$ and tied vertices. We decompose the edge set as $\mathcal{E}_{\mathrm{int}} = \mathcal{E}(\mathcal{F}) \sqcup \mathcal{E}_{extra}$, where $\mathcal{E}_{extra}$ contains all cycle edges and multiple edges.

\begin{enumerate}
    \item \textbf{Summation over $\mathcal{F}$:} Summing the labels over the leaves of the forest contributes a factor $\le 1$ due to the double stochasticity $\sum_j \sigma_{ij}^2 = 1$.
    
    \item \textbf{Roots contribution:} The summation stops at the roots.
    \begin{itemize}
        \item Roots of free marked vertices have $N$ choices, contributing at most $N^{|\mathcal{C}(G_0/\pi)|}$.
        \item Roots in boundary vertices $\widetilde{\mathcal{V}}_b(G_0/\pi)$ or tied vertices have $r$ choices, contributing at most $r^{|\mathcal{V}_{\mathrm{tied}}| + |\widetilde{\mathcal{V}}_b|}$.
    \end{itemize}

    \item \textbf{Excess edges contribution:} Each edge $e \in \mathcal{E}_{extra}$ contributes a weight factor of $\gamma (\sigma^*_N)^2$. The number of such edges is determined by
    \begin{equation}
        |\mathcal{E}_{extra}| = |\mathcal{E}_{\mathrm{int}}| - |\mathcal{E}(\mathcal{F})|.
    \end{equation}
    Since each component of $\mathcal{F}$ contains only free vertices with one root vertex,  the number of free root vertices is equal  to the number of free marked vertices, which is given by $|\mathcal{C}(G_0/\pi)|$.  Thus, the total  number of edges is
    \begin{equation}
    \begin{aligned}
        |\mathcal{E}(\mathcal{F})| &= |V(\mathcal{F})|-|\mathcal{C}(\mathcal{F})| = |\mathcal{V}_{\mathrm{free}}|-|\mathcal{C}(G_0/\pi)| + |\mathcal{C}(\mathcal{F})|- |\mathcal{C}(\mathcal{F})|= |\mathcal{V}_{\mathrm{free}}|-|\mathcal{C}(G_0/\pi)|.
    \end{aligned}
    \end{equation}
\end{enumerate}

Combining these factors, the total internal contribution is bounded by
\begin{equation}
\begin{aligned}
    N^{|\mathcal{C}(G_0/\pi)|} \,
    r^{|\mathcal{V}_{\mathrm{tied}}| + |\widetilde{\mathcal{V}}_b|} \,
    \bigl(\gamma (\sigma^*_N)^2\bigr)^{|\mathcal{E}_{extra}|}=N^{|\mathcal{C}(G_0/\pi)|} \,
    r^{|\mathcal{V}_{\mathrm{tied}}| + |\widetilde{\mathcal{V}}_b|} \,
    \bigl(\gamma (\sigma^*_N)^2\bigr)^{|\mathcal{E}_{\mathrm{int}}| - (|\mathcal{V}_{\mathrm{free}}| - |\mathcal{C}(G_0/\pi)|)}.
\end{aligned}
\end{equation}
By Lemma \ref{lem:V_b_upper_bound}, we have
\begin{equation}
\begin{aligned}
        |\mathcal{V}_{\mathrm{tied}}| + |\widetilde{\mathcal{V}}_b|&\le s+2\bigl(|\mathcal{E}_{\mathrm{int}}(G_0/\pi)|-|\mathcal{V}_{\mathrm{int}}(G_0/\pi)| \,+\, |\mathcal{V}_{\mathrm{tied}}(G_0/\pi)|\bigr)\\
        &= s+2\bigl(|\mathcal{E}_{\mathrm{int}}(G_0/\pi)|-|\mathcal{V}_{\mathrm{free}}(G_0/\pi)|\bigl).
\end{aligned}
\end{equation}
Combining the contribution of boundary edges, the total weight is bounded by
\begin{equation}
\begin{aligned}
    &a^{|\mathcal{E}_{b}(G_0)|}N^{|\mathcal{C}(G_0/\pi)|} \,
    r^{|\mathcal{V}_{\mathrm{tied}}| + |\widetilde{\mathcal{V}}_b|} \,
    \bigl(\gamma (\sigma^*_N)^2\bigr)^{|\mathcal{E}_{\mathrm{int}}| - (|\mathcal{V}_{\mathrm{free}}| - |\mathcal{C}(G_0/\pi)|)}\\
    &\le a^{|\mathcal{E}_{b}(G_0)|}(N\gamma (\sigma^*_N)^2)^{|\mathcal{C}(G_0/\pi)|}\bigl(\gamma ((r+1)\sigma^*_N)^2\bigr)^{|\mathcal{E}_{\mathrm{int}}| - |\mathcal{V}_{\mathrm{free}}| }(r+1)^{s}.
    \end{aligned}
\end{equation}
Note that 
\begin{equation}
    M\gamma ((r+1)\sigma^*_N)^2\le 1,\quad |\mathcal{C}(G_0/\pi)|\le |\mathcal{C}(G_0)|,\quad |\mathcal{E}_{\mathrm{int}}(G_0/\pi)|=|\mathcal{E}_{\mathrm{int}}(G_0)|.
\end{equation}
Rearranging the terms yields the bound in the lemma.
\end{proof}

 \subsubsection{Final proof of Theorem \ref{prop:gauss_irm}}
We then derive the estimate for $\mathcal{K}_{G_0,X}$ from  the above estimate for $\mathcal{F}_{G_0/\pi,X}$. 
 To do so, recalling the relation between
$\mathcal{K}_{G_0,X}$ and 
$\mathcal{F}_{G_0/\pi,X}$ in Corollary \ref{coro:upper_exclusion}, we need the following lemma. The  proof is detailed  in Appendix \ref{appendix A2}.
\begin{lemma}\label{lem:sum_partition}
For any integer $t\ge 1$, let $\mathcal{P}([t])$ be the set of all possible partition of $[t]$,  we have
\begin{equation}\label{equ:partition}
\sum_{\pi\in  \mathcal{P}([t])}
\prod_{B\in\pi} (|B|-1)!\; M^{-(|B|-1)}
\le \Big(1+\frac{t}{M}\Big)^{t} \le e^{\frac{t^2}{M}}.
\end{equation}
\end{lemma}

\begin{corollary}\label{coro:KX<KY} For $M\gamma((r+1)\sigma_N^*)^2\le 1, M\le N$ and $k:=k_1+\cdots+ k_s\le  {M}/{2}$,   let  $|\mathcal{C}(G_0/\pi)|$ be the number of connected components of $G_0/\pi$ without any boundary edge,  then 
    \begin{equation}
        |\mathcal{K}_{G_0,X}|\le (r+1)^{s}e^{\frac{2k^2}{M}}\big(\frac{N}{M}\big)^{|\mathcal{C}(G_0)|}\mathcal{K}_{G_0,Y}.
    \end{equation}

\end{corollary}
\begin{proof}
    We only consider the case $G_0$ is connected since the non-connected case is similar. By Corollary \ref{coro:upper_exclusion} and Lemma \ref{lem:F_upper_bound}, we have
    \begin{equation}
    \begin{aligned}
        \rho_a^{k}|\mathcal{K}_{G_0,X}|&\le \sum_{\mathcal{V}_{\mathrm{tied}}:\mathcal{V}_{\mathrm{tied}}\subset \mathcal{V}_{\mathrm{int}}}\sum_{\pi}
\Bigg(\prod_{B\in\pi}(|B|-1)!\Bigg)
\;\bigl|\mathcal{F}_{G_0/\pi,X,\mathcal{V}_{\mathrm{tied}}}\bigr|\\
&\le (r+1)^s\big(\frac{N}{M}\big)^{|\mathcal{C}(G_0)|}a^{|\mathcal{E}_{b}(G_0)|}\sum_{\mathcal{V}_{\mathrm{tied}}:\mathcal{V}_{\mathrm{tied}}\subset \mathcal{V}_{\mathrm{int}}}\sum_{\pi}
\Bigg(\prod_{B\in\pi}(|B|-1)!\Bigg)
            M^{-|\mathcal{E}_{\mathrm{int}}(G_0)| + |\mathcal{V}_{\mathrm{free}}(G_0/\pi )|}.
\end{aligned}
    \end{equation}
    From the definition of $G_0/\pi,$ the following graph identities hold
    \begin{equation}
        |\mathcal{V}_{\mathrm{free}}(G_0/\pi)|= |\mathcal{V}_{\mathrm{free}}(G_0)|-\sum_{B\in \pi} (|B|-1).
    \end{equation}
   So we have
    \begin{equation}
        \begin{aligned}
        &\sum_{\pi\in \mathcal{P}(\mathcal{V}_{\mathrm{free}}(G_0))}
\Bigg(\prod_{B\in\pi}(|B|-1)!\Bigg)
            M^{-|\mathcal{E}_{\mathrm{int}}(G_0)| + |\mathcal{V}_{\mathrm{free}}(G_0/\pi )|}\\
            &=M^{-|\mathcal{E}_{\mathrm{int}}(G_0)|+|\mathcal{V}_{\mathrm{free}}(G_0)|}\sum_{\pi\in \mathcal{P}(\mathcal{V}_{\mathrm{free}}(G_0))}
\Bigg(\prod_{B\in\pi}(|B|-1)!M^{-(|B|-1)}\Bigg)\\
&\le M^{-|\mathcal{E}_{\mathrm{int}}(G_0)|+|\mathcal{V}_{\mathrm{free}}(G_0)|}e^{\frac{1}{M}|\mathcal{V}_{\mathrm{free}}(G_0)|^2}\\
&\le M^{-|\mathcal{E}_{\mathrm{int}}(G_0)|+|\mathcal{V}_{\mathrm{free}}(G_0)|} e^{\frac{k^2}{M}}.
        \end{aligned}
    \end{equation}
    Here we use Lemma \ref{lem:sum_partition} and $|\mathcal{V}_{\mathrm{free}}(G_0)|\le k$ in last two inequalities.

Thus we obtain
    \begin{equation}\label{equ:KXupper}
    \begin{aligned}
        \rho_a^{k}|\mathcal{K}_{G_0,X}|
&\le (r+1)^s\big(\frac{N}{M}\big)^{|\mathcal{C}(G_0)|}a^{|\mathcal{E}_{b}(G_0)|}M^{-|\mathcal{E}_{\mathrm{int}}(G_0)|+|\mathcal{V}_{\mathrm{free}}(G_0)|} e^{\frac{k^2}{M}}\sum_{\mathcal{V}_{\mathrm{tied}}}1\\
&\le (r+1)^s\big(\frac{N}{M}\big)^{|\mathcal{C}(G_0)|}a^{|\mathcal{E}_{b}(G_0)|}M^{-|\mathcal{E}_{\mathrm{int}}(G_0)|+|\mathcal{V}_{\mathrm{free}}(G_0)|} e^{\frac{k^2}{M}}k^{|\mathcal{V}_{\mathrm{tied}}|}\\
&\le (r+1)^s\big(\frac{N}{M}\big)^{|\mathcal{C}(G_0)|}a^{|\mathcal{E}_{b}(G_0)|}M^{-|\mathcal{E}_{\mathrm{int}}(G_0)|+|\mathcal{V}_{\mathrm{int}}(G_0)|} e^{\frac{k^2}{M}}.
\end{aligned}
    \end{equation} In the last line, we use $k\le M$ and $|\mathcal{V}_{\mathrm{int}}|=|\mathcal{V}_{\mathrm{free}}|+|\mathcal{V}_{\mathrm{tied}}|$. On the other hand, from \eqref{G0expansion}, where we only need $\bar\eta$ is injective, we obtain
\begin{equation}\label{equ:KYlower}
\begin{aligned}
        \rho_a^{k}\mathcal{K}_{G_0,Y}&= a^{|\mathcal{E}_{b}(G_0)|} M^{-|\mathcal{E}_{\mathrm{int}}(G_0)|} M(M-1) \cdot (M-|\mathcal{V}_{\mathrm{int}}(G_0)|+1)\\
        &\ge a^{|\mathcal{E}_{b}(G_0)|}M^{-|\mathcal{E}_{\mathrm{int}}(G_0)|+|\mathcal{V}_{\mathrm{int}}(G_0)|}e^{-\frac{k^2}{M}}.
\end{aligned}
\end{equation}
Combining \eqref{equ:KXupper} and \eqref{equ:KYlower} yields the desired estimate
\begin{equation}
    |\mathcal{K}_{G_0,X}|\le (r+1)^{s}e^{\frac{2k^2}{M}}\big(\frac{N}{M}\big)^{|\mathcal{C}(G_0)|}\mathcal{K}_{G_0,Y},
\end{equation}  which completes the proof.
\end{proof}

Now we are ready to  prove Theorem \ref{prop:gauss_irm}.
\begin{proof}[Proof of Theorem \ref{prop:gauss_irm}] We focus on the case where $\Gamma$ has boundary edges, as the case without boundary edges is straightforward. By Corollary \ref{coro:KX<KY},  
    \begin{equation}
        |\mathcal{K}_{G_0,X}|\le (r+1)^{s}e^{\frac{2k^2}{M}}\big(\frac{N}{M}\big)^{|\mathcal{C}(G_0)|}\mathcal{K}_{G_0,Y}.
    \end{equation}
Noting that $|\mathcal{C}(G_0)|\le |\mathcal{C}(\Gamma)|$ because $G_0$ is a quotient of $\Gamma$, we apply Proposition \ref{prop:moments-iid} together with Proposition \ref{prop:relation:K} to obtain
 \begin{equation}
        \left|\mathcal{W}_{\Gamma, X_N}(k_1,\dots,k_s)\right|\le (r+1)^{s}e^{\frac{2k^2}{M}}\big(\frac{N}{M}\big)^{|\mathcal{C}(\Gamma)|}\mathcal{W}_{\Gamma, Y_M}(k_1,\dots,k_s),
    \end{equation}
which completes the proof.
\end{proof}

\subsection{Upper bounds of mixed moments}\label{sec:moments}

As a consequence of Theorem \ref{thm:dominating_function}, we obtain the   upper bound estimates for the moments of $X_N$, which will be applied several times in Section \ref{sec:proof_of_main}.

\begin{corollary}\label{upper_bounds_moments}
    With $X_N$ given in Definition  \ref{def:inhomo}, we assume that for $j=1,2,\ldots,s$,
    \begin{equation}\label{equ:summption_k_moments}
   k_j \le \frac{\xi}{2s\sqrt{\gamma} (r+1)\sigma_N^*},  
\end{equation} 
and set $k_{\rm min}=\min \{k_1,\ldots,k_s\}$. If $\xi \le \frac{1}{2}\gamma^{-1/2} ((r+1)\sigma_N^*)^{-1}$, then there exists an absolute constant $C(s,a)>0$ depending only on $a$ and $s$ such that the following upper bounds hold:
    
\noindent \text{(i)}   
\begin{equation}\label{equ:upper_bounds_moments_1}
\rho_a^{-\sum_{j}k_j} \Big| \mathbb{E}[\prod_{j=1}^{s}\tr X^{k_j}] - \mathbb{E}[\prod_{j=1}^{s}(\tr X^{k_j}-\tr H^{k_j})] \Big| \le C(s,a)(r+1)^{s}N^s\Big(\frac{2}{\rho_a}\Big)^{k_{\rm min}}e^{C(s,a)\xi^2},
\end{equation}
\noindent\text{(ii)} 
\begin{equation}\label{equ:upper_bounds_moments_2}
        \rho_a^{-\sum_{j}k_j}\mathbb{E}\Big[\prod_{j=1}^{s}\tr X_{N}^{k_j}\Big]  \le (r+1)^{s}C(s,a)e^{C(s,a)\xi^2}\Big(1+N^s\Big(\frac{2}{\rho_a}\Big)^{k_{\rm min}}\Big). \end{equation}
\end{corollary}
\begin{proof}
  Taking $M=\lfloor \gamma^{-1}((r+1)\sigma_N^*)^{-2} \rfloor$, we see from \eqref{equ:summption_k_moments} $\sum_{j=1}^s k_j \le \xi \sqrt{M}$. Since $\xi \le \frac{1}{2}\gamma^{-1/2} ((r+1)\sigma_N^*)^{-1} \le \sqrt{M}/2 $, we have $\sum_{j=1}^{s}k_j \le   M/2$. Therefore, the assumptions of Theorem \ref{Prop:Upper bound} and Theorem \ref{prop:gauss_irm} are satisfied.  We first prove \eqref{equ:upper_bounds_moments_1} and \eqref{equ:upper_bounds_moments_2} for the case 
$s=1$, and then extend the result to general 
$s$ case.

\paragraph{Proof of Part (i): $s=1$.} We first prove \eqref{equ:upper_bounds_moments_1} for $s=1$. Combining the  diagram expansions (Proposition \ref{prop:moments-iid},  Proposition \ref{prop:diagram_expansion_subGaussian_boundary}) and Theorem \ref{prop:gauss_irm}, we see that for any even $k\le\frac{\xi}{\sqrt{\gamma} (r+1)\sigma_N^*}$ (noting by symmetry, $\mathbb{E}\big[\tr H^{k}\big]=0$ if $k$ is odd),
\begin{equation}
\begin{aligned}   \rho_a^{-k}\mathbb{E}\big[\tr H^{k}\big] =& \sum_{\Gamma \in \mathcal{D}_{\beta,1}\backslash \mathcal{D}_{\beta,1}^{b}}\mathcal{W}_{\Gamma, X_N}(k)\\ 
    \le &\frac{N}{M}e^{\frac{2k^2}{M}}\sum_{\Gamma \in \mathcal{D}_{\beta,1}\backslash \mathcal{D}_{\beta,1}^{b}}\mathcal{W}_{\Gamma, H_{{\rm GOE}_M}}(k)=\frac{N}{M}e^{\frac{2k^2}{M}}\rho_a^{-k}\mathbb{E}\big[\tr H_{{\rm GOE}_M}^{k}\big].
\end{aligned}
\end{equation}

By \cite[Eq. (I.5.6)]{feldheim2010universality} (see also \cite[Theorem 8]{ledoux2009recursion}), we have the following upper bound for the trace of high power of $H_{{\rm GOE}_M}$:
\begin{equation}\label{ineq:H^k}
    \mathbb{E}\big[\mathrm{Tr}{H^{k}_{\mathrm{GOE}_M}}\big]\le C_0M2^{k}e^{C_0\frac{k^3}{M^2}},
\end{equation}
where $C_0$ is a positive constant independent of $k$ and $M$.
Hence, we get
\begin{equation}\label{equ:upper_for_H,s=1}
 \rho_a^{-k}\mathbb{E}\big[\tr H^{k}\big] \le C_0\rho_a^{-k}2^kNe^{C_0(\frac{k^3}{M^2}+\frac{k^2}{M})} \le C_0\rho_a^{-k}2^kNe^{C_0\frac{k^2}{M}},
\end{equation}
where in the last inequality we use $k \le M/2$.

\paragraph{Proof of Part (ii):  $s=1$.} Taking $s=1$ in Proposition \ref{prop:diagram_expansion_subGaussian_boundary}, we obtain 
\begin{equation} \rho_a^{-k}\mathbb{E}\big[\tr X^{k}\big]=\rho_a^{-k}\mathbb{E}\big[\tr H^{k}\big]+ \sum_{\Gamma \in \mathcal{D}_{\beta,1}^{b}}\mathcal{W}_{\Gamma, X_N}(k).
\end{equation}
By Theorem \ref{thm:dominating_function}, for any $k\le\frac{\xi}{\sqrt{\gamma} (r+1)\sigma_N^*}$, we obtain the following upper bound 
\begin{equation}\label{equ:upper_bound_X_s=1}
    \sum_{\Gamma \in \mathcal{D}_{\beta,1}^{b}}\mathcal{W}_{\Gamma, X_N}(k) \le  C_a(r+1)e^{C_a
    \frac{k^2}{M}}.
\end{equation}
The proof of (ii) when $s=1$ follows by combining \eqref{equ:upper_for_H,s=1} and \eqref{equ:upper_bound_X_s=1}.

\paragraph{Proof of Part (i):  $s \ge 2$.} Now we prove \eqref{equ:upper_bounds_moments_1} for $s \ge 2$. 
Noting the following expansion 
\begin{equation}\label{equ:E}
   \mathbb{E}\big[\prod_{j=1}^{s}\tr X^{k_j}\big]-\mathbb{E}\big[\prod_{j=1}^{s}(\tr X^{k_j}-\tr H^{k_j})\big]
 =\sum_{\substack{J \subset [s]\\ J \neq \emptyset}}(-1)^{|J|}\mathbb{E}\Big[\prod_{j\in J}\tr H^{k_j}\prod_{j\notin J}\tr X^{k_j}\Big],
\end{equation} it suffices to prove that  for any non-empty subset $J \subset \{1,2,\ldots,s\}$,  
\begin{equation}\label{equ:555}
\Big| \mathbb{E}\Big[\prod_{j\in J}\tr H^{k_j}\prod_{j\notin J}\tr X^{k_j}\Big]\Big| \le C(s,a)(r+1)^sN^s2^{\sum_{j \in J}k_j}\rho_a^{\sum_{j \notin J}k_j}.
\end{equation}
Using the H\"older inequality
\begin{equation}
    \Big|\mathbb{E}[\prod_{i=1}^s \xi_i]\Big|\le \prod_{i=1}^{s} \left(\mathbb{E}[|\xi_i|^{2s}]\right)^{\frac{1}{2s}},
\end{equation}
we see 
\begin{equation}  \mathbb{E}\Big[\prod_{j\in J}\tr H^{k_j}\prod_{j\notin J}\tr X^{k_j}\Big] \le \prod_{j \in J}\left(\mathbb{E}\Big[\big(\tr H^{k_j}\big)^{2s}\Big]\right)^{\frac{1}{2s}}\prod_{j \notin J}\left(\mathbb{E}\Big[\big(\tr X^{k_j}\big)^{2s}\Big]\right)^{\frac{1}{2s}}.
\end{equation}
Then, using  H\"older inequality again but for its eigenvalues we see for any $N \times N$ Hermitian matrix $L$ and any $k$,
\begin{equation}\label{equ:Holder}
(\tr L^k)^{2s} =\big(\sum_{\ell=1}^{N}\lambda_{\ell}^{k}\big)^{2s}\le N^{2s-1}\sum_{\ell=1}^{N}\lambda_\ell^{2ks}=  N^{2s-1}\tr L^{2ks}.
\end{equation}Combining with \eqref{equ:upper_for_H,s=1}   (noting that $2sk_j\le\frac{\xi}{\sqrt{\gamma} (r+1)\sigma_N^*}$ by \eqref{equ:summption_k_moments}), we  get   the   upper bound for the moments of the non-deformed matrix $H$
\begin{equation}\label{equ:moments_upper_bounds_H}
    \prod_{j \in J}\left(\mathbb{E}\Big[\big(\tr H^{k_j}\big)^{2s}\Big]\right)^{\frac{1}{2s}} \le N^{|J|(1-\frac{1}{2s})} \prod_{j \in J}\left(\mathbb{E}[\tr H^{2sk_j}]\right)^{\frac{1}{2s}} \le CN^{|J|}2 ^{\sum_{j\in J}k_j}e^{2C_0s^2 \xi^2}.
\end{equation}
Similarly, by applying \eqref{equ:Holder} for the moments of $X$ and using \eqref{equ:upper_bounds_moments_2} for $s=1$ (which has been  proved in the above step), we obtain
\begin{equation}\label{equ:moments_upper_bounds_X}
\prod_{j \notin J}\Big(\mathbb{E}\Big[\big(\tr X^{k_j}\big)^{2s}\Big]\Big)^{\frac{1}{2s}} \le N^{s-|J|}C_a \prod_{j \notin J}\Big(\mathbb{E}[\tr X^{2sk_j}]\Big)^{\frac{1}{2s}}\le N^{s-|J|}C_a (r+1)^s\rho_a^{\sum_{j \notin J}k_j} e^{2C_as^2 \xi^2}.
\end{equation}The proof of \eqref{equ:555} is thus completed by combining \eqref{equ:moments_upper_bounds_H} with \eqref{equ:moments_upper_bounds_X}.

\paragraph{Proof of (ii):  $s\ge 2$.} By \eqref{equ:upper_bounds_moments_1}, it remains to establish  the following upper bound:
\begin{equation}\label{equ:upper_bounds_boundary_irm}
   \Big| \mathbb{E}\Big[\prod_{j=1}^{s}(\tr X^{k_j}-\tr H^{k_j})\Big]\Big| \le  C_1(s,a)\, e^{C_{2}(s,a)\xi^2}, 
\end{equation}
which follows by combining Proposition \ref{prop:diagram_expansion_subGaussian_boundary} with Theorem \ref{thm:dominating_function}.
\end{proof}

\section{BBP transition}\label{sec:proof_of_main}
This section presents the proof of Theorem \ref{thm:LLN1}. Specifically, we will proceed by  establishing the almost sure upper  bounds  for the   spectral norm  $\|X_N\|_{\mathrm{op}}$,   concentration inequalities  of the central moments  of matrix powers,   and  convergence of the spectral projection measure  as stated   in Theorem \ref{thm:spectral_measure}.   

\subsection{Almost sure  bounds for the spectral norm 
}

As a first step, we establish an upper bound for the spectral norm. 
\begin{lemma}\label{lem:rho_a_a.s.}
     With $X_N$ given in 
Definition  \ref{def:inhomo}, assume that  the spectral    norm $\|A_N\|_{\mathrm{op}}\le a$ with some constant $a>1$. If   $(r+1)\sigma_N^*\sqrt{\log N}\to 0$ as $N\to \infty$, then  
\begin{equation}
        \limsup\limits_{N\rightarrow \infty}\|X_N\|_{\mathrm{op}}\le \rho_a~\ \  a. s.
    \end{equation}
 \end{lemma}
 \begin{proof}[Proof] 
Take  $s=1$, $k=\lfloor(\log N)^{1/2}\gamma^{-1/2}((r+1)\sigma_N^*)^{-1}\rfloor$, and $\xi=\sqrt{\log N}$ in Corollary \ref{upper_bounds_moments}. Since $\gamma((r+1)\sigma_N^*)^2\log N\rightarrow 0$, we see $\xi \le \frac{1}{2}\gamma^{-1/2} ((r+1)\sigma_N^*)^{-1}$ and $  k \gg \log N$. Combining  Corollary \ref{upper_bounds_moments} and the Markov inequality, for any given $\epsilon>0$ we obtain
    \begin{equation}
        \mathbb{P}(\|X_N\|_{\mathrm{op}}\ge (1+\epsilon)\rho_a)\le ((1+\epsilon)\rho_a)^{-2k}\mathbb{E}[\tr X_{N}^{2k}]\le CN^{C}(1+\epsilon)^{-2k}.
    \end{equation}
    
    Since $k\gg \log N$,  there is  a constant  $C_{\epsilon,a}$ depending on $ \epsilon,a$ such that 
    \begin{equation}
        \mathbb{P}(\|X_N\|_{\mathrm{op}}\ge (1+\epsilon)\rho_a)\le C_{\epsilon,a}N^{-2024}.
    \end{equation}
    By the Borel-Cantelli lemma,  we obtain 
    \begin{equation}
        \limsup\limits_{N\rightarrow \infty}\|X_N\|_{\mathrm{op}}\le (1+\epsilon)\rho_a~~~~ a. s.
    \end{equation}
    Taking $\epsilon\rightarrow 0$ yields  the desired result.
 \end{proof}

One of the key results established in this section is the absence of outliers whenever the perturbation remains below a certain threshold.
\begin{theorem}
\label{thm:no_outlier}
    With $X_N$ given in 
Definition  \ref{def:inhomo},  assume that  the spectral norm    $\|A_N\|_{\mathrm{op}}\le 1$. If    $(r+1)\sigma_N^*\sqrt{\log N}\to 0$ as $N\to \infty$, then    
    \begin{equation}
        \|X_N\|_{\mathrm{op}}\rightarrow 2~~~~  a.s. 
    \end{equation}

\end{theorem}

\begin{proof}

 Choose   $a>1$.  By Lemma \ref{lem:rho_a_a.s.}, we have 
    \begin{equation}
        \limsup\limits_{N\rightarrow \infty}\|X_N\|_{\mathrm{op}}\le \rho_a~~~ a.s.
    \end{equation}
    Take the limit as $a \to 1^+$ and we obtain an upper bound 
    \begin{equation} \label{supbound}
        \limsup\limits_{N\rightarrow \infty}\|X_N\|_{\mathrm{op}}\le 2 ~~~ a.s.
    \end{equation}

On the other hand, it's easy to see from 
  the stochastic  assumption (A2) in Definition \ref{def:inhomo} that   $\sigma_N^*\geq 1/\sqrt{N}$.  So 
  using the condition
    $(r+1)\sigma_N^*\sqrt{\log N}\to 0$, we   derive   
   $r+1= o\big(\sqrt{N/\log N}\big)$. 
    Since     the empirical  spectral measure  of  the IRM  ensemble  $X_N$ almost surely converges  weakly  to the semicircle law as $\sigma_N^*\rightarrow 0$,  as shown    in \cite{gotze2015limit}, 
     the same holds for the deformed IRM ensemble by Weyl’s interlacing inequalities. Therefore, we conclude that
    \begin{equation}
        \liminf\limits_{N\rightarrow \infty}\|X_N\|_{\mathrm{op}}\ge 2 ~~~ a.s.
    \end{equation}
    
Combining this with \eqref{supbound}, we finally obtain the   desired limit.     \end{proof}

\subsection{Concentration of central moments}

A crucial step in proving Theorem \ref{thm:spectral_measure} is to establish upper bounds on the central moments of the power entries $(H_N^t)_{ij}$
  for the sub-Gaussian IRM matrix  $H_N$. To this end, we first reduce the sub-Gaussian case to the Gaussian case using the following comparison lemma, whose   proof will be provided in Appendix \ref{appendix:sec4}.
 
\begin{lemma}\label{lem:Gaussian_moment_dominate}
    For any finite index set    $E$,  let $(Y_e)_{e \in E}$ be  independent symmetric $\gamma$-sub-Gaussian complex  random variables, and let $(Z_e)_{e \in E}$ be independent real centered Gaussian random variables with variance $16\gamma$. Then, for any integers $k(i,e),k'(i,e)\ge 0$ we have
    \begin{equation}
        \mathbb{E}\bigg[\prod_{i=1}^{s}\Big(\prod_{e\in E}Z_{e}^{k(i,e)+k'(i,e)}-\prod_{e\in E}\mathbb{E}[Z_{e}^{k(i,e)+k'(i,e)}]\Big)\bigg]\ge \bigg|\mathbb{E}\bigg[\prod_{i=1}^{s}\Big(\prod_{e\in E}Y_{e}^{k(i,e)}\bar{Y}_{e}^{k'(i,e)}-\prod_{e\in E}\mathbb{E}[Y_{e}^{k(i,e)}\bar{Y}_{e}^{k'(i,e)}]\Big)\bigg]\bigg| .
    \end{equation}
\end{lemma}
We are now ready to establish upper bounds for the central moments of the entries of $H_N^t$ and   $H_{{\rm GOE}_M}^{t}$. 
\begin{proposition}\label{coro:gaussian_bound_sub}
    With  the    sub-Gaussian IRM matrix $H_N$  in Definition  \ref{def:inhomo},   for   any integer $M$ such that   $M(\sigma_N^*)^2\le 1$  and for any integer $s,t>0$, we have 
    \begin{equation}
    \begin{aligned}
    \mathbb{E}\bigg[\bigg|(H_N^{t})_{m_im_j}-\mathbb{E}[(H_N^{t})_{m_im_j}]\bigg|^{2s}\bigg]&\le (16\gamma)^{st}\mathbb{E}\bigg[\bigg|(H_{{\rm GOE}_M}^{t})_{ij}-\mathbb{E}[(H_{{\rm GOE}_M}^{t})_{ij}]\bigg|^{2s}\bigg]\\
    &\le (16\gamma)^{st}\mathbb{E}\bigg[\bigg|(H_{{\rm GOE}_M}^{t})_{11}-\mathbb{E}[(H_{{\rm GOE}_M}^{t})_{11}]\bigg|^{2s}\bigg].
    \end{aligned}
\end{equation}
\end{proposition}

\begin{proof}
We begin with the first inequality. Let us denote a path of length   $t$ by $p^{t}=(p_0p_1\ldots p_t)$ and define    $H_{p^t}=H_{p_0p_1}H_{p_1p_2}\ldots H_{p_{t-1}p_{t}}$. Then the entries of the power matrix may be rewritten as 
    \begin{equation}
        (H_N^{t})_{m_im_j}=\sum_{p^t:m_i\rightarrow m_j} H_{p^t},
    \end{equation}
from which      we get
\begin{equation} \label{productterm}
\begin{aligned}
    \mathbb{E}\bigg[\bigg|(H_N^{t})_{m_im_j}-\mathbb{E}[(H_N^{t})_{m_im_j}]\bigg|^{2s}\bigg]&=\sum_{p^{t,1}, \ldots, p^{t,2s}}\mathbb{E}\bigg[\prod_{i=1}^{s}\Big(\big(H_{p^{t,i}}-\mathbb{E}[H_{p^{t,i}}]\big)\big(\bar{H}_{p^{t,i+s}}-\mathbb{E}[\bar{H}_{p^{t,i+s}}]\big)\Big) \bigg],
\end{aligned}
\end{equation}
where $p^{t,1},p^{t,2},\ldots, p^{t,2s}$ are paths from  $m_i$ to $m_j$.

By Lemma \ref{lem:Gaussian_moment_dominate}, every   product  term  on   the right-hand side of  \eqref{productterm} can be   dominated by the corresponding  term of the  Gaussian IRM matrix $H_N$ with variance $16\gamma$ as given in Definition  \ref{def:inhomo}. Therefore, by extracting the  factor $(16\gamma)^{st}$, it suffices to prove that when $H_{G_N}$ is a real Gaussian IRM matrix, its central moments  can be bounded by those of the GOE, namely,  
    \begin{equation}\label{equ:central_moments_Gaussian IRM}
         \mathbb{E}[|(H_{G_N}^{t})_{m_im_j}-\mathbb{E}[(H_{G_N}^{t})_{m_im_j}]|^{2s}]\le \mathbb{E}\bigg[\bigg|(H_{{\rm GOE}_M}^{t})_{ij}-\mathbb{E}[(H_{{\rm GOE}_M}^{t})_{ij}]\bigg|^{2s}\bigg].
    \end{equation}
    
    To prove \eqref{equ:central_moments_Gaussian IRM}, we repeat the same calculation as in Section \ref{sec:ribbon_expansion} and use the Wick formula to perform the ribbon graph expansion of 
    \begin{equation}
        \mathbb{E}[|(H_{G_N}^{t})_{m_im_j}-\mathbb{E}[(H_{G_N}^{t})_{m_im_j}]|^{2s}]
    \end{equation}
    again. To be precise, we can get
    \begin{equation}\label{equ:expansion_central_moments}
        \mathbb{E}[|(H_{G_N}^{t})_{m_im_j}-\mathbb{E}[(H_{G_N}^{t})_{m_im_j}]|^{2s}]=\sum_{\Upsilon}\sum_{\substack{\eta:V(\Upsilon) \to [N]\\ \eta(v_p)=m_i,\eta(w_p)=m_j}}\prod_{e=(x,y) \in \mathcal{E}(\Upsilon)}\sigma_{\eta(x)\eta(y)}^2.
    \end{equation}
Here    $L_1, \ldots, L_{2s}$ are $2s$ directed $t$-segments with starting vertex $v_p$ and ending vertex $w_p$ for each $L_p$, and the summation is taken over all ribbon graphs $\Upsilon$ obtained by gluing these segments to ribbon graphs such that no $L_p$ is glued entirely to itself. 

Therefore, in order to prove \eqref{equ:central_moments_Gaussian IRM} it suffices to show that for any fixed ribbon graph $\Upsilon$ constructed above,  
\begin{equation}\label{equ:central_moments_Gaussian_diagram_functions}
\sum_{\substack{\eta:V(\Upsilon) \to [N]\\ \eta(v_p)=m_i,\eta(w_p)=m_j}}\prod_{e=(x,y) \in \mathcal{E}(\Upsilon)}\sigma_{\eta(x)\eta(y)}^2 \le \sum_{\substack{\eta:V(\Upsilon) \to [M]\\ \eta(v_p)=i,\eta(w_p)=j}}M^{-|\mathcal{E}(\Upsilon)|}.
\end{equation} 
For the ribbon graph $\Upsilon$,  we choose a spanning forest $F$ such that each connected component contains exactly one of the vertices $v_p$ and $w_p$ $(1\le p\le 2s)$.  For   edges not in $F$ we bound the corresponding weights by  
\begin{equation}\label{equ:forest}
   \prod_{e=(x,y) \notin F}\sigma_{\eta(x)\eta(y)}^2 \le (\sigma^*_N)^{2(|\mathcal{E}(\Upsilon)|-|F|)} \le M^{-(|\mathcal{E}(\Upsilon)|-|F|)}.
\end{equation}
where the last inequality follows from the assumption $(\sigma_N^*)^2 \le 1/M$. Hence,
\begin{equation}\label{equ:inequ_forest}
    \sum_{\substack{\eta:V(\Upsilon) \to [N]\\ \eta(v_p)=m_i,\eta(w_p)=m_j}}\prod_{e=(x,y) \in \mathcal{E}(\Upsilon)}\sigma_{\eta(x)\eta(y)}^2 \le M^{-(|\mathcal{E}(\Upsilon)|-|F|)} \sum_{\substack{\eta:V(F) \to [N]\\ \eta(v_p)=m_i,\eta(w_p)=m_j}}\prod_{e=(x,y) \in \mathcal{E}(F)}\sigma_{\eta(x)\eta(y)}^2
\end{equation}
For edges in $F$, we use the double stochastic property of the transition matrix $\Sigma_N$ to obtain
\begin{equation}\label{equ:IRM_double_stochastic}
\sum_{\substack{\eta:V(F) \to [N]\\ \eta(v_p)=m_i,\eta(w_p)=m_j}}\prod_{e=(x,y) \in \mathcal{E}(F)}\sigma_{\eta(x)\eta(y)}^2=1.
\end{equation}
We note that equality in \eqref{equ:forest} holds when all $\sigma_{\eta(x)\eta(y)}$ are equal. In particular,
\begin{equation}\label{equ:GOE_forest}
  \sum_{\substack{\eta:V(\Upsilon) \to [M]\\ \eta(v_p)=i,\eta(w_p)=j}}M^{-|\mathcal{E}(\Upsilon)|}=  M^{-(|\mathcal{E}(\Upsilon)|-|F|)}.
\end{equation}
Equation \eqref{equ:central_moments_Gaussian_diagram_functions} follows directly from \eqref{equ:inequ_forest}, \eqref{equ:IRM_double_stochastic}, and \eqref{equ:GOE_forest}. This completes the proof of the first inequality.

 Next, we turn to  the proof of  the second inequality.   For any ribbon graph $\Upsilon$, let $\mathcal{S}_{\Upsilon,ij}$ denote the set of vertex labelings $\eta$ such that
$\eta(v_p)=i$ for all   starting vertex 
$v_p$ and $\eta(w_p)=j$ for all ending vertex  $w_p$, namely
\begin{equation}
    \mathcal{S}_{\Upsilon,ij}=\{\eta:V(\Upsilon) \to [M]:\ \eta(v_p)=i,\eta(w_p)=j \text{ for all } 1 \le p \le 2s\}.
\end{equation}
Applying \eqref{equ:expansion_central_moments} to $H_{{\rm GOE}_M}$, we obtain  
\begin{equation}\label{equ:expansion_cetral_moments_GOE}
\mathbb{E}\bigg[\bigg|(H_{{\rm GOE}_M}^{t})_{ij}-\mathbb{E}[(H_{{\rm GOE}_M}^{t})_{ij}]\bigg|^{2s}\bigg]=\sum_{\Upsilon}M^{-|E(\Upsilon)|} |\mathcal{S}_{\Upsilon,ij}|.
\end{equation}

Now for any ribbon graph $\Upsilon$, we construct an injective map from $\mathcal{S}_{\Upsilon,ij}$ to $\mathcal{S}_{\Upsilon,11}$ as follows: for any $\eta\in \mathcal{S}_{\Upsilon,ij}$, we modify the values of $\eta$ on each $v_p$ and $w_p$ to 1 and then get a labeling in $\mathcal{S}_{\Upsilon,11}$. 
Since the values of $\eta$
  at all other vertices remain unchanged, the map is injective. Hence, for any $i,j$,
 \begin{equation}\label{equ:S_ij}
    |\mathcal{S}_{\Upsilon,ij}| \le |\mathcal{S}_{\Upsilon,11}|.
\end{equation} 
 
Substituting \eqref{equ:S_ij} into \eqref{equ:expansion_cetral_moments_GOE} completes the proof of the second inequality. \end{proof}

The following  upper bounds for traces of the GOE model play a central role in this paper and are also of independent interest.

\begin{lemma}\label{Lem:moment_concentrate}
For any fixed integer  $ t\ge 1$, 
   there exists a  constant $C_t>0$  that depends only on  $t$  such that for any positive  integer  $s\leq M/C_t$, 
\begin{equation}
    \mathbb{E}\bigg[\bigg|\frac{1}{M}\tr H^{t}_{{\rm GOE}_M}-\frac{1}{M}\mathbb{E}[\tr H^{t}_{{\rm GOE}_M}]\bigg|^{2s}\bigg]\le \Big(\frac{C_t s }{M}\Big)^{2s}.
\end{equation}

\end{lemma}
\begin{proposition}\label{Prop:Up-bounds-G_M}
   For any fixed integer  $ t\ge 1$, 
   there exists a  constant $C_t>0$  that depends only on  $t$  such that for any positive  integer  $s\leq M/C_t$      and  for any $i,j\in [M]$,  
    \begin{equation}
    \mathbb{E}\bigg[\bigg|(H_{{\rm GOE}_M}^{t})_{11}-\mathbb{E}[(H_{{\rm GOE}_M}^{t})_{11}]\bigg|^{2s}\bigg]\le M\Big(\frac{C_t s}{M}\Big)^s.
    \end{equation}
\end{proposition}

\begin{proof}[Proof of Lemma   \ref{Prop:Up-bounds-G_M}]
By the Wick formula,  we first get  (possibly disconnected) ribbon graph expansion and then  classify  the ribbon graphs into  connected components $P$, like  
\begin{equation}\label{equ:4.33T}
    \mathbb{E}\bigg[\bigg|\frac{1}{M}\tr H^{t}_{{\rm GOE}_M}-\frac{1}{M}\mathbb{E}[\tr H^{t}_{{\rm GOE}_M}]\bigg|^{2s}\bigg]= \sum_{\pi \in \mathcal{P}'(2s)}\prod_{P \in \pi}\frac{1}{M^{|P|}}T\underbrace{(t,\dots, t)}_{|P|}.
\end{equation}
Here $\mathcal{P}'(2s)$ is the set of all partitions of $[2s]$ such that each block in the partition contains at least two elements and 
$T(t,\dots, t)$,  which consists of 
 $|P|$  components,          is given by
\begin{equation}
   T(t,\dots, t) =\sum_{\Upsilon}M^{V(\Upsilon)-E(\Upsilon)-2s}
\end{equation}
where the summation is over all  
 ribbon graphs \emph{constructed by gluing $|P|$ $t$-gons} so that each connected component is formed by gluing the polygons corresponding to a block of the partition $\pi$.

For $T(t,\dots, t)$, which consists of $k$ components, we now establish an upper bound using Okounkov's contraction method. Namely, we reduce each ribbon graph $\Upsilon$ to a diagram $\Gamma=(V(\Gamma),E(\Gamma))$, as described in Section \ref{sec:ribbon_expansion}.
Similar to Proposition \ref{prop:reduce to 3}, it suffices to consider trivalent diagrams $\bar{\Gamma}$, at the cost of introducing a factor of $2^{E(\bar{\Gamma})}$.

We then apply \cite[Proposition II.3.3]{feldheim2010universality} to bound the number of such trivalent diagrams.  
To be precise, let $g=E(\Gamma)-V(\Gamma)+k=E(\bar{\Gamma})-V(\bar{\Gamma})+k$, which counts the number of steps required to generate $\Gamma$ in the automation constructed by Feldheim and Sodin. Since $\bar{\Gamma}$ is trivalent,
we obtain the parameterization  $|E(\bar{\Gamma})|=3g-k$ and $|V(\bar{\Gamma})|=2g$ via the genus $g$. For any    $g$, it follows from \cite[Proposition II.3.3]{feldheim2010universality} that  there is a universal constant $C$ such that    the number of trivalent $k$-diagrams generated in $g$ steps, denoted by  $D_{k}(g)$,   

    \begin{equation}
    D_{k}(g)\le \frac{(Cg)^{g+k-1}}{(k-1)!}.
    \end{equation}

Since we just consider the $t$-th moment, the size of trees for each edge does not exceed $t$. Thus the number of possible tree configurations on each edge is bounded by a constant $C(t)$. So for any fixed diagram $\Gamma$, the number of its pre-images (i.e. un-reduced ribbon graphs $\Upsilon$ corresponding to $\Gamma$)  is at most
    \begin{equation}
        M^{|V(\Gamma)|-|E(\Gamma)|}C(t)^{|E(\Gamma)|}=M^{-g}C(t)^{|E(\Gamma)|} \le M^{k-g}C(t)^{3g-k},
    \end{equation}
 where we use the inequality $|E(\Gamma)| \le |E(\bar{\Gamma})| =3g-k$. Since a non-degenerate trivalent $k$-diagram 
  is generated by  at least $k$ steps in the Feldheim-Sodin   automation, we see that $g \ge k$.  Note that $3g-k=|E(\bar{\Gamma})| \le 2kt$, we have 
    \begin{equation}
    \begin{aligned}
        \frac{1}{M^k}T\underbrace{(t,\dots, t)}_{k}&\le \sum_{kt\ge g \ge k}D_{k}(g)2^{3g-k}M^{-g}C(t)^{E}\le  \frac{1}{(k-1)!}\sum_{kt\ge g \ge k}\frac{{(8CC(t)^3g)}^{g+k-1}}{M^g}\\       &\le \frac{1}{(k-1)!}\sum_{g \ge k}\frac{{(8CC(t)^3\cdot t k)}^{g+k-1}}{M^g}\le \sum_{g \ge k}\frac{(C'(t)k)^{g}}{M^g}\le \big(\frac{C''(t)k}{M}\big)^k,
    \end{aligned}
    \end{equation}
    where in the last line we need $M>C'(t)k$.
Hence by \eqref{equ:4.33T}, we have
\begin{equation} \label{goesum}
    \begin{aligned}   &\mathbb{E}\bigg[\bigg|\frac{1}{M}\tr H^{t}_{{\rm GOE}_M}-\frac{1}{M}\mathbb{E}[\tr H^{t}_{{\rm GOE}_M}]\bigg|^{2s}\bigg]
    \le \sum_{\pi \in \mathcal{P}'(2s)}\prod_{P \in \pi}\Big(\frac{C''(t)|P|}{M}\Big)^{|P|}\\
    &\le \Big(\frac{C'''(t)}{M}\Big)^{2s}\sum_{\pi \in \mathcal{P}'(2s)}\prod_{P \in \pi}{(|P|)!},
    \end{aligned}
\end{equation}
where in the last line we use $\sum_{P\in \pi}|P|=2s$, and $|P|^{|P|}\le C^{|P|}|P|!$. 

The above  sum     can be bounded by $(2s)^{2s}$.   In fact,  let $i_l$ be the number of $P\in \pi$ such that $|P|=l$. Since   the number of such partition is $(2s)!/\big(\prod_{l\ge 2}(i_l)!(l!)^{i_l}\big)$,   we have
\begin{equation}
    \begin{aligned}
       &\sum_{\pi \in \mathcal{P'}(2s)}\prod_{P \in \pi}(|P|)!=(2s)!\sum_{\sum_{l\ge 2}l\cdot i_l=2s}\prod_{l\ge 2}\frac{1}{(i_l)!}.
     \end{aligned}
\end{equation}
Consider the generating function  
\begin{equation}
    G(z)=\sum_{s=0}^{\infty}\sum_{\sum_{l\ge 2}l\cdot i_l=s}\prod_{l\ge 2}\frac{1}{(i_l)!}z^{s}=e^{\sum_{l\ge 2}z^l}=e^{\frac{z^2}{1-z}}.
\end{equation}
 Note that   the coefficient of $ z^2/(1-z)$ is dominated by that  of $-\ln(1-2z)$ term-by-term,    we obtain 
\begin{equation}\label{equ:P'(2s)}
    \sum_{\pi \in \mathcal{P'}(2s)}\prod_{P \in \pi}(|P|)!=(2s)!\big[e^{\frac{z^2}{1-z}}\big]_{2s}\le (2s)![e^{-\ln(1-2z)}]_{2s}=2^{2s}(2s)!\le (4s)^{2s}.
\end{equation}
     Therefore, the desired result follows from \eqref{goesum}. 
\end{proof}

\begin{proof}[Proof of Proposition \ref{Prop:Up-bounds-G_M}]
We introduce a spectral decomposition $H_{{\rm GOE}_M}=\sum_{i=1}^{M} \lambda_i \mathbf{v}_i\mathbf{v}^T_{i}$  with $v_{i1}$ as the  first  component  of  eigenvector $\mathbf{v}_i$. 
Obviously,
\begin{equation}
    (H_{{\rm GOE}_M}^t)_{11}= \sum_{i=1}^{M} \lambda_i^t v_{i1}^2,
\end{equation}
  from which we get 
\begin{equation}\label{equ:lln_G}
\begin{aligned}
    &\mathbb{E}\Big[\bigg|(H_{{\rm GOE}_M}^{t})_{11}-\mathbb{E}[(H_{{\rm GOE}_M}^{t})_{11}]\bigg|^{2s}\Big]=\mathbb{E}\bigg[\Big|\sum_{i=1}^{M} \lambda_i^t v_{i1}^2-\mathbb{E}\Big[\sum_{i=1}^{M} \lambda_i^t v_{i1}^2\Big]\Big|^{2s}\bigg].
\end{aligned}
\end{equation} 
By the rotational invariance of the GOE, we can take the expectation in order as $\mathbb{E}[\cdot]=\mathbb{E}_{\lambda}[\mathbb{E}_{v_i}[\cdot]],$ 
 where the   vector $\mathbf{v}=(v_{11}, \ldots,v_{M1})^{T}$ is   taken uniformly from the unit sphere  $\mathbb{S}^{M-1}$.  It's easy to see that 
\begin{equation}\label{equ:lln_divide}
    \begin{aligned}
    &\mathbb{E}\bigg[\bigg|\sum_{i=1}^{M} \lambda_i^t v_{i1}^2-\mathbb{E}\bigg[\sum_{i=1}^{M} \lambda_i^t v_{i1}^2\bigg]\bigg|^{2s}\bigg]
    =\mathbb{E}\bigg[\bigg|\sum_{i=1}^{M} \lambda_i^t \big(v_{i1}^2-\frac{1}{M}\big)+\frac{1}{M}\sum_{i=1}^{M} \big(\lambda_i^t- \mathbb{E}\big[ \lambda_i^t\big]\big)\bigg|^{2s}\bigg]\\
    &\le 2^{2s}\mathbb{E}\bigg[\bigg|\sum_{i=1}^{M} \lambda_i^t \big(v_{i1}^2-\frac{1}{M}\big)\bigg|^{2s}\bigg]+2^{2s}\mathbb{E}\bigg[\bigg|\frac{1}{M}\tr H_{{\rm GOE}_M}^t-\mathbb{E}\bigg[\frac{1}{M}\tr H_{{\rm GOE}_M}^t\bigg]\bigg|^{2s}\bigg].
    \end{aligned}
\end{equation}
The second term can be  bounded by Lemma \ref{Lem:moment_concentrate},  so it suffices to     bound the first term. 

Using  a representation of the random vector  $\mathbf{v}=\mathbf{Y}/R$ where all components  $Y_i$ are  i.i.d. normal variables with distribution      $ \mathcal{N}(0,1/M)$  and  $R^2=\sum_{i=1}^M Y_i^2$,   we  then obtain   
\begin{equation}
    \begin{aligned}
        \mathbb{E}_{\mathbf{Y}}\bigg[\bigg|\sum_{i=1}^{M} \lambda_i^t \Big(Y_{i}^2-\frac{1}{M}\Big)\bigg|^{2s}\bigg]&=\mathbb{E}_{\mathbf{v}}\bigg[\mathbb{E}_{R}\bigg[\bigg|R^2\sum_{i=1}^{M} \lambda_i^t v_{i1}^2-\frac{1}{M}\sum_{i=1}^{M} \lambda_i^t\bigg|^{2s}\bigg]\bigg]\\
        &\ge \mathbb{E}_{\mathbf{v}}\bigg[\bigg|\sum_{i=1}^{M} \lambda_i^t v_{i1}^2-\frac{1}{M}\sum_{i=1}^{M} \lambda_i^t\bigg|^{2s}\bigg],
    \end{aligned}
\end{equation}
where the last inequality  comes from $\mathbb{E}[R^2]=1$ and  the convexity of the function $|ax+b|^{2s}$.   Now we  expand the $2s$-moment   into
\begin{equation}
    \begin{aligned}
        &\mathbb{E}_{\mathbf{Y}}\bigg[\bigg|\sum_{i=1}^{M} \lambda_i^t \Big(Y_{i}^2-\frac{1}{M}\Big)\bigg|^{2s}\bigg]=\sum_{i_{k}\in [M], k=1,\ldots 2s} \Big(\prod_{k=1}^{2s} \lambda^t_{i_k} \Big)\mathbb{E}\bigg[\prod_{k=1}^{2s}\big(Y_{i_{k}}^2-\frac{1}{M}\big)\bigg].
    \end{aligned}
\end{equation}
By  Lemma \ref{lem:Gaussian_moment_dominate}   we know that  all $\mathbb{E}\big[\prod_{k=1}^{2s}(Y_{i_{k}}^2-M^{-1})\big]$ are  non-negative, so by the simple inequality  $\prod_{k=1}^{2s} \lambda^t_{i_k}\le |\lambda|_{\mathrm{max}}^{2st}$  we have  
\begin{equation} \label{upperb-0}
    \begin{aligned}
        &\mathbb{E}\bigg[\bigg|\sum_{i=1}^{M} \lambda_i^t \Big(Y_{i}^2-\frac{1}{M}\Big)\bigg|^{2s}\bigg]\le  \mathbb{E}\big[|\lambda|_{\mathrm{max}}^{2st}\big]\,\mathbb{E}\bigg[\bigg|\sum_{i=1}^{M} \Big(Y_{i}^2-\frac{1}{M}\Big)\bigg|^{2s}\bigg]=\mathbb{E}\big[|\lambda|_{\mathrm{max}}^{2st}\big] \,\mathbb{E}\bigg[\Big|\frac{1}{M}\chi_M^{2}-1\Big|^{2s}\bigg].
    \end{aligned}
\end{equation}
Here $\chi_M^{2}$ is the chi square distribution with freedom  $M$, whose   moment generating function   is  given by 
\begin{equation}
    \big(1-\frac{2t}{M}\big)^{-\frac{M}{2}}e^{-t}=e^{-t-\frac{M}{2}\ln (1-\frac{2t}{M})}.
\end{equation}

On one hand, notice the simple inequality 
\begin{equation}
    -t-\frac{M}{2}\ln \big(1-\frac{2t}{M}\big)\le \frac{t^2}{M-2t}, \quad 0\leq t
\leq \frac{M}{2},\end{equation}
consider the coefficient of $t^{2s}$ in the  above generating function  and we get 
\begin{equation}\label{upperb-1}
    \begin{aligned}
        \mathbb{E}\Big[\Big|\frac{1}{M}\chi_M^{2}-1\Big|^{2s}\Big]&\le (2s)! \Big[e^{ \frac{t^2}{M-2t}} \Big]_{2s}=(2s)!\sum_{i=0}^{s}\frac{1}{i! M^{i}}\big(\frac{2}{M}\big)^{2s-2i}\binom{i+2s-2i-1}{2s-2i}\\
        &= 2^{2s}\sum_{i=0}^{s}\frac{(2s-i)!}{M^{2s-i}}\binom{2s}{i}\binom{2s-i-1}{2s-2i}\le 4^{3s}\sum_{i=0}^s\frac{(2s-i)!}{M^{2s-i}}\\
        &\le 2\Big(\frac{64s}{M}\Big)^s,
    \end{aligned}
\end{equation}
where   the condition $s\le M/{2}$ has been used in the last line. 

On the other hand,
by\cite[Eq(I.5.6)]{feldheim2010universality} again we have
\begin{equation}
    \label{upperb-2}\mathbb{E}\big[|\lambda|_{\mathrm{max}}^{2st}\big]\le \mathbb{E}[\operatorname{Tr}H_{{\rm GOE}_M}^{2st}]\le \frac{CM}{(\sqrt{st})^3}e^{C\frac{(st)^3}{M^2}}\le M(C_t)^s,
\end{equation}
whenever  $s\le M/2$.
Combining   \eqref{upperb-1} and  \eqref{upperb-2},  we see from  \eqref{upperb-0} that 
\begin{equation}
\begin{aligned}
        &\mathbb{E}\bigg[\bigg|\sum_{i=1}^{M} \lambda_i^t \Big(v_{i1}^2-\frac{1}{M}\Big)\bigg|^{2s}\bigg]\le M\Big(\frac{C_t s}{M}\Big)^s.
\end{aligned}
\end{equation}
This shows that the first term  on the right-hand side of  \eqref{equ:lln_divide} has  the desired bound.   We thus complete the proof. 
\end{proof}

\subsection{Proof of Theorem \ref{thm:LLN1} }\label{sec:proof_lln}

 We now proceed to prove Theorem \ref{thm:spectral_measure}, which characterizes the projection measure for   $X_N$. This result will subsequently be used to establish Theorem \ref{thm:LLN1}.

\begin{proof}[Proof of Theorem \ref{thm:spectral_measure}]  
    Noting that  $(r+1)\sigma_N^* \sqrt{\log N}\rightarrow 0$ as $N\to \infty$,    taking $M=\lfloor (\sigma^*_N)^{-2} \rfloor$ and  $s=\log N$,  
    by Propositions \ref{coro:gaussian_bound_sub} and   \ref{Prop:Up-bounds-G_M} we have
    \begin{equation}\label{equ:concentration_H}
    \begin{aligned}
       & \mathbb{P}\Big(\bigg|(H_N^{t})_{m_im_j}-\mathbb{E}[(H_N^{t})_{m_im_j}]\bigg|>\frac{\epsilon}{r+1}\Big)\le \mathbb{E}\bigg[\bigg|(H_N^{t})_{m_im_j}-\mathbb{E}[(H_N^{t})_{m_im_j}]\bigg|^{2s}\bigg]\cdot \Big(\frac{r+1}{\epsilon}\Big)^{2s}\\
        &\le \big(C_t\epsilon^{-2} (r+1)^2(\sigma_N^*)^2\log N\big)^{s}\le C_{\epsilon,t}N^{-1958} 
    \end{aligned}
    \end{equation}
    for sufficiently large $N$.    By the  Borel-Cantelli lemma,  we get 
\begin{equation}\label{io}
    \begin{aligned}
        \mathbb{P}\Big(\max_{1\le i,j \le r}\bigg|(H_N^{t})_{m_im_j}-\mathbb{E}[(H_N^{t})_{m_im_j}]\bigg|>\frac{\epsilon}{r+1} \ i.o. \Big) =0.
    \end{aligned}
    \end{equation}

    Since $X_N^{t}$ can be expressed as a sum of mixed non-commutative products of $A_N^{i_l}$, $H_{N}^{j_l}$, namely, 
    \begin{equation}
        X_{N}^t=A_N^{t}+\sum_{n=1}^{t}\sum_{\substack{\sum_{l=1}^{n}(i_l+j_l)=t\\ \sum_{l\ge 1}j_l\geq 1}}\prod_{l=1}^{n} A_N^{i_l}H_{N}^{j_l},
    \end{equation}
 where the product of matrices is defined by $
\prod_{k=1}^{d} M_k = M_1 M_2 \cdots M_d$. There are at most $2^t$ terms in total. Pick up a term and replace $H_{N}^{j_l}$ by $\mathbb{E}[(H_N^{j_l})]$ one by one, we get
\begin{equation}
    \prod_{l=1}^n A_N^{i_l}H_{N}^{j_l}-\prod_{l=1}^n A_N^{i_l}\mathbb{E}[H_{N}^{j_l}]=\sum_{p=1}^{n} \Big(\prod_{l=1}^{p-1} A_N^{i_l}H_{N}^{j_l} \Big)A_N^{i_{p}} \big(H_{N}^{j_{p}}-\mathbb{E}[H_{N}^{j_{p+1}}]\big)\prod_{l=p+1}^n A_N^{i_l}\mathbb{E}[H_{N}^{j_l}].
\end{equation}
By considering the positions of the non-zero entries of  $A_N$, we introduce   the reduced matrices $\Tilde{A}$ and $\Tilde{H}_t$  of size $r\times r$ whose $(i,j)$ entry  $\Tilde{A}_{ij}=(A_N)_{m_im_j}$ and  $(\Tilde{H}_t)_{ij}=(H_{N}^{t})_{m_im_j}$ respectively. Then  
\begin{equation}
\begin{aligned}
    &\Big(\prod_{l=1}^n A_N^{i_l}H_{N}^{j_l}-\prod_{l=1}^n A_N^{i_l}\mathbb{E}[H_{N}^{j_l}]\Big)_{m_im_j}=\bigg(\sum_{p=1}^{n} \prod_{l=1}^s \Tilde{A}^{i_l}\Tilde{H}_{j_l} \prod_{l=p+1}^n \Tilde{A}^{i_l}(\Tilde{H}_{j_l}-\mathbb{E}[\Tilde{H}_{j_l}])\bigg)_{ij}.
\end{aligned}
\end{equation}
Now by \eqref{io}, we can always condition on \begin{equation}
\bigg|(\Tilde{H}_t)_{ij}-\mathbb{E}[(\Tilde{H}_{t})_{ij}]\bigg|\le \frac{\epsilon}{r+1},  \quad \forall 1\le i,j\le r,
\end{equation}
 which is an almost sure event as $N\rightarrow\infty$ and implies 
\begin{equation}
    \|\Tilde{H}_{j_l}-\mathbb{E}[\Tilde{H}_{j_l}]\|_{\mathrm{op}}\le \epsilon.
\end{equation}

Combining  the simple facts  that 
    \begin{equation}
        \|\mathbb{E}[H_{N}^{2m}]-\frac{1}{m+1}\binom{2m}{m}\mathbf{I}_N\|_{\mathrm{op}}=O((\sigma_N^*)^2),\quad \mathbb{E}[H_{N}^{2m+1}]=0
    \end{equation}
    and  all  the non-zero components of  $\mathbf{q}_i$  lying in  the $m_1$-, $\ldots$,$m_{r}$- rows and columns, 
 we  thus arrive at  
\begin{equation}
\begin{aligned}
    &\bigg|\Big\langle\big(\prod_{l=1}^n A_N^{i_l}H_{N}^{j_l}-\prod_{l=1}^n A_N^{i_l}\mathbb{E}[H_{N}^{j_l}]\big)\mathbf{q}_i,\mathbf{q}_i\Big\rangle \bigg|=\bigg|\Big\langle(\prod_{l=1}^n \Tilde{A}_{i_l}\Tilde{H}_{j_l}-\prod_{l=1}^n \Tilde{A}_{i_l}\mathbb{E}[\Tilde{H}_{j_l}])\mathbf{q}_i,\mathbf{q}_i\Big\rangle \bigg|\\
    &\le \Big\|\prod_{l=1}^n \Tilde{A}_{i_l}\Tilde{H}_{j_l}-\prod_{l=1}^n \Tilde{A}_{i_l}\mathbb{E}[\Tilde{H}_{j_l}]\Big\|_{\mathrm{op}}\le \bigg\|\sum_{p=1}^{n} \prod_{l=1}^p \Tilde{A}^{i_l}\Tilde{H}_{j_l} \prod_{l=p+1}^n \Tilde{A}^{i_l}(\Tilde{H}_{j_l}-\mathbb{E}[\Tilde{H}_{j_l}])\bigg\|_{\mathrm{op}}\\
    &\le C_{t,a}\epsilon
\end{aligned}
\end{equation}
for some constant $C_{t,a}$.

Set 
\begin{equation}
    \bar{X}_{N,t}=
    A_N^{t}+\sum_{\sum_{l\ge 1}(i_l+j_l)=t}\prod_{l} A_N^{i_l}\mathbb{E}[H_{N}^{j_l}].
\end{equation}
Summing up  all at most $2^t$ mixed  products shows 
    \begin{equation}
        \bigg|\langle(X_N^{t}-\bar{X}_{N,t})\mathbf{q}_i,\mathbf{q}_i\rangle \bigg|\le C'_{t,a}\epsilon.
    \end{equation}
    The final piece of the puzzle lies in some explicit calculations.    
    Standard random matrix theory operations yield
\begin{equation}\langle\bar{X}_{N,t}\mathbf{q}_i,\mathbf{q}_i\rangle= \sum_{j=0}^{t}C(j,t)\langle A_N^{t-j}\mathbf{q}_i,\mathbf{q}_i\rangle+O_t((\sigma_N^*)^2)=\sum_{j=0}^{t}C(j,t) a_i^{t-j}+o(1),
    \end{equation}
    where $C(j,t)$ are constants only  depending on   $t,j$ (c.f.  Example \ref{example:degenerate} for explicit calculations). Hence we have
    \begin{equation}
         \lim\limits_{N\rightarrow\infty}\langle X_N^{t}\mathbf{q}_i,\mathbf{q}_i\rangle= \sum_{j=0}^{t}C(j,t) a_i^{t-j}~~~~{a.s.},~~~\forall\, 1\le i\le r.
    \end{equation}
    The sum on the  right-hand side is exactly  the $t$-th moment of $\mu_{a_i}$  as shown by  in \cite{noiry2021spectral}   and \cite[Theorem 1.3]{Au23BBP}. 

Finally, since $\mu_{a_i}$ is compactly supported and therefore uniquely determined by its moments, the convergence of moments implies weak convergence. Consequently, we conclude that the spectral measure converges weakly to $\mu_{a_i}$ almost surely.
\end{proof}

\begin{proof}[Proof of Theorem \ref{thm:LLN1}]

We establish Theorem \ref{thm:LLN1} by integrating the absence of outliers (Theorem \ref{thm:no_outlier}) and the convergence of projection measures (Theorem \ref{thm:spectral_measure}). By using Weyl’s interlacing inequality and min-max principle  to sandwich the eigenvalues, we adapt the inductive framework of \cite{Au23BBP} to our setting.

\paragraph{Step 1: Base case ($j=1$).} We first establish the convergence for the extreme eigenvalues, verifying that \eqref{equ:lambdaj} and \eqref{equ:lambdaN-j} hold for $j=1$.

To simplify the exposition, we assume that all deformation eigenvalues satisfy $|a_i| > 1$, as the case $|a_i| \le 1$ follows from analogous arguments. We assume the spikes are ordered such that $a_1 \ge a_2 \ge \dots > 1$ and $a_{-1} \le a_{-2} \le \dots < -1$.

We first establish the convergence for the smallest eigenvalue $\lambda_N(X_N)$. Let $X_N^-$ be the matrix formed by incorporating all rank-one perturbations $a_i \mathbf{q}_i \mathbf{q}_i^T$ such that $|a_i| \le |a_{-1}|$. Note that by construction, $a_{-1} \mathbf{q}_{-1} \mathbf{q}_{-1}^T$ is the most extreme negative perturbation in $X_N^-$. Applying the spectral norm bound from Lemma \ref{lem:rho_a_a.s.}, we obtain
\begin{equation}
    \limsup_{N \to \infty} \|X_N^-\|_{\mathrm{op}} \le |\rho_{a_{-1}}| + \epsilon \quad \text{a.s.}
\end{equation}
Since $\rho_{a_{-1}} < 0$, this operator norm bound directly implies a lower bound for the smallest eigenvalue:
\begin{equation} \label{smallestlow_refined}
    \liminf_{N \to \infty} \lambda_N(X_N^-) \ge \rho_{a_{-1}} - \epsilon \quad \text{a.s.}
\end{equation}

Next, we incorporate the remaining positive deformations into $X_N^-$ to recover the full matrix $X_N$. According to Weyl’s interlacing inequality, the smallest eigenvalue is non-decreasing under the addition of positive semi-definite matrices (the rank-one perturbations $a_i \mathbf{q}_i \mathbf{q}_i^T$ for $a_i > 1$). Thus, \eqref{smallestlow_refined} implies
\begin{equation} \label{smallestlow_final}
    \liminf_{N \rightarrow \infty} \lambda_{N}(X_N) \ge \rho_{a_{-1}} - \epsilon \quad \text{a.s.}
\end{equation}

To establish the matching upper bound, we employ the min-max principle. For any $k \in \mathbb{N}$, we have
\begin{equation} \label{equ:4.64_refined}
    \lambda_{N}(X_N)^{2k+1} = \min_{\dim(V)=1} \max_{x \in V, \|x\|=1} \langle x, X_N^{2k+1} x \rangle.
\end{equation}
By choosing the test subspace $V = \mathrm{span}(\mathbf{q}_{-1})$, we obtain
\begin{equation} \label{equ:4.65_refined}
    \lambda_{N}(X_N)^{2k+1} \le \langle \mathbf{q}_{-1}, X_N^{2k+1} \mathbf{q}_{-1} \rangle.
\end{equation}

For any fixed $k$, Theorem \ref{thm:spectral_measure} ensures that the right-hand side of \eqref{equ:4.65_refined} converges almost surely to the $(2k+1)$-th moment of the projection measure $\mu_{a_{-1}}$. Since the support of $\mu_{a_{-1}}$ is bounded above by $\rho_{a_{-1}}$, by taking $k$ sufficiently large, we have
\begin{equation}
    \limsup_{N \rightarrow \infty} \lambda_{N}(X_N) \le \rho_{a_{-1}} + \epsilon \quad \text{a.s.}
\end{equation}
Combining this with \eqref{smallestlow_final} and letting $\epsilon \to 0$, we conclude that $\lim_{N \to \infty} \lambda_N(X_N) = \rho_{a_{-1}}$ a.s. 

A symmetric argument applies to the largest eigenvalue $\lambda_1(X_N)$. Specifically, we obtain an upper bound for $\lambda_1$ using Lemma \ref{lem:rho_a_a.s.} and a lower bound via the max-min principle with the test vector $\mathbf{q}_1$, yielding
\begin{equation}
    \lim_{N \rightarrow \infty} \lambda_{1}(X_N) = \rho_{a_{1}} \quad \text{a.s.}
\end{equation}

\paragraph{Step 2: Inductive step} 
Assume the inductive hypothesis that Theorem \ref{thm:LLN1} (specifically limits \eqref{equ:lambdaj} and \eqref{equ:lambdaN-j}) holds for all $1 \le j \le n$. We now show that the result remains valid for $j = n+1$.

Without loss of generality, assume $a_1 \ge |a_{-1}| > 1$. To isolate the effect of the $(n+1)$-th negative spike, we construct an auxiliary matrix $X_N^{-0}$ by incorporating all rank-one perturbations $a_i \mathbf{q}_i \mathbf{q}_i^T$ satisfying $|a_i| \le |a_{-1}|$, with the specific exclusion of the most extreme negative spike $a_{-1} \mathbf{q}_{-1} \mathbf{q}_{-1}^T$. 

By the inductive hypothesis, $X_N^{-0}$ has its $n$ smallest eigenvalues determined by the remaining negative spikes $\{a_{-2}, a_{-3}, \dots, a_{-(n+1)}\}$. Specifically, for $i = 1, \dots, n$, we have
\begin{equation} \label{equ:ind_hyp_X0}
    \limsup_{N \to \infty} |\lambda_{N-i+1}(X_N^{-0}) - \rho_{a_{-i-1}}| \le \epsilon \quad \text{a.s.}
\end{equation}
Note that the index shift $a_{-i-1}$ arises because the ``first" negative spike in $X_N^{-0}$ is actually $a_{-2}$. 

Next, we recover $X_N^-$ by adding back the perturbation $a_{-1} \mathbf{q}_{-1} \mathbf{q}_{-1}^T$. Since $a_{-1} < 0$, this is a negative rank-one perturbation. By Weyl's interlacing inequality, the eigenvalues of $X_N^-$ and $X_N^{-0}$ satisfy
\begin{equation} \label{equ:weyl_interlace_ind}
    \lambda_{N-i+1}(X_N^-)\ge \lambda_{N-i+2}(X_N^{-0}) \ge \rho_{a_{-i}} - \epsilon, \quad \text{for } i = 2, \dots, n+1.
\end{equation}

Furthermore, the spectral norm bound from Lemma \ref{lem:rho_a_a.s.} ensures that the smallest eigenvalue $\lambda_N(X_N^-)$ cannot drop below $\rho_{a_{-1}} - \epsilon$. Combining these interlacing relations with \eqref{equ:ind_hyp_X0}, we conclude that for all $1 \le i \le n+1$:
\begin{equation}
    \liminf_{N \to \infty} \lambda_{N-i+1}(X_N^-) \ge  \rho_{a_{-i}} - \epsilon \quad \text{a.s.}
\end{equation}

Next, we incorporate all remaining positive rank-one perturbations into $X_N^-$, excluding the largest spike $a_1 \mathbf{q}_1 \mathbf{q}_1^T$, and denote the resulting matrix by $X_N^+$. Since these are positive semi-definite perturbations, Weyl’s interlacing inequality implies that the eigenvalues are non-decreasing. Specifically, for the smallest eigenvalues, we have
\begin{equation} \label{equ:X_minus_to_plus}
    \lambda_{N-i+1}(X_N^-) \le \lambda_{N-i+1}(X_N^+), \quad  \forall \,1 \le i \le n+1.
\end{equation}
By the inductive hypothesis applied to $X_N^+$, the positions of these eigenvalues are determined by the corresponding negative spikes, yielding
\begin{equation}
    \limsup_{N \to \infty} |\lambda_{N-i+1}(X_N^+) - \rho_{a_{-i}}| \le \epsilon \quad \text{a.s.}
\end{equation}

Finally, we recover the full matrix $X_N$ by adding the last positive deformation $a_1 \mathbf{q}_1 \mathbf{q}_1^T$.  Applying the rank-one interlacing relation $\lambda_{i+1}(X_N) \le \lambda_i(X_N^+)$ for the positive side and the monotonicity $\lambda_{N-i+1}(X_N) \ge \lambda_{N-i+1}(X_N^+)$ for the negative side, we obtain the necessary upper and lower bounds. 

For the upper bounds in negative side, we employ the spectral norm bound from Lemma \ref{lem:rho_a_a.s.} for the extreme edge and the min-max principle for the internal outliers. By choosing the test subspace $V_i = \mathrm{span}(\mathbf{q}_{-1}, \dots, \mathbf{q}_{-i})$, we have
\begin{equation}
    \lambda_{N-i+1}(X_N)^{2k+1} \le \max_{x \in V_i, \|x\|=1} \langle x, X_N^{2k+1} x \rangle.
\end{equation}
As $N \to \infty$, the right-hand side converges to the $(2k+1)$-th moment of the projection measure associated with the $i$-th negative spike. By taking $k$ sufficiently large and $\epsilon \to 0$, we conclude that
\begin{equation}
    \lim_{N \to \infty} \lambda_{N-i+1}(X_N) = \rho_{a_{-i}} \quad \text{a.s.}
\end{equation}
An analogous argument using the max-min principle establishes the limit for the largest eigenvalues:
\begin{equation}
    \lim_{N \to \infty} \lambda_{i}(X_N) = \rho_{a_{i}} \quad \text{a.s.}
\end{equation}

This completes the inductive step and the proof of Theorem \ref{thm:LLN1}.
\end{proof}

\section{Fluctuations of outliers}\label{sec:fluctuation Gaussian}
We prove Theorem \ref{main result} and Theorem \ref{Thm-a>1:baby} in this section. In Section \ref{sec:typical_diagrams}, we investigate the asymptotics of diagram functions for typical diagrams. In Section \ref{sec:proof-moments}, we prove Theorem \ref{Thm-a>1:baby} by combining the limit of the high-order moments of $X_N$ and showing that it matches the exponential of $Z_{\beta}$. Theorem \ref{main result} is then deduced from Theorem \ref{Thm-a>1:baby} in Section \ref{appendix: main results}.

Throughout this section, we assume without loss of generality that the row parameters satisfy $m_i = i$ for all $i$ in Assumption (A3) of Definition~\ref{def:inhomo}; the general case follows by a suitable permutation.

\subsection{Typical diagram functions}\label{sec:typical_diagrams}
This subsection analyzes the asymptotics of the diagram function $\mathcal{W}_{\Gamma}(k_1,\cdots,k_s)$ at the fluctuation scale, where  $k_i \sim t_i/\sigma_N^*$. By Theorem \ref{Prop:Upper bound} (2), if $\Gamma$ is non-typical, then $\mathcal{W}_{\Gamma}(k_1,\cdots,k_s)$ goes to $0$. It therefore suffices to consider  the case in which $\Gamma$ is typical. 
 The limiting behavior of such diagram functions is characterized by the following definition:
\begin{definition}
Given  any typical diagram $\Gamma$ and $\sigma_a:=(a^2-1)/(a^2+1)$ with $a>1$, for $t_1,\ldots,t_s>0$,  we define the   limit  diagram function  as 
\begin{equation}\label{equ:W(Gamma)_lim}
     w_{\Gamma}(t_1,\dots,t_s):=\sum_{\eta:V_b\to[r]}\prod_{j=1}^{s}\frac{ (t_j\sigma_aa^{-1})^{|E_{b,j}|-1}}{(|E_{b,j}|-1)!}\prod_{(x,y) \in E_{\rm int}}S_{\eta(x)\eta(y)} \prod_{(z,w) \in E_b}(Q_{\beta}^{}Q_{\beta}^*)_{\eta(z)\eta(w)},
\end{equation} where  $S_{ij}$ is given by
\begin{equation}\label{equ:G}
S_{ij}=\sqrt{g_{ij}+\tilde{\sigma}_{ij}^2+\frac{\beta}{2}(\tau_i-\chi_i)\delta_{ij}}\ ,
\end{equation}and $g_{ij}$, $\tilde{\sigma}_{ij}$, $\chi_i$, $\tau_i$ are defined  respectively  in \eqref{assumption1}, \eqref{assumption2}, \eqref{def:sigma_i}, and \eqref{def:tau_i}.

\end{definition}

\begin{theorem}\label{theorem:typical_Gaussian}
    With  the Gaussian IRM matrix $X_N$   in Definition \ref{def:inhomo}, assuming that 
     $k_i=\lfloor t_i/\sigma_N^* \rfloor$ with $t_i>0$,  then for any typical diagram $\Gamma$ we have
\begin{equation}
      \lim_{N \to \infty} \frac{1}{2^s}\sum_{\chi_1,\ldots,\chi_s \in \{0,1\}}\mathcal{W}_{\Gamma,X}(k_1+\chi_1,\ldots,k_s+\chi_s)=w_{\Gamma}(t_1,\dots,t_s).
    \end{equation}
    \end{theorem}
 Recalling the diagram function defined in \eqref{equation:W_Gamma}, for any typical diagram $\Gamma$, since $E_{b,j} \neq \emptyset$ for each $j\in [s]$, the first product term of the Catalan numbers disappears. Moreover, since $V_{\rm int}=\emptyset$,  the map $\eta$ takes values in $[r]$. Thus, the diagram function is  simplified to   
\begin{multline}\label{equ:diagram_typical}
\mathcal{W}_{\Gamma,X}(k_1,\dots,k_s)=\rho_a^{-(k_1+\cdots+k_s)}\sum_{\eta:V(\Gamma) \to [r]}\sum_{n_{e}\ge 1, {e \in E(\Gamma)}}\prod_{j=1}^{s}\binom{k_j}{\frac{k_j-\sum_{\partial D_j}n_e}{2}}\\
\times\prod_{(x,y) \in E_{\rm int}}p_{n_e}(\eta(x),\eta(y))
\prod_{(z,w) \in E_b}\big(A^{n_e}\big)_{\eta(z)\eta(w)}.
\end{multline}
Therefore, to prove Theorem~\ref{theorem:typical_Gaussian}, we address two challenges through separate propositions:
\begin{itemize}
    \item Eigenvalues of $A$ with modulus less than $a$ vanish in the limit;
    \item The tree weight on the $j$-th face contributes a binomial   factor $  \binom{k_j}{\frac{1}{2}(k_j - \sum_{e \in \partial D_j} n_e)}$.
\end{itemize}
In each step, we compare the term with its limit using an auxiliary diagram function and show that the error vanishes as $N \to \infty$.

 \paragraph{Step 1: Remove smaller eigenvalues.} 
In this step, we handle the eigenvalues of $A$ with modulus less than $a$. We recall the spectral decomposition of $\widetilde{A}$ from Theorem~\ref{main result} and introduce a matrix discarding all less $a$ eigenvalues,
\begin{equation}\label{equ:A_eigenvalue=a}
    \hat{A} := U^{*} \operatorname{diag}(a I_{q}, -a I_{q_-}, O_{r-q-q_-}) U.
\end{equation}
Here $q_-$ is the multiplicities of $-a$ in $\widetilde{A}$.  Denote $\mathcal{W}_{\Gamma,\hat{X}}(k_1,\ldots,k_s)$ by the diagram function associated with   $\hat{X}: = H + \hat{A}$.

Our first step in proving Theorem~\ref{theorem:typical_Gaussian} is to establish the following proposition, which allows us to eliminate the contribution of the smaller perturbations.
\begin{proposition}\label{prop:lim-step2}
   With  the Gaussian IRM matrix $X_N$   in Definition \ref{def:inhomo}, assuming that    $k_i\sim t_i/\sigma_N^* $ with $t_i>0$,  then for any typical diagram $\Gamma$ we have
    \begin{equation}
    \lim\limits_{N\rightarrow\infty}|\mathcal{W}_{\Gamma,X}(k_1,\ldots,k_s)-\mathcal{W}_{\Gamma,\hat{X}}(k_1,\ldots,k_s)| =0.
    \end{equation}
\end{proposition}

\begin{proof}
By Lemma~\ref{lemma:Upper1} in the appendix, for any typical diagram $\Gamma$,   $\eta: V(\Gamma) \to [r]$  and   $(n_e)_{e \in E(\Gamma)}$, we have
\begin{equation}
    \Bigg| \prod_{(z,w) \in E_b} \big(A^{n_e}\big)_{\eta(z)\eta(w)} - \prod_{(z,w) \in E_b} \big(\hat{A}^{n_e}\big)_{\eta(z)\eta(w)} \Bigg|
    \le a^{\sum_{e \in E_b} n_e} \bigg( \prod_{e \in E_b} (1+\delta^{n_e}) - 1 \bigg),
\end{equation}
where $\delta \in (0,1)$ is chosen so that $\{|a_{q_++q_-+1}|,\dots,|a_r|\} < a\delta$. Recalling \eqref{equ:diagram_typical},
we combine the above inequality with the simple estimate for the $n_e$-step transition probability
\begin{equation}\label{equ:upper_p}
    p_{n_e}(\eta(x),\eta(y)) \le (\sigma_N^*)^2,
\end{equation}
to obtain the following bound for the difference of two diagram functions 
\begin{equation}
\begin{aligned}   & |\mathcal{W}_{\Gamma,X}(k_1,\ldots,k_s)-\mathcal{W}_{\Gamma,\hat{X}}(k_1,\ldots,k_s)| \\
\le & (\sigma_N^*)^{2|E_{\rm int}|} \rho_a^{-(k_1+\cdots+k_s)}\sum_{\eta:V(\Gamma) \to [r]}\sum_{n_e,{e \in E(\Gamma)}}\prod_{j=1}^{s}\binom{k_j}{\frac{k_j-\sum_{e \in \partial D_j}n_e}{2}} \Big| \prod_{(z,w) \in E_b} \big(A^{n_e}\big)_{\eta(z)\eta(w)} - \prod_{(z,w) \in E_b} \big(\hat{A}^{n_e}\big)_{\eta(z)\eta(w)} \Big|\\
\le & (\sigma_N^*)^{2|E_{\rm int}|} \rho_a^{-(k_1+\cdots+k_s)}\sum_{\eta:V(\Gamma) \to [r]}\sum_{n_e,{e \in E(\Gamma)}}\prod_{j=1}^{s}\binom{k_j}{\frac{k_j-\sum_{e \in \partial D_j}n_e}{2}}a^{\sum_{e \in E_b} n_e}\Big(\prod_{e \in E_b}(1+\delta^{n_e})-1\Big). 
\end{aligned}
\end{equation}

We now expand the last term as  
\begin{equation}
    \prod_{e \in E_b}(1+\delta^{n_e})-1 =\sum_{I \subset E_b,I\neq \emptyset}\delta^{\sum_{e \in I}n_e}. 
\end{equation}
Since $\delta<1$, we have $\delta^{\sum_{e \in I}n_e}<\delta^{n_e}$ for any $e \in I$. It suffices to show  that for any $e_0 \in E_{b,j_0}(\Gamma)$, as $N \to \infty$,
\begin{equation}\label{equ:C9}
\mathcal{W}_{\Gamma,e_0}^{\rm err}:= (\sigma_N^*)^{2|E_{\rm int}|}\rho_a^{-(k_1+\cdots+k_s)}\sum_{\eta:V(\Gamma) \to [r]}\sum_{n_e,{e \in E(\Gamma)}}a^{\sum_{e \in E_b} n_e}\delta^{n_{e_0}}\prod_{j=1}^{s}\binom{k_j}{\frac{k_j-\sum_{\partial D_j}n_e}{2}} =o(1).
\end{equation}
We now prove \eqref{equ:C9} follows a strategy similar  to that of Proposition~\ref{prop:W+upper_bound}.

  We fix $(n_e)_{e \in E_{\rm int}(\Gamma)}$, set  $n_j=\sum_{e \in \partial D_j}n_e$, $n_j^{\rm int}=\sum_{e \in \partial D_j\cap E_{\rm int}} n_e$, and  rewrite the power  factor of $a$ as      ${\sum_{e \in E_b}n_e}={\sum_j n_j-2\sum_{e \in E_{\rm int}}n_e}$. The summation over $(n_e)_{e \in E_b(\Gamma)}$ is then constrained by the linear equations  
\begin{equation}\label{equ:linear}
    \mathfrak{C}_j: \sum_{e \in \partial D_j \cap E_b(\Gamma)}n_e=n_j-n_j^{\rm int}.
\end{equation} 
So we   can  rewrite \eqref{equ:C9} as
\begin{multline} \label{equ:C9-2}
\mathcal{W}_{\Gamma,e_0}^{\rm err}=(\sigma_N^*)^{2|E_{\rm int}|}\sum_{\eta:V(\Gamma) \to [r]}\sum_{n_e,e \in E_{\rm int}} a^{-2\sum_{e \in E_{\rm int}} n_e}\Big(\prod_{j\neq j_0}\sum_{n_j \ge n_j^{\rm int}} \sum_{n_e,e \in E_{b,j},\mathfrak{C}_j}\binom{k_j}{\frac{k_j-n_j}{2}}a^{n_j}\rho_a^{-k_j}\Big)\\
\times \Big(\sum_{n_{j_0} \ge n_{j_0}^{\rm int}} \sum_{n_e,e \in E_{b,{j_0}},\mathfrak{C}_{j_0}}\binom{k_{j_0}}{\frac{k_{j_0}-n_{j_0}}{2}}a^{n_{j_0}}\delta^{n_{e_0}}\rho_a^{-k_{j_0}}\Big).
\end{multline}
For any ${j} \neq j_0$, we see from \eqref{equ:num_solution}, \eqref{equ:binom} and \eqref{eq:upper,inner part} that
 \begin{equation}  \label{UP-1}
\sum_{n_j \ge n_j^{\rm int}} \sum_{n_e,e \in E_{b,j},\mathfrak{C}_j}\binom{k_j}{\frac{k_j-n_j}{2}}a^{n_j}\rho_a^{-k_j}
\le \frac{a^2}{a^2-1}\frac{k_j^{|E_{b,j}|-1}}{(|E_{b,j}|-1)!}.
\end{equation}
While for $j = j_0$, we fix $n_{e_0}$ and observe that the system $\mathfrak{C}_{j_0}$ is equivalent to the following equation in the variables
$(n_e)_{e \in E_{b,j_0}\backslash e_0}$,
\begin{equation}
    \mathfrak{C}'_{j_0}: \sum_{e \in E_{b,j_0}\backslash e_0}n_e=n_{j_0}-n_{j_0}^{\rm int}-n_{e_0}.
\end{equation}
This  system  has at most 
 \begin{equation}
 \binom{n_{j_0}-n_{j_0}^{\rm int}-n_{e_0}-1}{|E_{b,j_0}|-2}\le \frac{n_{j_0}^{|E_{b,{j_0}}|-2}}{(|E_{b,{j_0}}|-2)!} 
 \end{equation} solutions.
Therefore,
 \begin{equation}   \label{UP-2}  \begin{aligned}
 &\sum_{n_{j_0} \ge n_{j_0}^{\rm int}} \sum_{n_e,e \in E_{b,{j_0}},\mathfrak{C}_{j_0}}\binom{k_{j_0}}{\frac{k_{j_0}-n_{j_0}}{2}}a^{n_{j_0}}\delta^{n_{e_0}}\rho_a^{-k_{j_0}}  
\le   \sum_{n_{j_0} \ge 0}\sum_{n_{e_0} \ge 0}\frac{n_{j_0}^{|E_{b,j}|-2}}{(|E_{b,j}|-2)!}a^{n_{j_0}}\binom{k_{j_0}}{\frac{k_{j_0}-n_{j_0}}{2}}\rho_a^{-k_{j_0}}\delta^{n_{e_0}}\\
\le & \frac{1}{1-\delta}\sum_{n_{j_0}}\frac{n_{j_0}^{|E_{b,j}|-2}}{(|E_{b,j}|-2)!}a^{n_{j_0}}\binom{k_{j_0}}{\frac{k_{j_0}-n_{j_0}}{2}}\rho_a^{-k_{j_0}} \le \frac{1}{1-\delta}\frac{a^2}{a^2-1}\frac{k_{j_0}^{|E_{b,{j_0}}|-2}}{(|E_{b,{j_0}}|-2)!},
     \end{aligned}
 \end{equation}
where the last inequality follows again from   \eqref{eq:upper,inner part}.

Based on   \eqref{equ:C9-2}, together with   \eqref{UP-1} and \eqref{UP-2},  we   obtain 
\begin{equation}
    \begin{aligned}
        \mathcal{W}_{\Gamma,e_0}^{\rm err} 
   \le & C_{a,\delta}\frac{k_{j_0}^{|E_{b,{j_0}}|-2}}{(|E_{b,{j_0}}|-2)!}\prod_{j \neq j_0}\frac{k_j^{|E_{b,j}|-1}}{(|E_{b,j}|-1)!}(\sigma_N^*)^{2|E_{\rm int}|}\sum_{n_e:e \in E_{\rm int}(\Gamma)}a^{-2\sum_{e \in E_{\rm int}} n_e}\\
   \le & C_{a,\delta,\Gamma,\{t_i\}}(\sigma_N^*)^{2|E_{\rm int}|-|E_{b}|+s+1}\prod_{e \in E_{\rm int}}\Big(\sum_{n_e \ge 1}a^{-2 n_e}\Big)\\
   = &  C_{a,\delta,\Gamma,\{t_i\}} \sigma_N^*(a^2-1)^{-|E_{\rm int}(\Gamma)|} \to 0, 
    \end{aligned}
\end{equation}
where  the positive constant    may  depend  only on $a,\delta,\Gamma$ and $t_1,\cdots,t_s$, and in the last line we use \eqref{eq:upper,inner part}
for the summation over $(n_e)_{e \in E_{\rm int}}$ and \eqref{equ:ineq} for the exponent of $\sigma_N^*$. 

This thus completes the proof.
  \end{proof}

\paragraph{Step 2: Counting trees.} In this second step, we study the limiting behavior of the term involving the binomial coefficient in the definition of $\mathcal{W}_{\Gamma,\hat{X}}(k_1,\ldots,k_s)$, which arises from Catalan trees. More precisely, we prove the following proposition:

\begin{proposition}\label{prop:lim_step3}
With the same notation and assumptions as in Proposition \ref{prop:lim-step2}, for any typical diagram $\Gamma$ we have
\begin{equation}
   \lim\limits_{N\rightarrow\infty} | \mathcal{W}_{\Gamma,\hat{X}}(k_1,\ldots,k_s)-\mathcal{W}_{\Gamma}^{'}(k_1,\dots,k_s)|= 0,
\end{equation}
where  the auxiliary diagram function  
\begin{multline}\label{def:diagram_function_2}
    \mathcal{W}^{'}_{\Gamma}(k_1,\dots,k_s):=\sum_{\eta:V(\Gamma) \to [r]}\sum_{(\chi_e)_{e \in E(\Gamma)} \in \{0,1\}} \prod_{j=1}^{s}\Big(\frac{(k_j\sigma_a)^{|E_{b,j}|-1}}{2^{|E_{b,j}|-1}(|E_{b,j}|-1)!}\Big)1_{\{k_j+\sum_{e \in \partial D_j}\chi_e \equiv 1\mod 2
    \}}\\
  \times \prod_{(x,y) \in E_{\rm int}} \Big(\sum_{n_e: n_e \equiv \chi_e \mod 2}{p_{n_e}(\eta(x),\eta(y))}{a^{-2n_e}}\Big)\prod_{(z,w) \in E_b}\big(Q_{\beta}^*Q_{\beta}^{}+(-1)^{\chi_e}Q_{\beta_-}^*Q_{\beta_-}^{}\big)_{\eta(z)\eta(w)},
\end{multline}
with  $Q_{\beta_-}$ being  a $q_-\times r$ matrix formed by orthogonal eigenvectors of eigenvalue $-a$.
\end{proposition}

Recalling the definition of $\hat{A}$ in \eqref{equ:A_eigenvalue=a}, a direct computation shows that for any boundary edge $e$,
\begin{equation}\label{equ:power_hat_A}
    \big(\hat{A}^{n_e}\big)_{\eta(z)\eta(w)}=a^{n_e}(Q_{\beta}^*Q_{\beta}^{}+(-1)^{n_e}Q_{\beta_-}^*Q_{\beta_-}^{})_{\eta(z)\eta(w)}.
\end{equation} Consequently, to handle the summation over boundary edges (i.e., over $(n_e)_{e \in E_b(\Gamma)}$) in the definition of $\mathcal{W}_{\Gamma,\hat{X}}(k_1,\ldots,k_s)$, one must treat the cases of odd and even $n_e$ separately. To this end, we fix the parity $\chi_e \in \{0,1\}$ of each $n_e$ and sum over all $n_e$ satisfying $n_e \equiv \chi_e \pmod 2$. 
With the factor involving $Q_{\beta}$ in \eqref{equ:power_hat_A} fixed, we proceed analogously to the proofs of Proposition~\ref{prop:W+upper_bound} and Proposition~\ref{prop:lim-step2}.

\begin{proof}[Proof of Proposition \ref{prop:lim_step3}]

Recalling the definition of $\mathcal{W}_{\Gamma,\hat{X}}(k_1,\ldots,k_s)$ in \eqref{equ:diagram_typical}, we fix $(n_e)_{e \in E_{\rm int}}$ and $(\chi_e)_{e \in E(\Gamma)}$, set $n_j=\sum_{e \in \partial D_j}n_e$, $n_j^{\rm int}=\sum_{e \in \partial D_j\cap E_{\rm int}} n_e$, and then take the summation over $(n_e)_{e \in E_b(\Gamma)}$ 
subject to the linear system $\mathfrak{C}_j$ \eqref{equ:linear} and the congruence conditions
\begin{equation}\label{equ:mod}
     \mathfrak{C}^{\rm mod}_j:  n_e \equiv \chi_e \mod 2, \quad \forall e \in E_{b,j}.
\end{equation} 
More explicitly, we rewrite $\mathcal{W}_{\Gamma,\hat{X}}(k_1,\ldots,k_s)$ as
\begin{equation}\label{equ:W_b,j—1}
   \sum_{\eta}\sum_{\chi_e}\sum_{\substack{(n_e)_{e \in E_{\rm int}}\\ n_e \equiv \chi_e\mod 2}} \prod_{j=1}^{s}\mathcal{W}_{j}\prod_{(x,y) \in E_{\rm int}}\big({p_{n_e}(\eta(x),\eta(y))}{a^{-2n_e}}\big)\prod_{(z,w) \in E_b}\big(Q_{\beta_+}^*Q_{\beta_+}+(-1)^{\chi_e}Q_{\beta_-}^*Q_{\beta_-}\big)_{\eta(z)\eta(w)},
\end{equation}
where   $\mathcal{W}_{j}$ as the summation over $E_{b,j}$ is defined by 
\begin{equation}
\mathcal{W}_{j}:=\sum_{n_j \ge n_j^{\rm int}} \sum_{n_{e}, e \in E_{b,j}:\mathfrak{C}_j,\mathfrak{C}^{\rm mod}_j}\binom{k_j}{\frac{k_j-n_j}{2}}a^{n_j}\rho_a^{-k_j}.
\end{equation} 

To compute $\mathcal{W}_{j}$, we recall the convention that $\binom{n}{k}=0$ whenever $n \notin \mathbb{Z}$ or $k \notin \mathbb{Z}$. A standard combinatorial fact shows that the number of solutions $(n_e)_{e \in E_{b,j}(\Gamma)}$ to the linear system $\mathfrak{C}_j$ together with the congruence conditions $\mathfrak{C}^{\mathrm{mod}}_j$ in \eqref{equ:mod} is
\begin{equation}
        \binom{\frac{1}{2}(n_j-n_j^{\rm int}+\sum_{e \in E_{b,j}}\chi_e)-1}{|E_{b,j}|-1}.
\end{equation}
In particular, $W_{b,j}$ is not zero if and only if $k_j+\sum_{e \in E_{j}}\chi_e$ is even. In this case,    
\begin{equation}
   \mathcal{W}_{j}=\sum_{n_j \ge 1}   \binom{\frac{1}{2}\big(n_j-n_j^{\rm int}+\sum_{e \in E_{b,j}}\chi_e\big)-1}{|E_{b,j}|-1}\binom{k_j}{\frac{k_j-n_j}{2}}a^{n_j}\rho_a^{-k_j}.
\end{equation} For the summation over $n_j$, we apply Lemma~\ref{lemma:LLN-trees} to obtain for some constant $C>0$, 
\begin{equation} \label{diffup}
    \Bigg|\mathcal{W}_{j}-\frac{\sigma_a^{|E_{b,j}|-1}k_j^{|E_{b,j}|-1}}{2^{|E_{b,j}|-1}(|E_{b,j}|-1)!}1_{\{k_j+\sum_{e \in E_{j}}\chi_e \equiv 0\mod 2\}}\Bigg| \le  C (n_j^{\rm int}+1)k_j^{|E_{b,j}|-\frac{3}{2}}. 
\end{equation} 

We now turn to \eqref{equ:W_b,j—1} and derive the desired estimate for $\prod_{j=1}^{s}\mathcal{W}_{j}$ from the above bounds on  
each $\mathcal{W}_{j}$. Similar to   \eqref{eq:upper,inner part},  we can obtain the following upper bound:
\begin{equation} \label{diffup-2}
    \mathcal{W}_{j} \le  \rho_a^{-k_j}\sum_{n_j\ge 0}a^{n_j}\frac{n_j^{|E_{b,j}|-1}}{(|E_{b,j}|-1)!} \le \frac{k_j^{|E_{b,j}|-1}}{(|E_{b,j}|-1)!}.
\end{equation}
We apply Lemma \ref{lemma:ineq} with\begin{equation}
    a_j=\mathcal{W}_{j},\quad 
    b_j=\frac{\sigma_a^{|E_{b,j}|-1}}{(|E_{b,j}|-1)!}k_j^{|E_{b,j}|-1},
\end{equation}
and observe from \eqref{diffup} and \eqref{diffup-2} that 
\begin{equation}
   \sum_{j=1}^{s}|a_j-b_j|\prod_{\ell \neq j}\max\{a_{\ell},b_{\ell}\} \le C\sum_{j=1}^{s}(n_j^{\rm int}+1)k_j^{|E_{b,j}|-\frac{3}{2}}\prod_{\ell \neq j}k_{\ell}^{|E_{b,\ell}|-1}\leq\frac{C_{\Gamma,a}}{\sqrt{k_{min}}}\prod_{j=1}^{s}k_j^{|E_{b,j}|-1}\Big(\sum_{e \in E_{\rm int}}n_e\Big),
\end{equation}
where we use $n_j^{\rm int}+1 \le 2n_j^{\rm int}$ in the last inequality, and $k_{min}:=\min \{k_1,\ldots,k_s\}$. This implies  
\begin{equation}\label{equ:prod_W_b,j}
\Bigg|\prod_{j=1}^{s}\mathcal{W}_{j}-\prod_{j=1}^{s}\frac{\sigma_a^{|E_{b,j}|-1}k_j^{|E_{b,j}|-1}}{2^{|E_{b,j}|-1}(|E_{b,j}|-1)!}1_{\{k_j+\sum_{e \in E_{j}}\chi_e \equiv 0\mod 2\}}\Bigg|\le \frac{C_{\Gamma,a}}{\sqrt{k_{min}}}\prod_{j=1}^{s}k_j^{|E_{b,j}|-1}\Big(\sum_{e \in E_{\rm int}}n_e\Big).
\end{equation}

Returning to the difference of diagram functions, recalling \eqref{equ:W_b,j—1} and  the upper bound \eqref{equ:upper_p} for transition probabilities $p_{n_e}$, we get 
\begin{equation}
    \begin{aligned}
    &  | \mathcal{W}_{\Gamma,\hat{X}}(k_1,\dots,k_s)-\mathcal{W}_{\Gamma}^{'}(k_1,\dots,k_s)|\\
\le &(\sigma_N^*)^{2|E_{\rm int}|}\sum_{\eta:V \to [r]}\sum_{(n_e)_{e \in E_{\rm int}}}\left|\prod_{j=1}^{s}\mathcal{W}_{j}-\prod_{j=1}^{s}\frac{\sigma_a^{|E_{b,j}|-1}k_j^{|E_{b,j}|-1}}{2^{|E_{b,j}|-1}(|E_{b,j}|-1)!}1_{\{k_j+\sum_{e \in E_{j}}\chi_e \equiv 0\mod 2\}}\right|\prod_{(x,y) \in E_{\rm int}}\frac{1}{a^{2n_e}}\\
\le & \frac{C_{\Gamma,a}}{\sqrt{k_{min}}}(\sigma_N^*)^{2|E_{\rm int}|}\prod_{j=1}^{s}\frac{k_j^{|E_{b,j}|-1}}{(|E_{b,j}|-1)!}\sum_{\eta:V \to [r]}\sum_{(n_e)_{e \in E_{\rm int}}}\Big(\sum_{e \in E_{\rm int}}n_e\Big) \prod_{(x,y) \in E_{\rm int}}\frac{1}{a^{2n_e}}\\
\le & \frac{C_{\Gamma,a}}{\sqrt{k_{min}}} (\sigma_N^*)^{2|E_{\rm int}|-|E_{b}|+s} \prod_{j=1}^{s}\Big(\frac{t_j^{|E_{b,j}|-1}}{(|E_{b,j}|-1)!}\Big)\prod_{(x,y) \in E_{\rm int}}\Big(\sum_{n_e \ge 1}\frac{n_e}{a^{2n_e}}\Big)\\
=& \frac{C_{\Gamma,a}}{\sqrt{k_{min}}}\prod_{j=1}^{s}\frac{t_j^{|E_{b,j}|-1}}{(|E_{b,j}|-1)!}\Big(\frac{a^2}{(a^2-1)^2}\Big)^{|E_{\rm int}|} =o(1),  
\end{aligned}
\end{equation} where in the third line we use the inequality $\sum_{e \in E_{\rm int}}n_e \le |E_{\rm int}| \prod_{e \in E_{\rm int}}n_e$ for $n_e \ge 1$ and absorb the factor $|E_{\rm int}|$ into the constant $C_{\Gamma,a}$, and in the last line we use \eqref{equ:ineq} for the exponent of $\sigma_N^*$. 

This thus completes the proof.
\end{proof}

With Proposition \ref{prop:lim-step2} and Proposition \ref{prop:lim_step3} in hand, we are now ready to complete the proof of Theorem \ref{theorem:typical_Gaussian}.

\begin{proof}[Proof of Theorem~\ref{theorem:typical_Gaussian}]
  In view of Propositions~\ref{prop:lim-step2} and~\eqref{def:diagram_function_2}, it remains to analyze the limit of $\mathcal{W}'_{\Gamma}$. Recalling its definition in \eqref{def:diagram_function_2}, and using $k_i,k_i+1 \sim t_i/\sigma_N^*$, we see that the first term on the right-hand side of ~\eqref{def:diagram_function_2} is
  \begin{equation*}
(1+o(1))(\sigma_N^{*})^{-|E_b|+s}\prod_{j=1}^{s}\frac{(t_j\sigma_a)^{|E_{b,j}|-1}}{2^{|E_{b,j}|-1}(|E_{b,j}|-1)!}.
  \end{equation*}

Now we take the summation over $\chi_j$, noting the identity
\begin{equation}
\sum_{\chi_1, \ldots,\chi_s\in\{0,1\}} \prod_{j=1}^{s}  1_{\{k_j+\chi_j+\sum_{e \in \partial D_j}\chi_e \equiv 0 \mod 2  \}}=1,
\end{equation}
we therefore obtain
    \begin{multline}\label{equ:W_sum}
        \frac{1}{2^s}\sum_{\chi_1,\ldots,\chi_s \in \{0,1\}}\mathcal{W}'_{\Gamma}(k_1+\chi_1,\ldots,k_s+\chi_s)=(1+o(1))(\sigma_N^{*})^{-|E_b|+s}\prod_{j=1}^{s}\frac{(t_j\sigma_a)^{|E_{b,j}|-1}}{2^{|E_{b,j}|}(|E_{b,j}|-1)!}\times \\\sum_{\eta}\sum_{\chi_e} \prod_{(x,y) \in E_{\rm int}} \Big(\sum_{n_e: n_e \equiv \chi_e} {p_{n_e}(\eta(x),\eta(y))}{a^{-2n_e}}\Big) \prod_{(z,w) \in E_b}\big(Q_{\beta}^*Q_{\beta}^{}+(-1)^{\chi_e}Q_{\beta_-}^*Q_{\beta_-}^{}\big)_{\eta(z)\eta(w)}.
    \end{multline}

We now address the summation over $\chi_e$ in the above equation. For any interior edge $e \in E_{\rm int}$, recalling $S_{ij}$ defined in \eqref{equ:G}, we use the assumptions \eqref{assumption1}, \eqref{assumption2}, \eqref{def:sigma_i}, and \eqref{def:tau_i} to obtain 
\begin{equation}
\begin{aligned}
     \sum_{\chi_e \in \{0,1\}}\sum_{n_e: n_e \equiv \chi_e} p_{n_e}(\eta(x),\eta(y)){a^{-2n_e}} = \sum_{n_e \ge 1}{p_{n_e}(\eta(x),\eta(y))}{a^{-2n_e}}=\big(1+o(1)\big)(\sigma_N^*)^2a^{-2}S_{\eta(x)\eta(y)}.
    \end{aligned}
\end{equation} 
And for any boundary edge $e=(z,w) \in E_b$, we note that
\begin{equation}
\sum_{\chi_e \in \{0,1\}}\big(Q_{\beta}^*Q_{\beta}^{}+(-1)^{\chi_e}Q_{\beta_-}^*Q_{\beta_-}^{}\big)_{\eta(z)\eta(w)}=2\big(Q_{\beta}^*Q_{\beta}^{}\big)_{\eta(z)\eta(w)}.
\end{equation}
Substituting the above two equations into \eqref{equ:W_sum}, and using \eqref{equ:ineq} for the exponent of $\sigma_N^*$ and $a$, we deduce that  the right-hand side of  \eqref{equ:W_sum} coincide with \eqref{equ:W(Gamma)_lim},  therefore complete the proof.
\end{proof}

\subsection{Proof of Theorem \ref{Thm-a>1:baby}}\label{sec:proof-moments}

To prove Theorem \ref{Thm-a>1:baby}, we begin  with the following expansion of the moments of $X_N$.

\begin{theorem}\label{coro:expansion}With the same notation and assumptions as in Theorem \ref{main result}, 
let  $k_i=\lfloor t_i/\sigma_N^* \rfloor$ with $t_i>0$. Then  
\begin{equation}
       \lim_{N \to \infty}\frac{1}{2^s}\mathbb{E}\Big[\prod_{j=1}^{s}\Big(\tr \Big(\frac{X_{N}}{\rho_a}\Big)^{k_j}+\tr \Big(\frac{X_N}{\rho_a}\Big)^{k_j+1}\Big)\Big]= \sum_{\Gamma}w_{\Gamma}(t_1,\dots,t_s)
    \end{equation}
   where the sum runs over all typical diagrams.
\end{theorem}
\begin{proof}
Recalling the assumption of sparsity $\sigma^{*}_N  \log N \to 0$ in Theorem \ref{main result}, we see from Corollary \ref{upper_bounds_moments} (i) that for any $k_i=\lfloor  t_i/\sigma_N^*\rfloor$ and $\chi_i \in \{0,1\}$, 
\begin{equation}
    \rho_a^{-\sum_j(k_j+\chi_j)} \Big| \mathbb{E}[\prod_{j=1}^{s}\tr X^{k_j+\chi_j}] - \mathbb{E}[\prod_{j=1}^{s}(\tr X^{k_j+\chi_j}-\tr H^{k_j+\chi_j})] \Big| \le CN^s\Big(\frac{2}{\rho_a}\Big)^{\sum_{j}\lfloor t_j/\sigma_N^*\rfloor} \to 0.
\end{equation}
Therefore, by \eqref{equ:expansion_boundary} we have 
    \begin{equation}\label{equ:expansion_o(1)}
      \begin{aligned}
& \frac{1}{2^s}\mathbb{E}\Big[\prod_{j=1}^{s}\Big(\tr \Big(\frac{X_{N}}{\rho_a}\Big)^{k_j}+\tr \Big(\frac{X_N}{\rho_a}\Big)^{k_j+1}\Big)\Big]\\
&= \frac{1}{2^s} \sum_{\chi_1,\ldots,\chi_s \in \{0,1\}}\rho_a^{-\sum_j(k_j+\chi_j)}\mathbb{E}[\prod_{j=1}^{s}(\tr X^{k_j+\chi_j}-\tr H^{k_j+\chi_j})]+o(1)\\
&= \frac{1}{2^s}\sum_{\Gamma \in \mathcal{D}_{\beta,s}^{b}}\sum_{\chi_1,\ldots,\chi_s \in \{0,1\}}\mathcal{W}_{\Gamma,X}(k_1+\chi_1,\dots,k_s+\chi_s)+o(1).  
      \end{aligned}
    \end{equation}

For any $\Gamma \in \mathcal{D}_{\beta,s}^{b}$, recall Theorem~\ref{Prop:Upper bound} (2) for non-typical $\Gamma$ and Theorem~\ref{theorem:typical_Gaussian} for typical $\Gamma$. It then suffices to show that the sum over $\Gamma \in \mathcal{D}_{\beta,s}^{b}$ and the limit $N \to \infty$ can be interchanged. To this end, we employ an argument analogous to the dominated convergence theorem.
    
For all    $k_i\sim t_i/\sigma_N^*$,  Theorem~\ref{thm:dominating_function} ensures that for fixed $g$ and $t$,
    \begin{equation}\label{equ:upper_D}
    \sum_{ \Gamma \in \mathcal{D}_{\beta,s}^{b}:  g(\Gamma)=g,t(\Gamma)=t}|\mathcal{W}_{\Gamma,X_N}(k_1,\cdots,k_s)|\le   r^s e^{4\xi^2}\mathcal{D}_{g,t,s}(\xi;a,1).
\end{equation} where
$\xi$ depends on $t_i$. Consequently,  we see  from Theorem \ref{theorem:typical_Gaussian} that,

\begin{equation}\label{equ:w_Gamma_D}
    \sum_{ \Gamma \text{ typical: }  g(\Gamma)=g,t(\Gamma)=t}|w_{\Gamma}(t_1,\dots,t_s)|\le   r^s e^{4\xi^2}\mathcal{D}_{g,t,s}(\xi;a,1).
\end{equation}

Given $\epsilon>0$, since $\mathcal{D}$ is absolutely summable \eqref{equ:sum_D}, there exists  $K=K_{\epsilon}>0$ depending on $\epsilon$ such that
\begin{equation}
   \Big| \sum_{g \ge 0}\sum_{t \ge 1}\mathcal{D}_{g,t,s}(\xi;a,1) - \sum_{K\ge g \ge 0}\sum_{K\ge t \ge 1}\mathcal{D}_{g,t,s}(\xi;a,1)\Big| \le \epsilon.
\end{equation}
Using  \eqref{equ:w_Gamma_D} again,  we obtain
\begin{equation}\label{equ:w_Gamma_sum_eps}
   \Big| \sum_{ \Gamma \in \mathcal{D}_{\beta,s}^{b}}w_{\Gamma}(t_1,\dots,t_s) - \sum_{0\le g \le K}\sum_{1\le t \le K}\sum_{ \Gamma \text{ typical: }  g(\Gamma)=g,t(\Gamma)=t}w_{\Gamma}(t_1,\dots,t_s)\Big| \le \epsilon.
\end{equation}

We now apply Theorem~\ref{theorem:typical_Gaussian} for typical diagrams and Theorem~\ref{Prop:Upper bound} for non-typical diagrams, summing over all $\Gamma \in \mathcal{D}_{\beta,s}^{b}$ with $0 \le g(\Gamma) \le K$ and $1 \le t(\Gamma) \le K$. We see from \eqref{equ:expansion_o(1)} that there exists $ N_\epsilon > 0$, depending on $\epsilon$, such that for all $N > N_\epsilon$,
\begin{equation}
    \Bigg|  \frac{1}{2^s}\mathbb{E}\Big[\prod_{j=1}^{s}\Big(\tr \Big(\frac{X_{N}}{\rho_a}\Big)^{k_j}+\tr \Big(\frac{X_N}{\rho_a}\Big)^{k_j+1}\Big)\Big]
    - \sum_{0 \le g \le K} \sum_{1 \le t \le K} \sum_{ \Gamma \text{ typical: }  g(\Gamma)=g,t(\Gamma)=t} 
    w_{\Gamma}( t_1,\ldots,t_s) \Bigg| \le 2\epsilon.
\end{equation}
Combining this with \eqref{equ:w_Gamma_sum_eps} yields, for all $N > N_\epsilon$,
\begin{equation}
     \Big|  \frac{1}{2^s}\mathbb{E}\Big[\prod_{j=1}^{s}\Big(\tr \Big(\frac{X_{N}}{\rho_a}\Big)^{k_j}+\tr \Big(\frac{X_N}{\rho_a}\Big)^{k_j+1}\Big)\Big]-\sum_{ \Gamma \text{ typical }}w_{\Gamma}(t_1,\dots,t_s)\Big| \le 3\epsilon.
\end{equation}

This thus completes the proof. 
\end{proof}

Next, we compare the summation over typical diagrams in Theorem~\ref{coro:expansion} with the corresponding moments of $Z_{\beta}$ defined in \eqref{equ:Z}. By an argument similar to that in Proposition~\ref{prop:moments-iid}, we obtain the following lemma.

\begin{lemma} \label{lemma:moments Z}
For $\beta = 1,2$, let $Z_{\beta}$ be the random matrix defined in \eqref{equ:Z}. Then, for positive integers $i_1,\ldots,i_s$, the following diagram expansion holds:
\begin{equation}\label{equ:expansion-of-Z}
\prod_{j=1}^{s}  \frac{1}{i_j!}\Big(\frac{t_j(a^2-1)}{a(a^2+1)}\Big)^{i_j} \mathbb{E}\big[\prod_{j=1}^{s}\tr Z_{\beta}^{i_j}\big]=\sum_{\substack{\Gamma {\rm \ typical} \\|E_{b,j}(\Gamma)|=i_j+1}}   w_{\Gamma}(t_1,\dots,t_s).
\end{equation}
\end{lemma}

\begin{proof}
We focus on the real case ($\beta=1$), as the complex case is analogous. For convenience of notation, we set $G=H_{\mathrm{ID}}+{H_{\mathrm{Gaussian}}}+H_{\mathrm{Diag}}$. Then $ 
    Z_{1}= Q_{1}G Q_{1}^*$. A direct calculation yields
\begin{equation}
    \begin{aligned}
     & \mathbb{E}\big[\prod_{j=1}^{s}\tr Z_{1}^{i_j}\big]=\mathbb{E}\big[\prod_{j=1}^{s}\tr \ Q_{1} Z_{1}^{i_j}Q_{1}^*\big]\\
     = &\sum_{\eta:[2l_s-s]\to [r]}\mathbb{E}\big[\prod_{j=1}^{s}(Q_{1}Q_{1}^*)_{\eta(l_{j-1})\eta(\gamma(l_{j-1}))}G_{\eta(l_{j-1}+1)\eta(\gamma(l_{j-1}+1))}\cdots(Q_{1}Q_{1}^*)_{\eta(l_{j}-1)\eta(\gamma(l_{j}-1))}\big]
    \end{aligned}
\end{equation}
where $l_j=(2i_1+1)+\cdots+(2i_{j}+1)$ for $j \ge 1$,   with the convention   $l_0=1$, and  $\gamma=(12\cdots l_1)(l_1+1\cdots l_2)\cdots(l_{s-1}+1\cdots l_s)$ is the permutation describing the cycle structure of the product. 

We now reinterpret the above expression as an enumeration of ribbon graphs as follows. 
Draw $s$ polygons $D_1,\ldots,D_s$ with perimeters $2i_1+1,\ldots,2i_s+1$, respectively. Label their vertices and edges in cyclic order: $ 
v_1, v_2, \ldots, v_{l_{s+1}},  
\vec{e}_j = \overrightarrow{v_j v_{\gamma(j)}}.$ 
For each polygon $D_j$, we designate the first, third, fifth, etc., edges as  {boundary edges} and the remaining edges as  {interior edges}. Denote the sets of boundary edges and interior edges by $E_b$, $E_{\rm int}$ respectively.
Thus, we obtain
\begin{equation}
\begin{aligned}
    \mathbb{E}\big[\prod_{j=1}^{s}\tr Z_{1}^{i_j}\big]=&\sum_{\eta:V \to [r]}\prod_{(j,\gamma(j)) \in E_b}(Q_{1}Q_{1}^*)_{\eta(j)\eta(\gamma(j))}\mathbb{E}\big[\prod_{(j,\gamma(j)) \in E_{\rm int}}G_{\eta(j)\eta(\gamma(j))}\big].
    \end{aligned}
\end{equation}

For the interior edges, recall that $\mathcal{P}_2(J)$ denotes the set of all   pairings of $J$ (in particular, $\mathcal{P}_2(J) = \emptyset$ if $|J|$ is odd). Recalling \eqref{equ:G}, we note that $\mathbb{E}[G_{ij}G_{ji}]=S_{ij}(1+\delta_{ij})$. Applying the Wick formula again for $\beta = 1$, we obtain
\begin{equation}
\mathbb{E}\big[\prod_{(j,\gamma(j)) \in E_{\rm int}}G_{\eta(j)\eta(\gamma(j))}\big]
= \sum_{\pi \in \mathcal{P}_2(E_{\rm int})}\prod_{(s,t) \in \pi}S_{\eta(s)\eta(t)}\Big(\delta_{\eta(s)\eta(\gamma(t))}\delta_{\eta(t)\eta(\gamma(s))}+\delta_{\eta(s)\eta(t)}\delta_{\eta(\gamma(s))\eta(\gamma(s))}\Big).
\end{equation}

The diagram $\Gamma$ is constructed by identifying the edges of these $s$ polygons: for each pair $(s,t) \in \pi$, the corresponding edges are identified either in the same or opposite direction. After gluing, we obtain a {typical  diagram} satisfying $|E_{b,j}(\Gamma)| = i_j + 1$ for $j = 1,2,\ldots,s$. Moreover, the coefficients involving $a$ match precisely the factor $\prod_{j=1}^{s} \frac{1}{i_j!} \big( \frac{t_j(a^2-1)}{a(a^2+1)} \big)^{i_j}$. 
Translating the above computation into the language of ribbon graph enumeration yields exactly Equation \eqref{equ:expansion-of-Z}. This thereby completes the proof of Lemma~\ref{lemma:moments Z}.\end{proof}

With Theorem \ref{theorem:typical_Gaussian} and Lemma \ref{lemma:moments Z} in hand, we can now complete the proof of Theorem\ref{Thm-a>1:baby}.

\begin{proof}[Proof of Theorem \ref{Thm-a>1:baby}  ]

Combining Theorem \ref{coro:expansion}  and Lemma \ref{lemma:moments Z}, we obtain
\begin{equation}
   \begin{aligned}
          \lim_{N \to \infty}&\frac{1}{2^s}\mathbb{E}\Big[\prod_{j=1}^{s}\Big(\tr \Big(\frac{X_{N}}{\rho_a}\Big)^{k_j}+\tr \Big(\frac{X_N}{\rho_a}\Big)^{k_j+1}\Big)\Big] 
           =\sum_{i_1,\dots,i_s \ge 0}\sum_{\substack{\Gamma \\ {\rm typical \ diagrams} \\E_{b,j}(\Gamma)=i_j+1}}w_{\Gamma}(t_1,\ldots,t_s)\\    &=\sum_{i_1,\dots,i_s \ge 0}\prod_{j=1}^{s}  \frac{1}{i_j!}\Big(\frac{t_j(a^2-1)}{a(a^2+1)}\Big)^{i_j}\mathbb{E}\big[\prod_{j=1}^{s}\tr Z_{\beta}^{i_j}\big]      =\mathbb{E}\Big[\prod_{j=1}^{s}\tr  
       \exp\!\Big\{ \frac{(a^2-1)}{(a^2+1)a} t_jZ_{\beta}\Big\} 
      \Big],
   \end{aligned}
\end{equation}
which  completes the proof of Theorem \ref{Thm-a>1:baby}.
\end{proof}

\subsection{Proof of Theorem \ref{main result}}\label{appendix: main results}

In this section, we deduce Theorem \ref{main result} from Theorem \ref{Thm-a>1:baby} under the assumption of $a= a_1= a_2= \cdots=a_{q}>a_{q+1}\ge   \cdots   \ge a_{r} > -a$, establishing   the weak convergence of the  first few largest eigenvalues  from the convergence of large power moments.  The general case  \eqref{finiteA}  can be treated   similarly.

Recall that for all $ 1 \le  m \le  N$,  the  $m$-point correlation measure $R_{N,\beta,m}$ of $X_N$ is a random measure on $\mathbb{R}^m$ defined by

\begin{equation}
\int_{\mathbb{R}^m}f(\lambda_1,\ldots,\lambda_m) dR_{N,\beta,m}(\lambda_1,\ldots,\lambda_m)=\mathbb{E}\Big[\sum_{j_1 \neq \cdots \neq j_m}f(\lambda_{i_1}(X_N),\ldots,\lambda_{i_m}(X_N))\Big]
\end{equation}for any continuous test function $f:\mathbb{R}^m \to \mathbb{R}$ with compact support. Similarly, for any $1 \le m \le q$, let $R_{Z,\beta,m}$ be the $m$-point correlation measure of eigenvalues associated with $Z_{\beta}$. By convention, we define $R_{Z,\beta,m}$ as the null measure if $m \ge q+1$.

Rescale the eigenvalues $\{ \lambda_j\}$ of  $X_N$
 as
\begin{equation} \label{rescaling}
   \lambda_j = 
        \rho_a\Big(1 +\frac{(a^2-1)\xi_j\sigma_N^*}{(a^2+1)a}\Big),
\end{equation}
and then introduce the Laplace transform of truncated local statistics
\begin{equation}
    S_{N,m}(t_1,\ldots,t_m)=\sum_{\substack{j_1 \neq \cdots \neq j_m \\ |\xi_{j_i}| \le s_N}} \exp \Big(\frac{(a^2-1)}{(a^2+1)a}(t_1\xi_{j_1}+\cdots+t_k\xi_{j_m})\Big),
\end{equation}
where   $s_N$ may be   chosen such that $\log N \ll s_N \ll  {1}/{\sigma_N^*}$.\\

We follow the strategy outlined in \cite{soshnikov1999universality,Férall2007deform} and proceed in three key steps.

\textbf{Step 1: Convergence of the Laplace transform}
\begin{lemma}\label{lemma:error_trace}
For  the deformed IRM $X_N$ defined in Definition \ref{def:inhomo}, assume $\sigma_N^*\log N \to 0$ as $N \to \infty$. Then for  any integer $p \ge 1$ and  $t > 0$,

\begin{equation}
\lim_{N \to \infty}\mathbb{E}\Big[\Big|\tr\Big(\frac{X_N}{\rho_a}\Big)^{2[t/\sigma_N^*]}-S_{N,1}(t)\Big|^p\Big]=0.
\end{equation}

\end{lemma}

\begin{proof}
Split  the trace into four parts for  summation of eigenvalues and we have   
\begin{equation}    \tr\Big(\frac{X_N}{\rho_a}\Big)^{2[t/\sigma_N^*]}:=r_1+r_2+r_3+r_4,
\end{equation} where 
\begin{equation}   r_1=\sum_{j:\lambda_j \le 0}\Big(\frac{\lambda_j}{\rho_a}\Big)^{2[t/\sigma_N^*]}, \ 
    r_2=\sum_{j: \lambda_j > 0, \, \xi_j >s_N}\Big(\frac{\lambda_j}{\rho_a}\Big)^{2[t/\sigma_N^*]},
\end{equation}

\begin{equation}    r_3=\sum_{j: \lambda_j > 0,\, |\xi_j| \le s_N}\Big(\frac{\lambda_j}{\rho_a}\Big)^{2[t/\sigma_N^*]},\ 
    r_4=\sum_{j: \lambda_j > 0,\,\xi_j <-s_N}\Big(\frac{\lambda_j}{\rho_a}\Big)^{2[t/\sigma_N^*]}.
\end{equation}

In view of the rescalings in \eqref{rescaling}, it is straightforward to see that the leading term yields
\begin{equation}
    r_3 = \Big(1 + O(s_N\sigma_N^*)\Big)\sum_{j:|\xi_j| \le s_N} e^{\frac{(a^2-1)}{(a^2+1)a}t\xi_j}.
\end{equation}   
It's also easy to obtain  an upper bound of  $r_4$,
\begin{equation}
    |r_4|=\sum_ { j : \lambda_ j > 0, \xi_ j \le -s_N}\Big(\frac{\lambda_j}{\rho_a}\Big)^{2[t/\sigma_N^*]} \le N(1-s_N\sigma_N^*)^{2[t/\sigma_N^*]} \le Ne^{-C_ts_N} \to 0,
\end{equation}
since   $s_N \gg \log N$.
 
It then suffices to verify that the $L^p$-norms of the remaining three terms $r_1$ and   $r_2$  converge to zero.
For the term $r_2$,    we arrive at 
\begin{equation}
\begin{aligned}
    \mathbb{E}[|r_2|^p]=&\mathbb{E}\Big[\Big(\sum_ { j : \lambda_ j > 0, \xi_ j \ge s_N}\Big(\frac{\lambda_j}{\rho_a}\Big)^{2[t/\sigma_N^*]}\Big)^p\Big] \le \mathbb{E}\Big[\Big(\sum_ { j : \lambda_ j > 0, \xi_ j \ge s_N}\Big(\frac{\lambda_j}{\rho_a}\Big)^{4[t/\sigma_N^*]} \Big(\frac{1}{1+s_N\sigma_N^*}\Big)^{2[t/\sigma_N^*]}\Big)^p\Big]\\
    \le & \mathbb{E}\Big[\Big(\tr\Big(\frac{X_N}{\rho_a}\Big)^{4[t/\sigma_N^*]}\Big)^p\Big] e^{-2tps_N}
    \le C_{t,p} e^{-2tps_N} 
    \end{aligned}
\end{equation} for some constant $C_{t,p}$.  
Here  in the last equality we have used    Corollary  \ref{upper_bounds_moments}, which relies on the condition $\sigma_N^*\log N \to 0$.

For the term $r_1$, recalling the assumption on eigenvalues of  $A$ in \eqref{finiteA}, we choose a sufficiently small constant  $\epsilon_0>0$ such that $a-\epsilon>\max\{|a_r|,1\}$ and rewrite    $A$ as
$ A=A_1+A_2$,
 where $A_1=U^*\text{diag}(\epsilon_0 I_q,O)U$ and $A_2=U^*\text{diag}(a-\epsilon_0,\cdots,a-\epsilon_0,a_{q+1},\cdots,a_{r})U$. We immediately see from  \eqref{finiteA} that   $\|A_2\|_{op}=a-\epsilon$ and $ \lambda_j(X_N) \ge  \widetilde{\lambda}_j$,
where 
  $\widetilde{\lambda}_1 \ge \dots \ge \widetilde{\lambda}_N$ are  eigenvalues of $H_N+A_2$. 
In particular, since $\lambda_j \le 0$ implies $\widetilde{\lambda}_j \le 0$, it follows that  
\begin{equation}
    \mathbb{E}[|r_1|^p]=\mathbb{E}\Big[\Big(\sum_{j:\lambda_j \le 0}\Big(\frac{\lambda_j}{\rho_a}\Big)^{2[t/\sigma_N^*]}\Big)^p\Big] \le \mathbb{E}\Big[\Big(\sum_{j:\tilde{\lambda}_j \le 0}\Big(\frac{\widetilde{\lambda}_j}{\rho_a}\Big)^{2[t/\sigma_N^*]}\Big)^p\Big].
\end{equation}
Now we apply Corollary \ref{upper_bounds_moments} to the deformed IRM $H_N + A_2$, which yields
\begin{equation}
    \mathbb{E}\Big[\Big(\sum_{j:\tilde{\lambda}_j \le 0}\Big(\frac{\tilde{\lambda}_j}{\rho_a}\Big)^{2[t/\sigma_N^*]}\Big)^p\Big] \le \mathbb{E}\Big[\Big(\tr\Big(\frac{H_N+A_2}{\rho_a}\Big)^{2[t/\sigma_N^*]}\Big)^p\Big]  \le C_{t,p}\Big(\frac{\rho_{a-\epsilon_0}}{\rho_a}\Big)^{2p[t/\sigma_N^*]}.
\end{equation}  The right-hand side tends to $0$ as $N \to \infty$ since $\rho_{a-\epsilon_0} < \rho_a$. This thus completes the desired proof.
\end{proof}

We now use Lemma \ref{lemma:moments} and deduce the convergence of $S_{n,m}$ from $S_{n,1}$.  

\begin{lemma}\label{lemma:convergence_S_n,k}
Under the same assumptions of Theorem \ref{Thm-a>1:baby}, for any $1 \le m \le q$ and   $t_1,\cdots,t_m>0$, we have
\begin{equation}
  \lim_{N \to \infty}\mathbb{E}[S_{N,m}(t_1,...,t_m)]=\int_{\mathbb{R}^m}  \exp(t_1\xi_1+\cdots+t_m\xi_m )dR_{Z,\beta,m}(\xi_1,\ldots,\xi_m).
  \end{equation}
\end{lemma}
\begin{proof}
Since $S_{N,m}(t_1,\dots,t_m)$ can be expressed as a polynomial in terms of ${S_{N,1}(\sum_{j \in J}t_j)}_{J \subset [m]}$, we apply Lemma \ref{lemma:moments} with $n = 2^m$, $Y_{J} = S_{N,1}(\sum_{j \in J}t_j)$, and $Z_{J} = \tr\left( {X_N}/{\rho_a}\right)^{2[\sum_{j \in J}t_j/\sigma_N^*]}$. 

Assumption (1) of Lemma \ref{lemma:moments} follows from Lemma \ref{lemma:error_trace}, and Assumption (2) follows from Corollary \ref{upper_bounds_moments}. With these conditions verified, the proof is completed by combining Lemma \ref{lemma:error_trace} and Theorem \ref{Thm-a>1:baby}.\end{proof} 

\textbf{Step 2: Convergence of correlation measures}

\begin{lemma}\label{lemma:correlation}
    Under the same assumptions of Theorem \ref{Thm-a>1:baby}, for any $m \ge 1$,  the rescaled correlation measure of $X_N$
    \begin{equation}
    \widetilde{R}_{N,\beta,m}(\xi_1,\dots,\xi_m) := R_{N,\beta,m}\Big(\rho_a +\rho_a\frac{(a^2-1)\xi_1\sigma_N^*}{(a^2+1)a},\dots, \rho_a +\rho_a\frac{(a^2-1)\xi_m\sigma_N^*}{(a^2+1)a}\Big)
\end{equation}
converges weakly to $ R_{Z,\beta,m}$.\end{lemma}

\begin{proof}
    We  first consider the case $1 \le m \le q$. Noting that 
    \begin{equation}
        \mathbb{P}\big(\lambda_1(X_N) \ge \rho_a(1+s_N\sigma_N^*)\big) \le \Big(\frac{1}{1+s_N\sigma_N^*}\Big)^{2[t/\sigma_N^*]}\mathbb{E}\Big[\tr\Big(\frac{X_N}{\rho_a}\Big)^{2[t/\sigma_N^*]}\Big] \le \widetilde{C} e^{-C_t s_N} \to 0 
    \end{equation} 
    for some constants  $\widetilde{C}, C_t$, we    
 have 
    \begin{equation}
     \int_{\mathbb{R}^{m}}  1_{\{\exists j, \  \xi_j>s_N\}}(\xi_1,\dots,\xi_m) \,d \tilde{R}_{N,\beta,m}(\xi_1,\ldots,\xi_m) \le \mathbb{P}\big(\lambda_1(X_N) \ge \rho_a(1+s_N\sigma_N^*)\big) \to 0.
    \end{equation}
Therefore, it suffices to prove the weak convergence for
$1_{(-\infty,s_N]^{m}}(\xi_1,\dots,\xi_m)\widetilde{R}_{N,\beta,m}(\xi_1,\dots,\xi_m)$. 
By Lemma \ref{lemma:convergence_S_n,k},   its Laplace transform  converges to that    of $R_{Z,\beta,m}$.  

Next we turn to  the case $m \ge q+1$. 
For any  $\ell<0$,   Theorem \ref{thm:LLN1}  on the almost  sure  convergence of $\lambda_{q+1}(X_N)$ implies  that 
\begin{equation}
    \tilde{R}_{N,\beta,m}([\ell,   +\infty)^{m}) \le \mathbb{P}\big(\lambda_{q+1}(X_N) \ge \rho_a(1+\ell s_N\sigma_N^*)\big) \to 0.
\end{equation}  
Thus, $\widetilde{R}_{N,\beta,m}$ converges to the null measure. 
 
 This completes the proof.
\end{proof}

\textbf{Step 3: Completion of the proof of Theorem~\ref{main result}.}
\begin{proof}[Proof of Theorem \ref{main result}]
    For any real $s_1,\ldots,s_q$, we first rewrite the probability $\mathbb{P}(\xi_1 \le s_1,\ldots,\xi_q \le s_q)$ as a finite linear combination of
probabilities
\begin{equation}
\mathbb{P}(\#\{j \ge 1: 
\xi_j\in I_i\}=m_i, \
i=1,2,\ldots ),
\end{equation}
where $I_i$ are intervals of the form $(s_i,s_j]$ or $(s_i, +\infty)$. Here we stress that   the probability goes to 0 if $\sum^q_{i=1}m_i > q$. 

Define   $\eta_{N,i}=\#\{j \ge 1: 
\xi_j\in I_i\}$ and   similarly let $\eta_{Z_{\beta},i}$ be the counting variable for the number of eigenvalues of $Z_{\beta}$ in $I_i$.
  The definition of the correlation measure directly relates the factorial moments of the counting variables  to the measure:
\begin{equation}
\mathbb{E}\big[\prod^k_{i=1}\prod^{m_i-1}_{\ell =0}(\eta_{N,i}-\ell )\big]=\tilde{R}_{N,\beta,m}\big(\prod^k_{i=1}I_i^{m_i}\big)
\end{equation}
where $m = \sum^k_{i=1}m_i$. 
Consequently, Lemma \ref{lemma:correlation} implies that the joint moments of $\eta_{N,i}$ for $i=1,\ldots, k$ converge to those of $\eta_{Z_{\beta},i}$. Furthermore, since $\eta_{Z_{\beta},i} \le q$, the convergence of these joint moments is sufficient to establish the convergence in distribution of $\eta_{N,i}$.

 This   completes the proof of Theorem \ref{main result}.
\end{proof}

\section*{Acknowledgments}
This work was supported by the National Natural Science Foundation of China under Grants \#12371157 and \#12090012. We are very grateful to Zhantao He for his careful reading of the manuscript and for pointing out several typos, and to Zhonggen Su and Jun Yin for valuable discussions. We also thank Ramon van Handel for insightful discussions about his recent work \cite{bandeira2024matrixconcentrationinequalitiesfree}
 and on inhomogeneous random matrices during the 2024 International Congress of Basic Science (Beijing, July 2024).

\appendix
\section{Proofs in Section \ref{Section:GUE/GOE case}}

\subsection{Proof of Lemma \ref{lem:exclusion_inclusion}}\label{sec:proof_inclusion}
\begin{proof}[Proof of Lemma \ref{lem:exclusion_inclusion}]
Expanding the product
\begin{equation}
\prod_{u\ne v}\big(1-1_{\{\eta(u)=\eta(v)\}}\big)
\end{equation}
amounts to summing over subsets of edges $E$ of the complete graph $K_{|V|}$: choosing an edge $\{u,v\}$ corresponds to inserting the factor $-1_{\{\eta(u)=\eta(v)\}}$.  

Each edge set $E$ determines a partition $\pi(E)$ of $V$: the blocks are the connected components of $(V,E)$. Thus, regrouping terms by $\pi(E)$, we obtain
\begin{equation}
\prod_{u\ne v}\big(1-1_{\{\eta(u)=\eta(v)\}}\big)=
\sum_{\pi} C_\pi \prod_{B\in\pi} 1_{\{\eta(u)=\eta(v)\ \forall u,v\in B\}},
\end{equation}
with
\begin{equation}
C_\pi=\prod_{B\in\pi}\; \bigg(\sum_{\substack{E\subset \binom{B}{2}\\ (B,E)\text{ connected}}} (-1)^{|E|}\bigg).
\end{equation}
Thus it suffices to compute, for $m\ge 1$,
\begin{equation}
c_m:=\sum_{\substack{E\subset \binom{[m]}{2}\\ ( [m],E)\text{connected}}}(-1)^{|E|}.
\end{equation}

We claim that 
\begin{equation} \label{cmidensity}
c_m=(-1)^{m-1}(m-1)!.
\end{equation} 
If so, substituting this into the expression for  $C_\pi$ yields 
\begin{equation}
C_\pi=\prod_{B\in\pi} c_{|B|} = \prod_{B\in\pi}(-1)^{|B|-1}(|B|-1)!,
\end{equation}
which is precisely the desired formula.

The identity  \eqref{cmidensity}   can be proved by induction as follows. For $m=1$, clearly $c_1=1$.  
Assume $c_m=(-1)^{m-1}(m-1)!$ and we consider $c_{m+1}$. To form a connected graph on $m+1$ vertices, the $(m+1)$-th vertex must be adjacent to at least one of the first $m$ vertices. If it connects to exactly one vertex, then after removing this edge, the remaining $m$ vertices must form a connected graph; if it connects to more than one (say to a nonempty subset $S$), then the contributions from such configurations cancel out (when accounting for the edges   in the set $S$).  Consequently, we obtain the recurrence $c_{m+1}=-m\,c_m.$ 
Hence, by induction, we have  $c_{m+1}=(-1)^{m}m!$.  \end{proof}

\subsection{Proof of Lemma \ref{lem:V_b_upper_bound}}\label{appendix:lemma_graph}

\begin{proof}[Proof of Lemma \ref{lem:V_b_upper_bound}]
We first remove all boundary edges of $G$. 
By definition, $\widetilde{\mathcal{V}}_b(G)$ is the set of the boundary vertices that remain after deleting all boundary trees and removing non-marked divalent boundary vertices.  

For any $v\in \widetilde{\mathcal{V}}_b(G)$, let $C$ be the connected component containing $v$. Since we have removed all boundary trees, each $C$ must fall into one of the following cases:
\begin{enumerate}[(i)]
    \item $C$ contains a loop or multiple edges;
    \item $C$ is connected to more than one boundary vertex;
    \item $C$ contains at least one vertex in $\mathcal{V}_{\mathrm{tied}}(G)$.
    \item $C$ is a marked point.
\end{enumerate}

Let $L(C)$ denote the set of loops in $C$, then Euler’s formula gives the number of loops in    $C$ 
\begin{equation}
        |L(C)| = |E(C)| - |V(C)| + 1,\quad |V(C)|=|\widetilde{V}_b(C)|+|V_{int}(C)|,
\end{equation}
which implies 
\begin{equation}\label{equ:Euler_2}
        |\widetilde{\mathcal{V}}_b(C)| 
    = |\mathcal{E}_{\mathrm{int}}(C)| - |\mathcal{V}_{\mathrm{int}}(C)| + 1 - |L(C)|.
\end{equation}

\smallskip
\noindent\textbf{Case (i): $|L(C)| \ge 1$.}  
In this case, it is easy to see from \eqref{equ:Euler_2} that
\begin{equation}
|\widetilde{\mathcal{V}}_b(C)|
    \le |\mathcal{E}_{\mathrm{int}}(C)| - |\mathcal{V}_{\mathrm{int}}(C)|.
\end{equation}

\smallskip
\noindent\textbf{Case (ii) and (iii): $L(C)=0$.}  
Here $C$ is a tree, if nonempty. Two subcases are discussed below.
\begin{itemize}
    \item If $|\widetilde{\mathcal{V}}_b(C)| > 1$, then 
    \begin{equation}
        1<|\widetilde{\mathcal{V}}_b(C)| =|\mathcal{E}_{\mathrm{int}}(C)| - |\mathcal{V}_{\mathrm{int}}(C)| + 1
        \le 2\bigl(|\mathcal{E}_{\mathrm{int}}(C)| - |\mathcal{V}_{\mathrm{int}}(C)|\bigr).
    \end{equation}
    \item If $|\widetilde{\mathcal{V}}_b(C)| = 1$, then the unique boundary vertex must be adjacent to a tied vertex. This  implies 
    \begin{equation}
        |\widetilde{\mathcal{V}}_b(C)| \le |\mathcal{V}_{\mathrm{tied}}(C)|.
    \end{equation}
\end{itemize}

\smallskip
\noindent\textbf{Case (iv): $|\mathcal{E}_{\mathrm{int}}(C)|=0$.} 
Then $C$ reduces to a single point $v$, which implies $v$ is of degree 2 in $G$ (only connects to boundary edges). Thus $v$ is a marked point, which can occur at most $s$ times in total.

\smallskip
Summing over all connected components $C$ gives the desired bound.
\end{proof}

\subsection{Proof of Lemma \ref{lem:sum_partition}}\label{appendix A2}
\begin{proof}[Proof of Lemma \ref{lem:sum_partition}]
Let $\pi$ be a partition of $[t]$ and denote by $i_l$ the number of blocks of size $l$ in $\pi$ for each $l \ge 1$. Then $\sum_{l \ge 1} l i_l = t$. The number of partitions with given multiplicities $(i_l)_{l \ge 1}$ is
\begin{equation}
\frac{t!}{\prod_{l \ge 1} (i_l!)\,(l!)^{i_l}}.
\end{equation}
Consequently, the sum on the left-hand side of \eqref{equ:partition} becomes

\begin{equation}
\sum_{\substack{(i_l)_{l\ge1}\\\sum l i_l=t}}
\frac{t!}{\prod_{l\ge1} (i_l!)(l!)^{i_l}}
\prod_{l\ge1} \bigl((l-1)!\,M^{-(l-1)}\bigr)^{i_l}=t!\sum_{\substack{(i_l)_{l\ge1}\\\sum l i_l=t}}
\prod_{l\ge1} \frac{1}{i_l!}\left(\frac{1}{l\,M^{\,l-1}}\right)^{i_l}.
\end{equation}

Consider the exponential generating function in a formal variable $z$:
\begin{equation}
\sum_{t\ge0}\Bigg[\,
\sum_{\substack{(i_l)_{l\ge1}\\\sum l i_l=t}}
\prod_{l\ge1} \frac{1}{i_l!}\left(\frac{z^l}{l\,M^{\,l-1}}\right)^{i_l}
\Bigg]
= \exp\!\Bigg(\sum_{l\ge1}\frac{z^l}{l\,M^{\,l-1}}\Bigg).
\end{equation}
The inner sum is therefore the coefficient $[z^t]$ of the right-hand side. We first compute the series appearing in the exponent
\begin{equation}
\sum_{l\ge1}\frac{z^l}{l\,M^{\,l-1}}
= M\sum_{l\ge1}\frac{(z/M)^l}{l}
= -M\ln\!\big(1-\tfrac{z}{M}\big),
\end{equation}
which gives 
\begin{equation}
\exp\!\Big(\sum_{l\ge1}\frac{z^l}{l\,M^{\,l-1}}\Big)
= \big(1-\tfrac{z}{M}\big)^{-M}.
\end{equation}
The coefficient of $z^t$ in $(1- {z}/{M})^{-M}$ is
\begin{equation}
[z^t]\,(1-\tfrac{z}{M})^{-M}
= \frac{1}{M^t}\binom{M+t-1}{t}.
\end{equation}
Therefore, the original sum  on the left-hand side of \eqref{equ:partition} equals
\begin{equation}
  \frac{t!}{M^t}\binom{M+t-1}{t}.
\end{equation}

Now use the elementary bound $\binom{M+t-1}{t}\le\big(M+t\big)^t/t!$ to get
\begin{equation}
 \frac{t!}{M^t}\binom{M+t-1}{t}\le \frac{t!}{M^t}\cdot\frac{(M+t)^t}{t!}
=\Big(1+\frac{t}{M}\Big)^t\leq e^{\frac{t^2}{M}},
\end{equation}
where in the last inequality  we use  $1+ {t}/{M}  \le\exp(t/M)$. This gives the desired result.
\end{proof}

\section{Proof of Lemma \ref{lem:Gaussian_moment_dominate}}
\label{appendix:sec4}

\begin{lemma}\label{lem:moment_dominate}

   For real  random variables
   $A, Z_1,\ldots,Z_s$ and  complex  random variables $B, Y_1,\ldots,Y_s$, suppose that   $A$ is independent of   $\{Z_1,\ldots,Z_s\}$   and    $B$ is independent of   $\{Y_1,\ldots,Y_s\}$.  If the following three conditions hold:
    \begin{itemize}
      \item[(i)]  
    \begin{equation}\label{equ:A-B}
        \mathbb{E}[Z_i]\ge \big|\mathbb{E}[Y_i]\big|, \quad \forall i=1,2,\dots,s;
    \end{equation}
        \item[(ii)]  
\begin{equation}\label{equ:Z-X}
        \mathbb{E}\bigg[\prod_{i\in I}(Z_i-\mathbb{E}[Z_i])\bigg]\ge \bigg|\mathbb{E}\bigg[\prod_{i\in I}(Y_i-\mathbb{E}[Y_i])\bigg] \bigg|, \quad \forall I\subset [s];
    \end{equation}

     \item[(iii)]   for any non-negative integers $k_i$ and $k'_i$ ($i=1,\ldots, s$), 
    \begin{equation}
        \mathbb{E}\bigg[A^{k+k'}\prod_{i\in I}(A^{k_i+k'_i}-\mathbb{E}[A^{k_i+k'_i}])\bigg]\ge \bigg|\mathbb{E}\bigg[B^{k}\bar{B}^{k'}\prod_{i\in I}(B^{k_i}\bar{B}^{k'_i}-\mathbb{E}[B^{k_i}\bar{B}^{k'_i}])\bigg] \bigg|,
    \end{equation}
     \end{itemize}
  then for any $I\subset [s]$, we have   
    \begin{equation}
        \mathbb{E}\bigg[\prod_{i\in I}(Z_iA^{k_i+k'_i}-\mathbb{E}[Z_i]\mathbb{E}[A^{k_i+k'_i}])\bigg]\ge \bigg|\mathbb{E}\bigg[\prod_{i\in I}(Y_iB^{k_i}\bar{B}^{k'_i}-\mathbb{E}[Y_i]\mathbb{E}[B^{k_i}\bar{B}^{k'_i}])\bigg] \bigg|.
    \end{equation}
\end{lemma}
\begin{proof}
    Rewrite \begin{equation}
        Z_iA^{k_i+k'_i}-\mathbb{E}[Z_i]\mathbb{E}[A^{k_i+k'_i}]=(Z_i-\mathbb{E}[Z_i])A^{k_i+k'_i}+\mathbb{E}[Z_i](A^{k_i+k'_i}-\mathbb{E}[A^{k_i+k'_i}]),
     \end{equation}
    and then expand the inner term, we have
    \begin{equation}
        \begin{aligned}
            &
            \prod_{i\in I}\left(Z_iA^{k_i+k'_i}-\mathbb{E}[Z_i]\mathbb{E}[A^{k_i+k'_i}]\right)
            =\sum_{J\subset I}\prod_{i\in J}(Z_i-\mathbb{E}[Z_i])A^{k_i+k'_i}\prod_{i\in J^c}\mathbb{E}[Z_i](A^{k_i+k'_i}-\mathbb{E}[A^{k_i+k'_i}]).
        \end{aligned}
    \end{equation}
  For each $J \subset I$, by   the independence of $\{Z_i\}$ and $A$,   we take expectation  and get   
    \begin{equation}\label{equ:tec-lemma-Z,A}
        \begin{aligned}
            &\mathbb{E}\bigg[ \prod_{i\in I}\left(Z_iA^{k_i+k'_i}-\mathbb{E}[Z_i]\mathbb{E}[A^{k_i+k'_i}]\right)\bigg]\\
            &=\sum_{J\subset I}\mathbb{E}\bigg[\prod_{i\in J}(Z_i-\mathbb{E}[Z_i])\prod_{i\in J^c}\mathbb{E}[Z_i]\bigg]\mathbb{E}\bigg[\prod_{i\in J}A^{k_i+k'_i}\prod_{i\in J^c}(A^{k_i+k'_i}-\mathbb{E}[A^{k_i+k'_i}])\bigg].
        \end{aligned}
    \end{equation}
Similarly, we obtain
    \begin{equation}\label{equ:tec-lemma-X,B}
     \begin{aligned}
         &\mathbb{E}\bigg[\prod_{i\in I}(Y_iB^{k_i}\bar{B}^{k'_i}-\mathbb{E}[Y_i]\mathbb{E}[B^{k_i}\bar{B}^{k'_i}])\bigg]\\
         =&\sum_{J \subset I}\mathbb{E}\bigg[\prod_{i\in J}(Y_i-\mathbb{E}[Y_i])\prod_{i\in J^c}\mathbb{E}[Y_i]\bigg]\mathbb{E}\bigg[\prod_{i\in J}B^{k_i+k'_i}\prod_{i\in J^c}(B^{k_i+k'_i}-\mathbb{E}[B^{k_i+k'_i}])\bigg].  
     \end{aligned} 
    \end{equation}

By applying the comparison inequality conditions to each term in the summation over $J \subset I$ on the right-hand sides of both \eqref{equ:tec-lemma-Z,A} and \eqref{equ:tec-lemma-X,B}, the desired inequality follows immediately.
\end{proof}

The following lemma provides a domination inequality for a symmetric $\gamma$-sub-Gaussian random variable $Y$, whose moments satisfy
$\mathbb{E}[|Y|^{2k} ]\le \gamma^{k-1}(2k-1)!!$ for every $ k\geq 2$, in terms of a Gaussian variable $Z$.
\begin{lemma}\label{lem:gaussian_moment_dominate}
Let $Y$ be a real or complex symmetric $\gamma$-sub-Gaussian random variable and $Z$ be a real Gaussian random variable  $Z \sim \mathcal{N}(0,16\gamma)$. Then, for any positive integer  $s \ge 1$ and any integers $k_i,k'_i,k,k'\ge 0$,
    \begin{equation}
        \mathbb{E}\bigg[Z^{k+k'}\prod_{i=1}^{s}(Z^{k_i+k'_i}-\mathbb{E}[Z^{k_i+k'_i}])\bigg]\ge \bigg|\mathbb{E}\bigg[Y^{k}\bar{Y}^{k'}\prod_{i=1}^{s}(Y^{k_i}\bar{Y}^{k'_i}-\mathbb{E}[Y^{k_i}\bar{Y}^{k'_i}])\bigg]\bigg| .
    \end{equation}
\end{lemma}
\begin{proof}
The case for $s=1$ follows from a direct computation. Without loss of generality, we assume that $k_i+k'_i\ge 1$ for $i$. We also assume that $n=k+k'+\sum_{i}k_i+k'_i$ is even,  since  both sides vanish by symmetry when  $n$  is odd.
    
A direct calculation provides  an upper bound of the $Y$ term \begin{equation}
        \begin{aligned}
        &\bigg|\mathbb{E}\bigg[Y^{k}\bar{Y}^{k'}\prod_{i=1}^{s}(Y^{k_i}\bar{Y}^{k'_i}-\mathbb{E}[Y^{k_i}\bar{Y}^{k'_i}])\bigg]\bigg| 
        \le\mathbb{E}\bigg[\bigg|Y^{k}\bar{Y}^{k'}\prod_{i=1}^{s}(Y^{k_i}\bar{Y}^{k'_i}-\mathbb{E}[Y^{k_i}\bar{Y}^{k'_i}])\bigg|\bigg]\\
        &\le \sum_{I\subset [s]}\mathbb{E}\bigg[\bigg|Y^{k}\bar{Y}^{k'}\prod_{i\in I}Y^{k_i}\bar{Y}^{k'_i}\prod_{i\in I^c}\mathbb{E}[Y^{k_i}\bar{Y}^{k'_i}]\bigg|\bigg]\le \sum_{I \subset [s]}\mathbb{E}[|Y|^{k+k'+\sum_{I \in I}k_i+k_i'}]\prod_{i\in I^c}\mathbb{E}[|Y|^{k_i+k'_i}]\\
        &\le 2^s \mathbb{E}\Big[|Y|^{k+k'+\sum_{i}k_i+k'_i}\Big]\le \mathbb{E}\Big[\mathcal{N}(0,4\gamma)^n\Big],
        \end{aligned}
    \end{equation}
where in the last line we use the simple inequality \begin{equation}
    \mathbb{E}[|Y|^{k_1+k_2}]-\mathbb{E}[|Y|^{k_1}]\mathbb{E}[|Y|^{k_2}]=\frac{1}{2}\mathbb{E}[(|Y|^{k_1}-|Y'|^{k_1})(|Y|^{k_2}-|Y'|^{k_2})]\ge 0,
\end{equation} with  
  $Y'$  a iid copy of $Y$.

   Next,    we give a lower bound of the $Z$ term   by  the  Wick formula. Without loss of generality, we assume that all $n_i:=k_i+k'_i$ are even. In fact, if some $k_1+k_1'$ is odd, then $\mathbb{E}[Y^{k_1}\bar{Y}^{k_1'}]=0$ and $\mathbb{E}[Z^{k_1+k_1'}]=0$, so   one can merge the term $Y^{k_1}\bar{Y}^{k_1'}$ to $Y^{k}\bar{Y}^{k'}$.
   
Consider $s$ boxes with $n_i$ elements in the $i$-th box and $k$ elements in the $(s+1)$-th box. Denote by  $\Pi$   the set of pairings of all elements, with the restriction that for all $1\le i\le s$, edges in the $i$-th box do not pair exclusively with edges in the $i$-th box. By the Wick formula we have
    \begin{equation}
         \mathbb{E}\bigg[Z^{k+k'}\prod_{i=1}^{s}(Z^{k_i+k'_i}-\mathbb{E}[Z^{k_i+k'_i}])\bigg]=(16\gamma)^{\frac{n}{2}} |\Pi|.
    \end{equation}
    
   We now  provide a lower bound for $|\Pi|$.  For the first $s$ boxes,  we pick up the first element in each box  when   $s$ is even,  while for odd   $s$   we choose the first element in  the first $s-1$ boxes and the first two elements in the  $s$-th box. Then we consider the pairing of these elements, such that each element is paired with an element in the other set. There are
   \begin{equation}
         \pi(s)=\begin{cases}
            2^{-\frac{s}{2}}\frac{s!}{(\frac{s}{2})!},&~~~~s\text{  even,}\\
           2^{-\frac{s+1}{2}} \frac{(s+1)!}{ (\frac{s+1}{2})!}-2^{-\frac{s-1}{2}} \frac{(s-1)!}{(\frac{s-1}{2})!},&~~~~s\text{  odd},
        \end{cases}
    \end{equation}
    pairings and we have $\pi(s)\ge (\lfloor\frac{s}{2}\rfloor)!$ for $s>1$. 
    The number of pairing for other $n-s$ or $n-s-1$ elements is lower bounded by $(\lfloor\frac{n-s}{2}\rfloor)!$. Hence we have 
    \begin{equation}
        |\Pi|\ge \big(\lfloor\frac{s}{2}\rfloor\big)!\big(\lfloor\frac{n-s}{2}\rfloor\big)!\ge 2^{-n}\big(\frac{n-2}{2}\big)!\ge 4^{-n}\big(\frac{n}{2}\big)!\ge 4^{-n}\mathbb{E}[\mathcal{N}(0,1)^n].
    \end{equation}
    
   Therefore,
    \begin{equation}
        \mathbb{E}\bigg[Z^{k+k'}\prod_{i=1}^{s}(Z^{k_i+k'_i}-\mathbb{E}[Z^{k_i+k'_i}])\bigg]\ge (4\gamma)^n\mathbb{E}[\mathcal{N}(0,1)^n]\ge \bigg|\mathbb{E}\bigg[Y^{k}\bar{Y}^{k'}\prod_{i=1}^{s}(Y^{k_i}\bar{Y}^{k'_i}-\mathbb{E}[Y^{k_i}\bar{Y}^{k'_i}])\bigg]\bigg|.
    \end{equation}
    This completes the proof of  Lemma \ref{lem:gaussian_moment_dominate}.\end{proof}
With the above two lemmas we can prove Lemma \ref{lem:Gaussian_moment_dominate}, which can be viewed as a high-dimensional  generalization of Lemma \ref{lem:gaussian_moment_dominate}.

\begin{proof}[Proof of Lemma \ref{lem:Gaussian_moment_dominate}]
    We prove  by induction on the size of $E$. The case for $|E|=1$ is just Lemma \ref{lem:gaussian_moment_dominate}. Now, suppose the result holds for $E$, and consider the case for $E'=E\cup \{e_0\}$. Set  \begin{equation}
        Y_i=\prod_{e\in E}Y_{e}^{k(i,e)}\bar{Y}_{e}^{k'(i,e)}, \quad Z_i=\prod_{e\in E}Z_{e}^{k(i,e)+k'(i,e)}, \quad A=Z_{e_0}, \quad  B=Y_{e_0}.\end{equation}
    By the induction assumption, we know $A,B, Y_i,Z_i$ satisfy the conditions in Lemma \ref{lem:moment_dominate} and  Lemma \ref{lem:gaussian_moment_dominate}. Thus, by Lemma \ref{lem:moment_dominate} and Lemma \ref{lem:gaussian_moment_dominate}, the desired result holds for $E'$. 
\end{proof}

\section{Proofs in Section \ref{sec:fluctuation Gaussian}}

We need the following auxiliary lemma  to remove smaller eigenvalues of $A$.
\begin{lemma}\label{lemma:Upper1}
For    an $r\times r$ Hermitian  matrix $A$ with eigenvalues $a_1,\dots,a_r$,  assume that  $a_1=a_2=\cdots=a_{q_+}=a>1, \ a_{q_++1}=a_{q_++2}=\cdots=a_{q_++q_-}=-a,$ and $\max \{|a_{q_++q_-}+1|,\dots,|a_r|\}<\delta a$ with $\delta<1$. Let two matrices  $U_+$ of size    $q_+\times r$ and $U_-$ of size $q_-\times r$ be     created by the orthogonal eigenvectors respectively associated  with  $a_1,a_2,\ldots a_{q_+}$ and  $a_{q_++1},a_{q_++2},\ldots,a_{q_++q_-}$.   
Then, for any $t \ge 1$ and any integers $n_1,\dots ,n_t \ge 1$, 
    \begin{equation}
    \bigg|\prod_{i=1}^t(A^{n_i})_{x_iy_i}-\prod_{i=1}^t a^{n_i}(U_+^*U_+^{}+(-1)^{n_i}U_-^*U_-^{})_{x_iy_i}\bigg| \le a^{\sum_{i=1}^t n_i}\Big(\prod_{i=1}^t(1+\delta^{n_i})-1\Big).
    \end{equation}
\end{lemma}

\begin{proof}
By the spectral decomposition of ${A}$,
\begin{equation}
   A=a\begin{pmatrix}
        U_+^* & U_-^* & U^*
     \end{pmatrix}
     \begin{pmatrix}
        I_{q_+} & O & O\\
        O & -I_{q_-} & O\\
        O & O & B\\
     \end{pmatrix}
     \begin{pmatrix}
        U_+ \\
        U_- \\
        U
     \end{pmatrix},
\end{equation}
where $B={\rm diag}({a_{q_++q_-+1}}/{a},\dots,{a_r}/{a})$ is the diagonal matrix constructed by eigenvalues of $A$ whose absolute value is strictly smaller than $a$, 
we have
\begin{equation}
\begin{aligned}
&\bigg|\prod_{i=1}^t (A^{n_i})_{x_iy_i}-\prod_{i=1}^t a^{n_i}(U_+^*U_+^{}+(-1)^{n_i}U_-^*U_-^{})_{x_iy_i}\bigg| \\
= &  \bigg|\prod_{i=1}^ta^{n_i}\big(U_+^*U_+^{}+(-1)^{n_i}U_-^*U_-^{}+(U^*B^{n_i}U)\big)_{x_iy_i}-\prod_{i=1}^ta^{n_i}(U_+^*U_+^{}+(-1)^{n_i}U_-^*U_-^{})_{x_iy_i}\bigg|.\\
\end{aligned}
\end{equation}

Expanding the product, we observe that  
  for   any two disjoint $J_1, J_2 \subset [t]$, the corresponding term is bounded by
\begin{equation}
\begin{aligned}
    &\bigg|\prod_{i\in J_1}(U_+^*U_+^{})_{x_iy_i} \prod_{i\in J_2}(-1)^{n_i}(U_-^*U_-^{})_{x_iy_i}\prod_{i\notin J_1 \cup J_2}(U_2^*B^{n_i}U_2^{})_{x_iy_i}\bigg|\\
   \le & \prod_{i\in J_1}\|U_+^*U_+^{}\|_{\rm op}\prod_{i\in J_2}\|U_-^*U_-^{}\|_{\rm op}\prod_{i\notin J_1 \cup J_2}\|U_2^*B^{n_i}U_2^{}\|_{\rm op} \le \prod_{i\in J_1}1 \cdot \prod_{i\in J_2} 1 \cdot \prod_{i\notin J_1 \cup J_2}\delta^{n_i}.
\end{aligned}
\end{equation}
The desired result follows immediately.\end{proof}

\begin{lemma}\label{lemma:LLN-trees}For any positive integer $k_j$ and any $n$ satisfying $1 \le n \le k_j$, we have
\begin{equation}\label{equ:c-1}
\bigg|\sum_{n_j=n}^{k_j}\frac{(n_j-n)^{|E_{b,j}|-1}}{(|E_{b,j}|-1)!}\binom{k_j}{\frac{k_j-n_j}{2}}a^{n_j}\rho_a^{-k_j}- \frac{\sigma_a^{|E_{b,j}|-1}}{(|E_{b,j}|-1)!}k_j^{|E_{b,j}|-1}\bigg| \le C\cdot n\cdot k_j^{|E_{b,j}|-\frac{3}{2}},
\end{equation} for some  constant $C>0$  depending only on $a$ and $|E_{b,j}|$.
\end{lemma}
\begin{proof}
When  $n \ge \sigma_ak_j/2$,  the inequality is trivial, since by \eqref{eq:upper,inner part} we have 
\begin{equation}
\begin{aligned}
&\bigg|\sum_{n_j=n}^{k_j}\frac{(n_j-n)^{|E_{b,j}|-1}}{(|E_{b,j}|-1)!}\binom{k_j}{\frac{k_j-n_j}{2}}a^{n_j}\rho_a^{-k_j}- \frac{\sigma_a^{|E_{b,j}|-1}}{(|E_{b,j}|-1)!}k_j^{|E_{b,j}|-1}\bigg|\\
\le &\sum_{n_j \ge 0}\frac{n_j^{|E_{b,j}|-1}}{(|E_{b,j}|-1)!}\binom{k_j}{\frac{k_j-n_j}{2}}a^{n_j}\rho_a^{-k_j}+\frac{\sigma_a^{|E_{b,j}|-1}}{(|E_{b,j}|-1)!}k_j^{|E_{b,j}|-1}\\
\le & \Big(1+\frac{\sigma_a^{|E_{b,j}|-1}}{(|E_{b,j}|-1)!}\Big)k_j^{|E_{b,j}|-1}
\end{aligned}
\end{equation}
On the other hand, from $n\geq \sigma_ak_j/2$ we deduce that
\begin{equation}
   n k_j^{|E_{b,j}|-\frac{3}{2}} \ge \sigma_a/2k_j^{|E_{b,j}|-\frac{1}{2}}.
\end{equation}
Hence, the desired inequality holds.

We now consider the case $n<\sigma_ak_j/2$. Let $S_{k_j}$ be a binomial random variable defined in \eqref{equ:binom}, then
    \begin{equation}\label{D-2}
    \begin{aligned}
&\bigg|\sum_{n_j \ge n}\frac{(n_j-n)^{|E_{b,j}|-1}}{(|E_{b,j}|-1)!}\binom{k_j}{\frac{k_j-n_j}{2}}a^{n_j}\rho_a^{-k_j}- \frac{\sigma_a^{|E_{b,j}|-1}}{(|E_{b,j}|-1)!}k_j^{|E_{b,j}|-1}\bigg| \\
= &\bigg|\sum_{n_j \ge n}\mathbb{P}\Big(S_{k_j}=\frac{k_j+n_j}{2}\Big)\frac{(n_j-n)^{|E_{b,j}|-1}}{(|E_{b,j}|-1)!}- \frac{\sigma_a^{|E_{b,j}|-1}}{(|E_{b,j}|-1)!}k_j^{|E_{b,j}|-1}\bigg|\\
=&\frac{k_j^{|E_{b,j}|-1}}{(|E_{b,j}|-1)!}\bigg|\mathbb{E}\Big[\Big(2\frac{S_{k_j}}{k_j}-1-\frac{n}{k_j}\Big)^{|E_{b,j}|-1}1_{S_{k_j} \ge \frac{1}{2}(k_j+n)}\Big]- \sigma_a^{|E_{b,j}|-1}\bigg|\\
\le & \frac{k_j^{|E_{b,j}|-1}}{(|E_{b,j}|-1)!}\left(\bigg|\mathbb{E}\Big[\Big(2\frac{S_{k_j}}{k_j}-1-\frac{n}{k_j}\Big)^{|E_{b,j}|-1}\Big]- \sigma_a^{|E_{b,j}|-1}\bigg|+ \mathbb{E}\Big[\Big(-2\frac{S_{k_j}}{k_j}+1+\frac{n}{k_j}\Big)^{|E_{b,j}|-1}1_{S_{k_j} \le \frac{1}{2}(k_j+n)}\Big]\right)
\end{aligned}
\end{equation} 

By using the elementary bound $|1+\frac{n}{k_j}-\frac{2}{k_j}S_{k_j}| \le 2$, we see the second term in the last line of \eqref{D-2} can be bounded by
\begin{equation}
   \mathbb{E}\Big[\Big(-2\frac{S_{k_j}}{k_j}+1+\frac{n}{k_j}\Big)^{|E_{b,j}|-1}1_{S_{k_j} \le  \frac{1}{2}(k_j+n)}\Big] \le 2^{|E_{b,j}|-1}\mathbb{P}\Big(S_{k_j} \le \frac{k_j+n}{2}\Big).
\end{equation}
By Markov's inequality, the deviation probability can be bounded as follows:
\begin{equation}\label{equ:second_term}
  \mathbb{P}\Big(S_{k_j} \le \frac{k_j+n}{2}\Big) \le  \mathbb{P}\Big(\Big|S_{k_j}-\frac{a^2{k_j}}{a^2+1}\Big|\ge  \frac{1}{4}\sigma_a{k_j}\Big) \le \frac{16}{\sigma_a^2{k_j}^2}{\rm Var}(S_{k_j}) \le \frac{C_a}{k_j},
\end{equation}
where   the assumption $n<\sigma_a{k_j}/2$ has been used. 

For the first term in the last line of \eqref{D-2}, we use the   inequality
$ 
    |a^n-b^n| \le n|a-b|\mathrm{max}\{|a|,|b|\}^{n-1}
$  
to get (noting $S_{k_j} \le k_j$, hence $|(2S_{k_j}-k_j-n)/{k_j}| \le 4$)
\begin{equation} \begin{aligned}  
\bigg|\Big(2\frac{S_{k_j}}{k_j}-1-\frac{n}{k_j}\Big)^{|E_{b,j}|-1}-\sigma_a^{|E_{b,j}|-1}\bigg| &\le 4^{|E_{b,j}|-1}\bigg|2\frac{S_{k_j}}{k_j}-1-\sigma_a-\frac{n}{k_j}\bigg| \\
&\le 4^{|E_{b,j}|-1}\bigg|2\frac{S_{k_j}}{k_j}-1-\sigma_a\bigg|+\frac{n4^{|E_{b,j}|-1}}{k_j}.
\end{aligned}
\end{equation}
Hence the first term in the last line of \eqref{D-2} is bounded by
\begin{equation}\label{equ:first_term}
   \begin{aligned} &4^{|E_{b,j}|-1}\mathbb{E}\bigg[\bigg|2\frac{S_{k_j}}{k_j}-1-\sigma_a\bigg|\bigg]+\frac{n4^{|E_{b,j}|-1}}{k_j}\le 4^{|E_{b,j}|-1}\sqrt{\mathbb{E}\bigg[\Big(2\frac{S_{k_j}}{k_j}-1-\sigma_a\Big)^2\bigg]}+\frac{n4^{|E_{b,j}|-1}}{k_j}\\
=&4^{|E_{b,j}|-1}\sqrt{\mathrm{Var}\Big(\frac{2S_{k_j}}{k_j}\Big)}+\frac{n4^{|E_{b,j}|-1}}{k_j} \le 4^{|E_{b,j}|-1}\frac{2+n}{\sqrt{k_j}} \le \frac{Cn}{\sqrt{k_j}}.
 \end{aligned}
\end{equation}

  Substituting \eqref{equ:second_term} and \eqref{equ:first_term} into \eqref{D-2} gives the desired result.
\end{proof}

\begin{lemma}\label{lemma:ineq}
    For any complex numbers $a_1,\cdots,a_l$ and $b_1,\cdots,b_l$, the following inequality holds:
    \begin{equation}
      \Big|\prod_{i=1}^{l}a_i-\prod_{i=1}^{l}b_i\Big| \le \sum_{i=1}^{l}|a_i-b_i|\cdot\prod_{j\neq i}\max \{|a_j|,|b_j|\}.  
    \end{equation}
\end{lemma}
\begin{proof}
  The inequality follows easily by induction; see, e.g., \cite[Lemma 3.4.3]{durrett2019} for a proof.
\end{proof}

The following lemma  has been   used in the proof of Lemma   \ref{lemma:convergence_S_n,k}, and  can be verified  directly via  the triangle inequality and the H\"older inequality.
\begin{lemma}\label{lemma:moments}
    Given any  positive integers $n$ and    $p \ge 1$, assume that two sequences of random variables $\{Y_{k,N}\}_{k=1}^n$ and $\{Z_{k,N}\}_{k=1}^n$ satisfy 
\begin{enumerate}
\item[(1)] For  any $1 \le k \le n$, $\lim_{N \to \infty}\mathbb{E}[|Y_{k,N}-Z_{k,N}|^p]= 0$;
\item[(2)] There exists a positive constant $C_p$ depending only on $p$ such that for any integers $p_1,\dots,p_n$ with $\sum_{i=1}^{n}p_i \le p$,
    \begin{equation}
    \sup_{N \ge 1}\mathbb{E}\big[\prod_{j=1}^{n}|Z_{j,N}^{p_j}|\big] \le C_p.
    \end{equation}
    \end{enumerate}
Then, for any polynomial $P(x_1,\ldots,x_n) \in \mathbb{C}[x_1,\dots,x_n]$,  
    \begin{equation}
     \lim_{N \to \infty}\mathbb{E}[|P(Y_{1,N},\dots,Y_{n,N}) -P(Z_{1,N},\dots,Z_{n,N})|]=0.
    \end{equation}
\end{lemma}

\let\oldthebibliography\thebibliography
\let\endoldthebibliography\endthebibliography
\renewenvironment{thebibliography}[1]{
  \begin{oldthebibliography}{#1}
    \setlength{\itemsep}{0.3em}
    \setlength{\parskip}{0em}
}
{
  \end{oldthebibliography}
}

\bibliographystyle{alpha}

\begin{spacing}{0}
\small
\bibliography{Reference}
\end{spacing}

\end{document}